\documentclass[11pt]{article}
\usepackage[english]{babel}
\usepackage{times}
\usepackage[T1]{fontenc}
\usepackage{amsmath, amsthm, amsfonts}
\usepackage{amssymb}
\usepackage{amscd}
\usepackage{xspace}
\usepackage{verbatim}
\usepackage{latexsym}
\usepackage{amsfonts}
\usepackage{graphicx}
\usepackage{float}
\date{}
\begin{document}
\title{
\vspace{0.5in}{\bf\Large Boundary singularities of solutions of semilinear \\elliptic equations with critical Hardy potentials}}

\author{{\bf\large \bf  Konstantinos T. Gkikas}\\
 \small Centro de Modelamiento Matem\`atico\\
 \small  Universidad de Chile, Santiago de Chile, Chile\\
{\it\small email: kugkikas@gmail.com}\\[0.5cm]
 {\bf\large Laurent V\'eron}\\
{\small Laboratoire de Math\'ematiques et Physique Th\'eorique }\\
{\small  Universit\'e Fran\c{c}ois Rabelais,  Tours,  FRANCE}\\
{\it\small email: veronl@univ-tours.fr}}

\maketitle
\newcommand{\ana}{\nabla}
\newcommand{\R}{I \!  \! R}
\newcommand{\N}{I \!  \! N}
\newcommand{\Ren}{ I \! \! R^n}
\newcommand{\Proof}{{\it Proof}}
\newcommand{\qeda}{\hspace{10mm}\hfill $\square$}
\newcommand{\po}{\mathbb{K}_{\mathcal L_{\xk }}}
\newcommand{\gre}{\mathbb{G}_{\mathcal L_{\xk }}}
\newcommand{\be}{\begin{equation}}
\newcommand{\ee}{\end{equation}}
\newcommand{\bea}{\begin{eqnarray}}
\newcommand{\eea}{\end{eqnarray}}
\newcommand{\bean}{\begin{eqnarray*}}
\newcommand{\eean}{\end{eqnarray*}}
\newcommand{\la}{\label}
\newcommand{\bx}{\bar{x}}
\newcommand{\xa}{\alpha}
\newcommand{\xb}{\beta}
\newcommand{\xg}{\gamma}
\newcommand{\xG}{\Gamma}
\newcommand{\xd}{\delta}
\newcommand{\xD}{\Delta}
\newcommand{\xe}{\varepsilon}
\newcommand{\xz}{\zeta}
\newcommand{\xh}{\eta}
\newcommand{\Th}{\Theta}
\newcommand{\xk}{\kappa}
\newcommand{\xl}{\lambda}
\newcommand{\xL}{\Lambda}
\newcommand{\CC}{\mathcal{C}}
\newcommand{\xm}{\mu}
\newcommand{\xn}{\nu}
\newcommand{\ks}{\xi}
\newcommand{\KS}{\Xi}
\newcommand{\xp}{\pi}
\newcommand{\xP}{\Pi}
\newcommand{\xr}{\rho}
\newcommand{\xs}{\sigma}
\newcommand{\xS}{\Sigma}
\newcommand{\xf}{\phi}
\newcommand{\xF}{\Phi}
\newcommand{\ps}{\psi}
\newcommand{\PS}{\Psi}
\newcommand{\xo}{\omega}
\newcommand{\xO}{\Omega}
\newcommand{\tA}{\tilde{A}}
\newcommand{\tC}{\bar{k}_s}
\newcommand{\qcap}{C_{\frac{2}{q},\;q'}}
\newcommand{\qsub}{\subset^q}
\newcommand{\qeq}{\sim^q}
\newcommand{\ei}{\phi_{\xk }}
\newcommand{\BA}{
\begin{array}}
\newcommand{\EA}{
\end{array}}
\newcommand{\opname}[1]{\mbox{\rm #1}\,}
\newcommand{\supp}{\opname{supp}}
\newcommand{\dist}{\opname{dist}}
\newcommand{\myfrac}[2]{{\displaystyle \frac{#1}{#2} }}
\newcommand{\myint}[2]{{\displaystyle \int_{#1}^{#2}}}
\newcommand{\mysum}[2]{{\displaystyle \sum_{#1}^{#2}}}
\newcommand {\dint}{{\displaystyle \int\!\!\int}}
\newcommand{\q}{\quad}
\newcommand{\qq}{\qquad}
\newcommand{\hsp}[1]{\hspace{#1mm}}
\newcommand{\vsp}[1]{\vspace{#1mm}}
\newcommand{\ity}{\infty}
\newcommand{\prt}{
\partial}
\newcommand{\sms}{\setminus}
\newcommand{\ems}{\emptyset}
\newcommand{\ti}{\times}
\newcommand{\pr}{^\prime}
\newcommand{\ppr}{^{\prime\prime}}
\newcommand{\tl}{\tilde}
\newcommand{\sbs}{\subset}
\newcommand{\sbeq}{\subseteq}
\newcommand{\nind}{\noindent}
\newcommand{\ind}{\indent}
\newcommand{\ovl}{\overline}
\newcommand{\unl}{\underline}
\newcommand{\nin}{\not\in}
\newcommand{\pfrac}[2]{\genfrac{(}{)}{}{}{#1}{#2}}

\newcommand{\finedim}{{\hfill $\Box$}}

\newcommand{\bth}[1]{\def\name{Theorem}
\begin{sub}\label{t:#1}}
\newcommand{\blemma}[1]{\def\name{Lemma}
\begin{sub}\label{l:#1}}
\newcommand{\bcor}[1]{\def\name{Corollary}
\begin{sub}\label{c:#1}}
\newcommand{\bdef}[1]{\def\name{Definition}
\begin{sub}\label{d:#1}}
\newcommand{\bprop}[1]{\def\name{Proposition}
\begin{sub}\label{p:#1}}
\newtheorem{theorem}{Theorem}[section]
\newtheorem{lemma}[theorem]{Lemma}
\newtheorem{prop}[theorem]{Proposition}
\newtheorem{coro}[theorem]{Corollary}
\newtheorem{defin}[theorem]{Definition}
\newtheorem{remark}[theorem]{Remark}
\newtheorem{notation}[theorem]{Notation}

\def\ga{\alpha}     \def\gb{\beta}       \def\gg{\gamma}
\def\gc{\chi}       \def\gd{\delta}      \def\ge{\epsilon}
\def\gth{\theta}                         \def\vge{\varepsilon}
\def\gf{\phi}       \def\vgf{\varphi}    \def\gh{\eta}
\def\gi{\iota}      \def\gk{\kappa}      \def\gl{\lambda}
\def\gm{\mu}        \def\gn{\nu}         \def\gp{\pi}
\def\vgp{\varpi}    \def\gr{\rho}        \def\vgr{\varrho}
\def\gs{\sigma}     \def\vgs{\varsigma}  \def\gt{\tau}
\def\gu{\upsilon}   \def\gv{\vartheta}   \def\gw{\omega}
\def\gx{\xi}        \def\gy{\psi}        \def\gz{\zeta}
\def\Gg{\Gamma}     \def\Gd{\Delta}      \def\Gf{\Phi}
\def\Gth{\Theta}
\def\Gl{\Lambda}    \def\Gs{\Sigma}      \def\Gp{\Pi}
\def\Gw{\Omega}     \def\Gx{\Xi}         \def\Gy{\Psi}

\def\CS{{\mathcal S}}   \def\CM{{\mathcal M}}   \def\CN{{\mathcal N}}
\def\CR{{\mathcal R}}   \def\CO{{\mathcal O}}   \def\CP{{\mathcal P}}
\def\CA{{\mathcal A}}   \def\CB{{\mathcal B}}   \def\CC{{\mathcal C}}
\def\CD{{\mathcal D}}   \def\CE{{\mathcal E}}   \def\CF{{\mathcal F}}
\def\CG{{\mathcal G}}   \def\CH{{\mathcal H}}   \def\CI{{\mathcal I}}
\def\CJ{{\mathcal J}}   \def\CK{{\mathcal K}}   \def\CL{{\mathcal L}}
\def\CT{{\mathcal T}}   \def\CU{{\mathcal U}}   \def\CV{{\mathcal V}}
\def\CZ{{\mathcal Z}}   \def\CX{{\mathcal X}}   \def\CY{{\mathcal Y}}
\def\CW{{\mathcal W}} \def\CQ{{\mathcal Q}}
\def\BBA {\mathbb A}   \def\BBb {\mathbb B}    \def\BBC {\mathbb C}
\def\BBD {\mathbb D}   \def\BBE {\mathbb E}    \def\BBF {\mathbb F}
\def\BBG {\mathbb G}   \def\BBH {\mathbb H}    \def\BBI {\mathbb I}
\def\BBJ {\mathbb J}   \def\BBK {\mathbb K}    \def\BBL {\mathbb L}
\def\BBM {\mathbb M}   \def\BBN {\mathbb N}    \def\BBO {\mathbb O}
\def\BBP {\mathbb P}   \def\BBR {\mathbb R}    \def\BBS {\mathbb S}
\def\BBT {\mathbb T}   \def\BBU {\mathbb U}    \def\BBV {\mathbb V}
\def\BBW {\mathbb W}   \def\BBX {\mathbb X}    \def\BBY {\mathbb Y}
\def\BBZ {\mathbb Z}

\def\GTA {\mathfrak A}   \def\GTB {\mathfrak B}    \def\GTC {\mathfrak C}
\def\GTD {\mathfrak D}   \def\GTE {\mathfrak E}    \def\GTF {\mathfrak F}
\def\GTG {\mathfrak G}   \def\GTH {\mathfrak H}    \def\GTI {\mathfrak I}
\def\GTJ {\mathfrak J}   \def\GTK {\mathfrak K}    \def\GTL {\mathfrak L}
\def\GTM {\mathfrak M}   \def\GTN {\mathfrak N}    \def\GTO {\mathfrak O}
\def\GTP {\mathfrak P}   \def\GTR {\mathfrak R}    \def\GTS {\mathfrak S}
\def\GTT {\mathfrak T}   \def\GTU {\mathfrak U}    \def\GTV {\mathfrak V}
\def\GTW {\mathfrak W}   \def\GTX {\mathfrak X}    \def\GTY {\mathfrak Y}
\def\GTZ {\mathfrak Z}   \def\GTQ {\mathfrak Q}

\newcounter{newsec} \renewcommand{\theequation}{\thesection.\arabic{equation}}
\abstract{ We study the boundary behaviour of the of (E) $-\Gd u-\myfrac{\xk }{d^2(x)}u+g(u)=0$, where $0<\xk <\frac{1}{4}$ and $g$ is a continuous nonndecreasing function in a bounded convex domain of $\BBR^N$. We first construct the Martin kernel associated to the the linear operator
$\CL_{\xk }=-\Gd-\frac{\xk }{d^2(x)}$ and give a general condition for solving equation (E) with any Radon measure $\gm$ for boundary data. When $g(u)=|u|^{q-1}u$ we show the existence of a critical exponent $q_c=q_c(N,\xk )>1$: when $0<q<q_c$ any measure is eligible for solving (E) with
$\gm$ for boundary data; if $q\geq q_c$, a necessary and sufficient condition is expressed in terms of the absolute continuity of $\gm$with respect to some Besov capacity. The same capacity characterizes the removable compact boundary sets. At end any positive solution
(F) $-\Gd u-\frac{\xk }{d^2(x)}u+|u|^{q-1}u=0$ with $q>1$ admits a boundary trace which is a positive outer regular Borel measure. When  $1<q<q_c$ we prove that to any positive outer regular Borel measure we can associate a positive solutions of ($F$) with  this boundary trace. }
\tableofcontents\medskip

  \vspace{1mm}

\medskip

\section{Introduction}
Let $\xO$ be a  bounded smooth domain in $\BBR^N$ and $d(x)=\dist(x,\Gw^c)$. In this article we study several aspects of the nonlinear boundary value associated to the equation
\be\label{IE1}
-\Gd u-\myfrac{\xk }{d^2(x)}u+|u|^{p-1}u=0\qquad\text{in }\;\Gw
\ee
where $p>1$. The study of the boundary trace of solutions of $(\ref{IE1})$ is a natural framework for a general study of several nonlinear problems where the nonlinearity, the geometric properties of the domain and the coefficient $\xk $ interact. On this point of view, the case $\xk =0$ has been thoroughly treated  by Marcus and V\'eron \cite{MV-JMPA01}, \cite{MV-CPAM}, \cite{MV}, \cite{MV-CONT}, for example and the synthesis presented in \cite{book}. The associated linear Schr\"{o}dinger operator
\be\label{IE2}
u\mapsto \CL_{\xk }u:= -\Gd u-\myfrac{\xk }{d^2(x)}u
\ee
plays an important role in functional analysis because of the particular singularity of is
potential $V(x):=-\frac{\xk }{d^2(x)}$. The case $\xk <0$ and more generally of nonnegative potential has been studied by Ancona \cite {An} who has shown the existence of a Martin kernel which allows a general representation formula of nonnegative solutions of
\be\label{IE3}
\CL_{\xk }u=0\qquad\text{in }\;\Gw,
\ee
This representation turned out to be the key ingredient of the full classification of positive solutions of
\be\label{IE4}
-\Gd u+u^q=0\qquad\text{in }\;\Gw
\ee
which was obtained by Marcus \cite{Ma}. In a more general setting, V\'eron and Yarur \cite {VY} constructed a capacitary theory associated to the linear equation
\be\label{IE5}
\CL_{V}u:=-\Gd u+V(x)u=0\qquad\text{in }\;\Gw,
\ee
where the potential $V$ is nonnegative and singular near $\prt\Gw$. When  $V(x):=-\frac{\xk }{d^2(x)}$ with $\xk >0$, $V$ is called a Hardy potential. There is a critical value $\xk =\frac{1}{4}$. If $\xk >\frac{1}{4}$, no positive solution of $(\ref{IE3})$ exists.
When $0<\xk \leq \frac{1}{4}$, there exist positive solutions and the geometry of the domain plays a fundamental role in the study of the mere linear equation  $(\ref{IE3})$.
We define the constant $c_\Gw$ by
\be\label{IE6}
c_\Gw=\inf_{v\in H^1_0(\Gw)\setminus\{0\}}\myfrac{\myint{\Gw}{}|\nabla v|^2dx}{\myint{\Gw}{}\frac{v^2}{d^{2}(x)}dx}.
\ee
It is known that $c_\Gw$ belongs to $(0,\frac{1}{4}]$. If $\Gw$ is convex or if the distance function $d$ is super harmonic in the sense of distributions, then $c_\Gw=\frac{1}{4}$. Furthermore
there holds $c_\xO=\frac{1}{4}$ if and only if problem (\ref{IE6}) has no minimizer. (see \cite{hardy-marcus}).
When $ 0< \xk \leq \frac{1}{4}$, {\it which is which is always assumed in the sequel} and $-\xD d\geq0$ in the sense of distributions, it is possible to define the first eigenvalue $\xl_{\xk }$ of the operator $\CL_{\xk }$. If we define the two fundamental exponents $\ga_+$ and $\ga_-$ by
\be\label{alpha}
\ga_+=1+\sqrt{1-4\xk }\quad \text{and}\quad\ga_-=1-\sqrt{1-4\xk }
\ee
then the first eigenvalue is achieved by an eigenfunction $\gf_{\xk }$ which satisfies $\xf_{\xk }(x)\approx d^{\frac{\xa_+}{2}}(x)$ as $d(x)\to 0$. Similarly, the Green kernel $G_{\CL_{\xk }}$ associated to $\CL_{\xk }$ inherits this type of boundary behaviour since there holds
\be\label{alpha-green}
\frac{1}{C_\gk}\min\left\{\frac{1}{|x-y|^{N-2}},\frac{d^{\frac{\ga_+}{2}}(x)d^{\frac{\ga_+}{2}}(y)}{|x-y|^{N+\ga_+-2}}\right\}\leq G_{\CL_{\xk }}(x,y)\leq C_\gk\min\left\{\frac{1}{|x-y|^{N-2}},\frac{d^{\frac{\ga_+}{2}}(x)d^{\frac{\ga_+}{2}}(y)}{|x-y|^{N+\ga_+-2}}\right\}
\ee
We show that $\CL_{\xk }$ satisfies the maximum principle in the sense that if $u\in H^1_{loc}\cap C(\Gw)$ is a subsolution i.e.  $\CL_{\xk }u\leq 0$ such that
\be\label{subsol}\BA {llll}\displaystyle
(i) \quad&\displaystyle\limsup_{x\to y}\myfrac{u(x)}{d^{\ga_-}(x)}\leq 0\quad&\text{if } 0< \xk < \frac{1}{4}\\
\displaystyle(ii)\quad&\displaystyle\limsup_{x\to y}\myfrac{u(x)}{\sqrt{d(x)}|\ln d(x)|}\leq 0\quad&\text{if } \xk = \frac{1}{4}
\EA\ee
for all $y\in\prt\Gw$, then $u\leq 0$. If $\xi\in\prt\Gw$ and $r>0$, we set $\Gd_r(\xi)=\prt\Gw\cap B_r(\xi)$. We prove that a positive solution of $\CL_{\xk }u=0$ which vanishes on a part of the boundary in the sense that

\be\label{vanis}\BA {llll}\displaystyle
(i) \quad&\displaystyle\lim_{x\to y}\myfrac{u(x)}{d^{\ga_-}(x)}= 0\quad&\forall y\in\Gd_r(\xi) \quad&\text{if } 0< \xk < \frac{1}{4}\\
\displaystyle(ii)\quad&\displaystyle\lim_{x\to y}\myfrac{u(x)}{\sqrt{d(x)}|\ln d(x)|}= 0\quad&\forall y\in\Gd_r(\xi) \quad&\text{if } \xk = \frac{1}{4},
\EA\ee
satisfies
\be\label{Comp}
\myfrac{u(x)}{\gf_{\xk }(x)}\leq C_1\myfrac{u(y)}{\gf_{\xk }(y)}\qquad\forall x,y\in \Gd_{\frac r2}(\xi),
\ee
for some $C_1=C_1(\Gw,\xk )>0$. \medskip

For any $h\in C(\prt\Gw)$ we construct the unique solution $v:=v_h$ of the Dirichlet problem
\be\label{Diri}\BA {llll}
\CL_{\xk }v=0\qquad\text{in }\Gw\\
\phantom{\CL_{\xk }}
v=h\qquad\text{on }\prt\Gw
\EA\ee
Using this construction and estimates (\ref{subsol}) we  show the existence of the $\CL_{\xk }$-measure, which is a Borel measure $\gw^x$ with the property that for any $h\in C(\prt\Gw)$, the above function $v_h$ satisfies
\be\label{harm}\BA {llll}
v_h(x)=\myint{\prt\Gw}{}h(y)d\gw^x(y).
\EA\ee
Because of Harnack inequality, the measures $\gw^x$ and $\gw^z$ are mutually absolutely continuous for $x,z\in\Gw$ and for any $x\in\Gw$ we can define the Radon-Nikodym derivative
\be\label{RD}
K(x,y):=\myfrac{d\gw^x}{d\gw^{x_0}}(y)\quad\text{for }\gw^{x_0}\text{-almost }y\in\prt\Gw.
\ee
There exists $r_0:=r_0(\Gw)$ such that for any $x\in\Gw$ such that $d(x)\leq r_0$, there exists a unique $\xi=\xi_x\in\prt\Gw$ such that
$d(x)=|x-\xi_x|$. If we denote by $\Gw'_{r_0}$ the set of $x\in\Gw$ such that $0< d(x)< r_0$, the mapping $\Gp$ from $\overline\Gw'_{r_0}$ to $[0,r_0]\ti\prt\Gw$ defined by $\Gp(x)=(d(x),\xi_x)$ is a $C^1$ diffeomorphism.
If $\xi\in\prt\Gw$ and $0\leq r\leq r_0$, we set $x_r(\xi)=\Gp^{-1}(r,\xi)$. Let $W$ be defined in $\Gw$ by
\be\label{W}
W(x)=\left\{\BA {lll}d^{\frac{\ga_-}{2}}(x)\qquad&\text{if }\;\xk <\frac{1}{4}\\[2mm]
\sqrt{d(x)}|\ln d(x)|\qquad&\text{if }\;\xk =\frac{1}{4},
\EA
\right.
\ee
we prove that the $\CL_{\xk }$-harmonic measure can be equivalently defined by
\be\label{harm-1}
\gw^x(E)=\inf\left\{\psi:\psi\in C_+(\Gw), \;\CL_{\xk }\text{-superharmonic in }\Gw\text{ and s.t. }\liminf_{x\to E}\myfrac{\psi(x)}{W(x)}\geq 1\right\}
\ee
for any compact set $E\subset\prt\Gw$ and then extended classically to Borel subsets of $\prt\Gw$.
\medskip

The $\CL_{\xk }$-harmonic measure is connected to the Green kernel of $\CL_{\xk }$ by the following estimates\medskip

\noindent{\bf Theorem A} {\it There exists $C_3:=C_3(\Gw)>0$ such that for any $r\in (0,r_0]$ and $\xi\in\prt\Gw$, there holds
\be\label{harm2}\BA {llll}
\frac{1}{C_3}r^{N+\frac{\ga_-}{2}-2}G_{\CL_{\xk }}(x_r(\xi),x)
&\leq\gw^x(\Gd_r(\xi))\\&\leq C_3r^{N+\frac{\ga_-}{2}-2}G_{\CL_{\xk }}(x_r(\xi),x)\quad\forall x\in\Gw\setminus B_{4r(\xi)}
\EA\ee
if $0<\xk <\frac{1}{4}$, and
\be\label{harm3}\BA {llll}
\frac{1}{C_3}r^{N-2+\frac{1}{2}}|\ln d(x)|G_{\CL_{\frac{1}{4}}}(x_r(\xi),x)&\leq\gw^x(\Gd_r(\xi))\\&\leq C_3r^{N-2+\frac{1}{2}}|\ln d(x)|G_{\CL_{\frac{1}{4}}}(x_r(\xi),x)\quad\forall x\in\Gw\setminus B_{4r(\xi)}.
\EA\ee
}
As a consequence $\gw^x$ has the doubling property. The previous estimates allow to construct a kernel function
 of $\CL_{\xk }$ in $\Gw$, prove its uniqueness up to an homothety.  When normalized, the kernel function denoted by $K_{\CL_{\xk }}$
 is the Martin kernel, defined by
\be\label{martin}\BA {llll}
K_{\CL_{\xk }}(x,\xi)=\displaystyle\lim_{x\to\xi}\myfrac{G_{\CL_{\xk }}(x,y)}{G_{\CL_{\xk }}(x,x_0)}\quad\forall\xi\in\prt\Gw.
\EA\ee
for some $x_0\in\Gw$. An important property of the Martin kernel is that it allows to represent a positive   $\CL_{\xk }$-harmonic function $u$ by mean of a Poisson type formula which endows the form

\be\label{martin1}\BA {llll}
u(x)=\myint{\prt\Gw}{}K_{\CL_{\xk }}(x,\xi)d\gm(\xi)\quad\text{for }\gw^{x_0}\text{-almost }x\in\Gw
\Gw.
\EA\ee
for some positive Radon measure $\gm$ on $\prt\Gw$. The measure $\gm$ is called the boundary trace of $u$. Furthermore
$K_{\CL_{\xk }}$ satisfies the following two-side estimates\medskip

\noindent{\bf Theorem B} {\it There exists $C_3:=C_3(\Gw,\xk )>0$ such that for any $(x,\xi)\in\Gw\ti\prt\Gw$ there holds
\be\label{harm4}\BA {llll}
\myfrac{1}{C_3}\myfrac{d^{\frac{\ga_+}{2}}}{|x-\xi|^{N+\ga_+-2}}\leq K_{\CL_{\xk }}(x,\xi)\leq C_3\myfrac{d^{\frac{\ga_+}{2}}}{|x-\xi|^{N+\ga_+-2}}.
\EA\ee}

Thanks to these estimates we can adapt the approach developed in \cite{GV} to prove the existence of weak solutions to the nonlinear boundary value problem
\be\label{Non1}\BA {ll}
-\Gd u-\myfrac{\xk }{d^2(x)}u+g(u)=\gn\qquad\text{in }\;\Gw\\[2mm]
\phantom{-\Gd -\myfrac{\xk }{d^2(x)}u+g(u)}
u=\gm\qquad\text{in }\;\prt\Gw,
\EA\ee
where $g$ is a continuous nondecreasing function such that $g(0)\geq 0$ and $\gn$ and $\gm$ are Radon measures on $\Gw$ and $\prt\Gw$ respectively .
We define the class ${\bf X}_\gk(\Gw)$ of test functions by
\be\label{Non2}
{\bf X}_\gk(\Gw)=\left\{\BA {llll}
\eta\in L^2(\Gw)\text{ s.t. }\nabla (d^{-\frac{\ga_+}{2}}\eta)\in L^2_{\xf_k}(\Gw)\text{ and }\gf^{-1}_{\xk }\CL_\gk\eta\in L^{\infty}(\Gw)
\EA\right\}\ee
and we prove\medskip

\noindent{\bf Theorem C} {\it Assume $g$ satisfies
\be\label{Non3}\BA {llll}
\myint{1}{\infty}(g(s)+|g(-s)|)s^{-2\frac{N-1+\frac{\ga_+}{2}}{N-2+\frac{\ga_+}{2}}} ds<\infty.
\EA\ee
Then for any Radon measures $\gn$ on $\Gw$ and such that $\int_\Gw\gf_\gk d|\gm|<\infty$ and $\gm$ on $\prt\Gw$ there exists a unique $u\in L^1_{\gf_{\xk }}(\Gw)$ such that
$g(u)\in L^1_{\gf_{\xk }}(\Gw)$ which satisfies
\be\label{Non3'}\BA {llll}
\myint{\Gw}{}\left(u\CL_{\xk }\eta +g(u)\eta\right) dx=\myint{\Gw}{}\left(\eta d\gn+\BBK_{\CL_{\xk }}[\gm]\CL_{\xk }\eta dx\right)\quad\forall\eta\in {\bf X}_\gk(\Gw).
\EA\ee}

When $g(r)=|r|^{q-1}r$ the critical value is $q_c=\frac{N+\frac{\ga_+}{2}}{N+\frac{\ga_+}{2}-2}$ and (\ref{Non3}) is satisfied for
$0\leq q<q_c$ (the subcritical range). In this range of values of $q$, existence and uniqueness of a solution to 
\be\label{Non4}\BA {llll}
-\Gd u-\myfrac{\xk }{d^2(x)}u+|u|^{q-1}u=0\qquad\text{in }\;\Gw\\[2mm]
\phantom{-\Gd -\myfrac{\xk }{d^2(x)}u+|u|^{q-1}u}
u=\gm\qquad\text{in }\;\prt\Gw,
\EA\ee
has been recently obtained by Marcus and Nguyen \cite{MaNg}. When $q\geq q_c$ not all the Radon measures are eligible for solving problem (\ref{Non4}).\medskip 

We prove the following result in which statement $C^{\BBR^{N-1}}_{2-\frac{2+\ga_+}{2q'},q'}$ denotes the Besov capacity associated to the Besov space $B^{2-\frac{2+\ga_+}{2q'},q'}(\BBR^{N-1})$.\medskip

\noindent{\bf Theorem D} {\it  Assume $q\geq q_c$ and $\gm$ is a positive Radon measure on $\prt\Gw$. Then problem (\ref{Non4}) admits a weak solution if and only if $\gm$ vanishes on Borel sets $E\subset\prt\Gw$ such that $C^{\BBR^{N-1}}_{2-\frac{2+\ga_+}{2q'},q'}(E)=0$.}\medskip

Note that a special case of this result is proved in (\cite{MaNg}) when $\gm=\gd_a$ for a boundary point and $q\geq q_c$. In that case $\gd_a$  does not vanish on $\{a\}$ although $C^{\BBR^{N-1}}_{2-\frac{2+\ga_+}{2q'},q'}(\{a\})=0$. \smallskip

This capacity plays a fundamental for characterizing the removable compact boundary sets which exist only in the {\it supercritical range $q\geq q_c$}.\medskip

\noindent{\bf Theorem E} {\it  Assume $q\geq q_c$ and $K\subset\prt\Gw$ is compact. Then any function $u\in C(\overline \Gw\setminus K)$ which satisfies
\be\label{Non5}\BA {llll}
-\Gd u-\myfrac{\xk }{d^2(x)}u+|u|^{q-1}u=0\qquad\text{in }\;\Gw\\[2mm]
\phantom{-\Gd -\myfrac{\xk }{d^2(x)}u+|u|^{q-1}u}
u=0\qquad\text{in }\;\prt\Gw\setminus K,
\EA\ee
is identically zero if and only if $C^{\BBR^{N-1}}_{2-\frac{2+\ga_+}{2q'},q'}(K)=0$.}
\medskip

We show that any positive solution $u$ of (\ref{IE1}) admits a boundary trace, and more precisely we prove that the following dichotomy holds: et \medskip

\noindent{\bf Theorem F} {\it Let $u$ be a positive solution of  (\ref{IE1})  in $\Gw$ and $a\in\prt\Gw$. Then\smallskip

\noindent (i) either for any $\ge>0$
\be\label{Tr1}\BA {llll}\displaystyle
\lim_{\gd\to 0}\myint{\Gs_\gd\cap B_\ge(a)}{}ud\gw_{\Gw'_\gd}^{x_0}=\infty,
\EA\ee
where $\Gw'_\gd=\{x\in\Gw:d(x)>\gd\}$, $\Gs_\gd=\prt\Gw'_\gd$ and $\gw_{\Gw'_\gd}^{x_0}$ is the harmonic measure in $\Gw'_\gd$,
\smallskip

\noindent (ii) or there exists $\ge_0>0$ and  a positive Radon measure $\gl$ on $\prt\Gw\cap B_{\ge_0}(a)$ such that for any $Z\in C(\overline\Gw)$ with support in $\Gw\cup(\prt\Gw\cap B_{\ge_0}(a))$, there holds
\be\label{Tr2}\BA {llll}\displaystyle
\lim_{\gd\to 0}\myint{\Gs_\gd\cap B_\ge(a)}{}Zud\gw_{\Gw'_\gd}^{x_0}=\myint{\prt\Gw\cap B_\ge(a)}{}Zd\gl.
\EA\ee.} \medskip

The set of points $a\in\prt\Gw$ such that (i) (resp. (ii)) holds is closed (resp. relatively open) and denoted by $\CS_u$ (resp $\CR_u$). There exists a unique radon measure $\gm_u$ on $\CR_u$ such that, for any $Z\in C(\overline\Gw)$ with support in $\Gw\cup\CR_u$ there holds
\be\label{Tr3}\BA {llll}\displaystyle
\lim_{\gd\to 0}\myint{\Gs_\gd}{}Zud\gw_{\Gw'_\gd}^{x_0}=\myint{\CR_u}{}Zd\gm_u.
\EA\ee
 The couple $(\CS_u,\gm_u)$ is called {\it the boundary trace of $u$} and denoted by $Tr_{\prt\Gw}(u)$. A notion of normalized boundary trace of positive  {\it moderate solutions} of (\ref{IE1}), i.e. solutions such that $u\in L^q(\gf_\gk)$,  is developed in \cite{MaNg}. They proved the existence of a boundary trace $\gm\approx (\{\emptyset\},\gm_u)$ and corresponding representation of $u$ via the Martin and Green kernels.
\medskip

If $1<q<q_c$ we denote by $u_{k\gd_a}$  positive solution of (\ref{IE1}) with $\gm=k\gd_a$ for some $a\in\prt\Gw$ and $k\geq 0$. There exists $\lim_{k\to\infty}u_{k\gd_a}=u_{\infty,a}$. We prove the following\medskip

\noindent{\bf Theorem G} {\it  Assume $1<q< q_c$ and $a\in\prt\Gw$. Then If $u$ is a positive solution of (\ref{IE1}) such that
$a\in \CS_u$, then $u\geq u_{\infty,a}$.}\medskip

In order to go further in the study of boundary singularities, we construct separable solutions of (\ref{IE1})  in
$\BBR^N_+=\{x=(x',x_N):x_N>0\}=\{(r,\gs)\in \BBR_+\ti S^{N-1}_+\}$ which vanish on
$\prt\BBR^N_+\setminus\{0\}$ under the form $u(r,\gs)=r^{-\frac{2}{q-1}}\gw(\gs)$, where $r>0$, $\gs\in S^{N-1}_+$. They are solutions of
\be\label{Tr5}\BA {llll}
-\Gd_{S^{N-1}}\gw-\ell_{q,N}\gw-\myfrac{\xk }{{\bf e}_N.\gs}\gw+|\gw|^{q-1}\gw=0\qquad&\text{in }S^{N-1}_+\\
\phantom{-\Gd_{S^{N-1}}\gw-\ell_{q,N}\gw-\myfrac{\xk }{{\bf e}_N.\gs}\gw+|\gw|^{q-1}}
\gw=0&\text{in }\prt S^{N-1}_+
\EA\ee
where $\Gd_{S^{N-1}}$ is the Laplace-Beltrami operator, ${\bf e}_N$ the unit vector pointing toward the North pole and $\ell_{q,N}$ is a positive constant. We prove that if $1<q<q_c$ problem (\ref{Tr5}) admits a unique positive solution $\gw_\gk$ while no such solution exists if $q\geq q_c$. To this phenomenon is associated a result of classification of positive solutions of (\ref{IE1}) in $\Gw$ which vanishes of $\prt\Gw\setminus\{0\}$ (here we assume that $0\in\prt\Gw$ and that the tangent plane to $\prt\Gw$ at $0$ is $\{x:x.{\bf e}_N=0\},$ there exists $r_0>0$ such that $B_{r_0}(r_0{\bf e}_N) \subset \xO,$  $B_{r_0}(r_0{\bf e}_N)\subset\{x:x.{\bf e}_N\geq 0\}$ and $d(r_0{\bf e}_N)=|r_0{\bf e}_N|=r_0$)\medskip

\noindent{\bf Theorem H} {\it Assume $1<q< q_c$ and let $u\in C(\overline\Gw\setminus\{a\}$ be a solution of (\ref{IE1}) in $\Gw$ which vanishes of $\prt\Gw\setminus\{a\}$. Then\smallskip

\noindent (i) Either $u=u_{\infty,a}$ and
\be\label{Tr6}\BA {llll}
\lim_{r\to 0}r^{\frac{2}{q-1}}u(r,.)=\gw_\gk
\EA\ee
locally uniformly in $S^{N-1}_+$.\smallskip

\noindent (ii) Or there exists $k\geq 0$ such that $u=u_{k\gd_a}$ and
\be\label{Tr7}\BA {llll}
u(x)=kK_{\CL_{\xk }}(x,a)(1+o1))\qquad\text{as }x\to 0
\EA\ee

}. \medskip

If $1<q< q_c$ we prove that to any couple $(F,\gm)$ where $F$ is a closed subset of $\prt\Gw$ and $\gm$ a positive Radon measure on $R=\prt\Gw\setminus F$ we can associate a positive solution $u$ of (\ref{IE1}) in $\Gw$ with $Tr_{\prt\Gw}(u)=(F,\gm)$.

\section{The linear operator $\CL_{\xk }=-\xD-\frac{\xk }{d^2(x)}$}
\setcounter{equation}{0}

Throughout this article $c_j$  (j=1,2,...) denote positive constants the value of which may change from one occurrence to another. The notation $\xk $ is reserved to the value of the coefficient of the Hardy potential

\subsection {The eigenvalue problem}
We recall some known results concerning the eigenvalue problem (see \cite {dd}, \cite {F.M.T2}).\smallskip

\noindent 1- Since $\Gw$ is convex, $c_\Gw=\frac{1}{4}$ and for any $\xk \in (0,\frac{1}{4}]$ there exists
$$\gl_{\xk }=\inf_{u\in H_0^1(\Gw)}\myfrac{\myint{\Gw}{}\left(|\nabla u|^2-\myfrac{\xk }{d^2}u^2\right)dx}
{\myint{\Gw}{} u^2 dx}.
$$

\noindent 2-  If $0<\xk <\frac{1}{4}$ the minimizer $\gf_{\xk }$ belongs $H^1_0(\Gw)$ and it satisfies
\be\label{Lin1}
\gf_{\xk }\approx d^{\frac{\ga_+}{2}}(x),
\ee
where $\ga_+$ (as well as $\ga_-$) are defined by (\ref{alpha}). \smallskip

\noindent 3- If $\xk =\frac{1}{4}$, there exists a non-negative minimizer $\xf_{\frac{1}{4}}\in H_{loc}^1(\xO)$  such that
\be\label{Lin2}\xf_{\frac{1}{4}}\approx d^\frac{1}{2}(x).\ee
Furthermore, the function $\psi_{\frac{1}{4}}=d^{-\frac{1}{2}}$ belongs to $H^1_0(\Gw;d(x)dx)$
\smallskip

\noindent 4-  Let $H_0^1(\xO,d^\xa(x)dx)$ denote the closure of $C_0^\infty(\xO)$ functions under the norm
\be\label{behd(x)}||u||_{H_0^1(\xO,d^\xa(x)dx)}^2=\int_\xO|\nabla u|^2d^\xa(x)dx+\int_\xO |u|^2d^\xa(x)dx.\ee
If $\ga\geq 1$ there holds \cite[Th. 2.11]{F.M.T2}
\be\label{test}
H_0^1(\xO,d^{\xa}(x)dx)=H^1(\xO,d^\xa(x)dx)\qquad\forall\xa\geq1.
\ee
\smallskip

\noindent 5- Let $0<\gk\leq\frac{1}{4}$. Let ${\bf H}_\gk(\Gw)$ be the subset of functions of $H^1_{loc}(\Gw)$ satisfying
\be\label{Wnorm1}
\myint{\Gw}{}\left(|\nabla\gf|^2-\frac{\gk}{d^2}\gf^2\right)dx<\infty.
\ee
Then the mapping
\be\label{Wnorm2}
\gf\mapsto \left(\myint{\Gw}{}\left(|\nabla\gf|^2-\frac{\gk}{d^2}\gf^2\right)dx\right)^{\frac{1}{2}}
\ee
is a norm on ${\bf H}_\gk(\Gw)$. The closure ${\bf W}_\gk(\Gw)$ of  $C_0^\infty(\xO)$ into ${\bf H}_\gk(\Gw)$ satisfies
\be\label{Wnorm3}
{\bf W}_{\gk}(\Gw)= H^{1}_0(\Gw)\quad\forall 0<\gk<\frac{1}{4}\,\text{ and }\;{\bf W}_{\frac{1}{4}}(\Gw)\subset W^{1,q}_0(\Gw)\quad\forall 1\leq q<2,
\ee
see \cite[Th B]{BFT1}. As a consequence ${\bf W}_{\gk}(\Gw)$ is compactly imbedded into $L^r(\Gw)$ for any $r\in [1,2^*)$.\smallskip

\noindent 6- Let $\ga>0$ and $\Gw\subset\BBR^N$ be a bounded domain. There exists $c^*>0$ depending on diam($\xO$), N and $\ga$ such that for any $v\in C^{\infty}_0(\Gw)$
\be\label{Sobolev}
\left(\myint{\Gw}{}|v|^{\frac{2(N+\ga)}{N+\ga-2}}d^\ga dx\right)^{\frac{N+\ga-2}{N+\ga}}
\leq c^*\myint{\Gw}{}|\nabla v|^2d^\ga dx.
\ee
For a proof see \cite[Th. 2.9]{F.M.T2}.\medskip

The boundary behaviour of the first eigenfunction yield two-side similar estimates of the Green kernel
for Schr\"odinger operators with a general Hardy type potentials \cite[Corollary 1.9]{F.M.T2}.

\begin{prop}
Consider the operator $E:=-\xD -V,$ in $\xO$ where $V=V_1+V_2,$
with
$$|V_1|\leq\frac{1}{4d^2(x)}\quad\mathrm{and }\;V_2\in L^p(\xO),\;p>\frac{N}{2}.$$
We also assume that

$$0<\xl_1:=\inf_{u\in H_0^1(\xO)}\myfrac{\myint{\xO}{}\left(|\nabla u|^2dx- Vu^2\right)dx}{\myint{\xO}{}u^2dx},$$
and that to $\xl_1$ is associated a positive eigenfunction $\xf_1$. If, for some  $\xa\geq1$ and $C_1,C_2>0$, there holds
$$c_1 d^\frac{\xa}{2}(x)\leq \xf_1(x)\leq c_2 d^\frac{\xa}{2}(x)\qquad\forall x\in\Gw,$$
then the Green kernel $G^\Gw_E$ associated to $E$ in $\xO$ satisfies
\be\label{greenest}G^\Gw_E(x,y)\approx c_3\min\left(\frac{1}{|x-y|^{N-2}},\frac{d^\frac{\xa}{2}(x)d^\frac{\xa}{2}(y)}{|x-y|^{N+\xa-2}}\right).\ee\label{green}
\end{prop}
We set
\be\label{foliat}\xO_\xd=\{x\in\xO:\;d(x)<\xd\}\,,\; \xO'_\xd=\{x\in\xO:\;d(x)>\xd\}\,\text{and }\; \xS_\xd=\{x\in\xO:\;d(x)=\xd\}.\ee
\begin{defin}
Let $G\subset\xO$
 and let $H^1_c(G)\subset H^1(G)$ denote the subspace of functions with
compact support.
A function $h\in W^{1,1}_{loc}(G)$ is $\CL_{\xk }$-harmonic in $G$ if
$$\int_G\nabla h.\nabla\psi dx-\xk \int_\xO\frac{1}{d^2(x)}h\psi dx=0\qquad\forall\psi\in H_c^1(G).$$
A function $\underline{h}\in  H^1_{loc}(G)\cap C(G)$ is $\CL_{\xk }$-subharmonic in $G$ if
$$\int_G\nabla \underline{h}.\nabla\psi dx-\xk \int_\xO\frac{1}{d^2(x)}h\psi dx\leq0\qquad\forall\psi\in H_c^1(G), \;\psi\geq 0.$$
We say that $\underline{h}$ is a local $\CL_{\xk }$-subharmonic function if there exists $\xd>0$
 such that
$\underline{h}\in H^1_{loc}(\xO_\xd)\cap C(\xO_\xd)$ is $\CL_{\xk }$-subharmonic in
$\xO_\xd.$ Similarly, (local) $\CL_{\xk }$-superharmonics $\overline{h}$ are
defined with $"\geq"$ in the above inequality.
\end{defin}
Note that  $\CL_{\xk }$-harmonic functions are $C^2$ in $G$ by standard elliptic equations regularity theory. The Phragmen-Lindel\"of principle yields the following alternative.
\begin{prop}Let $\xk \leq\frac{1}{4}$. If $\underline h$ is a local $\CL_{\xk }$-subharmonic function, then the following alternative holds:\smallskip

\noindent (i) either for every local positive $\CL_{\xk }$-superharmonic function $\overline h$
\begin{equation}\label{PhLi1}
\limsup_{d(x)\to 0}\myfrac{\underline h(x)}{\overline h(x)}>0,
\end{equation}
\smallskip

\noindent (ii) or for every local positive $\CL_{\xk }$-superharmonic function $\overline h$
\begin{equation}\label{PhLi2}
\limsup_{d(x)\to 0}\myfrac{\underline h(x)}{\overline h(x)}<\infty.
\end{equation}
\end{prop}

\begin{defin} If a local $\CL_{\xk }$-subharmonic function $\underline h$ satisfies (i) (resp. (ii)) it is called a large $\CL_{\xk }$-subharmonic ((resp. a small $\CL_{\xk }$-subharmonic).
\end{defin}

The next statement is \cite[Theorem 2.9]{1}.

\begin{prop}
Let $\underline{h}$ be a small local $\CL_{\xk }$-subharmonic
of $\CL_{\xk }.$

\noindent (i) If $\xk <\frac{1}{4}$ then the following alternative holds:
$$\mathrm{either}\;\;\;\;\limsup_{x\rightarrow\partial\xO}\frac{\underline{h}}{{d^{\frac{\xa_-}{2}}}}>0
\quad\mathrm{or}\;\;\;\;\limsup_{x\rightarrow\partial\xO}\frac{\underline{h}}{d^{\frac{\xa_+}{2}}}<\infty.$$
(ii) If $\xk =\frac{1}{4}$ then the following alternative holds:
$$\mathrm{either}\;\;\;\;\limsup_{x\rightarrow\partial\xO}\frac{\underline{h}}{d^{\frac{1}{2}}\log(\frac{1}{d})}>0
\quad\mathrm{or}\;\;\;\;\limsup_{x\rightarrow\partial\xO}\frac{\underline{h}}{d^{\frac{1}{2}}}<\infty.$$
\end{prop}
\begin{defin}
Let $f_0 \in L^2_{loc}(\xO)$. We say that a function $u\in H^1_{loc}(\xO)$ is a solution of
\begin{equation}\label{ELin1}
\CL_{\xk }u=f_0\qquad\text{in }\,\Gw
\end{equation}
if there holds
\begin{equation}\label{ELin2}
\int_\xO\nabla u.\nabla\psi dx-\xk \int_\xO\frac{1}{d^2(x)}u\psi dx=\int_\xO f_0\psi dx\qquad\forall\psi\in C_0^\infty(\xO).
\end{equation}
\end{defin}
\subsection{Preliminaries}

In this part we study some regularity properties of solutions of linear equations involving $\CL_{\xk }$.
\begin{lemma}\label{space}
(i) If $\xa>1$ and $d^{-\frac{\xa}{2}}u\in H^1(\xO,d^{\xa}(x)dx),$ then $u\in H_0^1(\xO).$

\noindent(ii) If $\xa=1$ and $d^{-\frac{1}{2}}u\in H^1(\xO,d(x)dx),$ then $u\in W_0^{1,p}(\xO),\;\forall p<2.$
\end{lemma}  

\begin{proof}
There exists $\gb_0>0$ such that $d\in C^2(\overline{\xO_{\gb_0}})$ and set $u=d^{\frac{\xa}{2}}v$. In the two cases (i)-(ii), our assumptions imply
\begin{equation}\label{ELin5}
u\in L^2(\Gw)\quad\text{and }\;\nabla u-\frac{\ga}{2}u d^{-1}\nabla d\in L^2(\Gw).
\end{equation}
(i)  Since $v\in H^1(\xO,d^{\xa}(x)dx),$ by \eqref{test} there exists a sequence $v_n\in C^\infty_0(\xO)$ such that $v_n\rightarrow v$ in $H^1(\xO,d^{\xa}(x)dx).$ Set $u_n=d^\xa v_n.$ Let $0<\xb\leq \frac{\gb_0}{2}$  and $\psi_\xb$ be a cut of function such that $\psi_\xb=0$ in $\xO_{\gb}'$ and $\psi_\xb=1$ in $\xO_{\frac{\gb}{2}}.$
Then $u_n=d^{\frac{\xa}{2}}(\psi_\xb v_n+(1-\psi_\xb)v_n).$ Thus it is enough to prove that $\widetilde{u}_n=d^{\frac{\xa}{2}}\psi_\xb v_n$ remains bounded in $H^1(\xO)$ independently of $n$. Set $w_n=\psi_\xb v_n$, then
$$\int_\xO|\nabla \widetilde{u}_n|^2dx=\int_{\xO_{\gb}}|\nabla w_n|^2dx\leq c_4\left(\int_{\xO_{\gb}} {d^{\xa}}|\nabla w_n|^2dx+\int_{\xO_{\gb}} d^{\xa-2}w_n^2dx \right).$$
Note that $\xa-2>-1.$ Now
$$\int_{\xO_{\gb}} d^{\xa-2}w_n^2dx=\frac{1}{\xa-1}\int_{\xO_{\gb}} w_n^2\mathrm{div} (d^{\xa-1}\nabla d) dx-\frac{1}{\xa-1}\int_{\xO_{\gb}} (d^{\xa-1} (\xD d)w_n^2dx.$$
Now since $|\xD d(x)|<c_5,\;\forall x\in{\xO_{\gb_0}},$ we have
$$\left|\frac{1}{\xa-1}\int_{\xO_{\gb}} d^{\xa-1} (\xD d)w_n^2dx\right|\leq \frac{c_5\gb_{0}^{\xa-1}}{\xa-1}\int_{\xO_{\gb}}w_n^2dx.$$
Also
\begin{align}\nonumber
\left|\int_{\xO_\xb} w_n^2\mathrm{div} (d^{\xa-1}\nabla d)dx\right|&=2\left|\int_{\xO_\xb} w_n d^{\frac{\xa}{2}}d^{\frac{\xa}{2}-1}\nabla d.\nabla w_n dx\right|\\ \nonumber
&\leq c_6\int_{\xO_\xb} d^{\xa}|\nabla w_n|^2dx +\xd\int_{\xO_\xb} d^{\xa-2}w_n^2 dx.
\end{align}
where $c_6=c_6(\gd)>0$, and the result follows in this case, if we choose $\xd$ small enough and then we send $n$ at infinity.\smallskip

\noindent (ii) By the same calculations we have
$$\int_\xO d^{-\frac{p}{2}}|w_n|^pdx\leq c_7\int_{\xO_\xb} d^{\frac{p}{2}}|\nabla w_n|^pdx\leq
c_7\left(\int_\xO d(x)dx\right)^{\frac{p}{2}}
\int_{\xO_\xb} d|\nabla w_n|^2dx.$$
\end{proof}

In the following statement we prove regularity up to the boundary for the function $\frac{u}{\gf_{\xk }}$.

\begin{prop}\label{exi}
Let $f_0\in L^2(\xO).$ Then there exists a unique $u\in H^1_{loc}(\xO)$ such that $\xf^{-1}_{\xk }u\in H^1(\xO,d^{\xa_+}(x)dx)$,  satisfying (\ref{ELin1}).
Furthermore,  if $f_0\in L^q(\xO),\;q>\frac{N+\xa}{2}$, then there exists $0<\xb<1$ such that
\begin{equation}\label{ELin3}\sup_{x,y\in\xO,\;x\neq y}|x-y|^{-\xb}\left|\frac{u(x)}{\gf_{\xk }(x)}-\frac{u(y)}{\gf_{\xk }(y)}\right|<
c_8||f_0||_{L^q}.
\end{equation}
\end{prop}  

\begin{proof}
If there exists a solution $u$, then $\psi=\frac{u}{\xf_{\xk }}$ satisfies
\be
-\xf^{-2}_{\xk }div(\xf^2_{\xk }\nabla \psi)+\xl_{\xk }\psi=\xf^{-1}_{\xk }f_0.\label{doom}
\ee
and we recall that  $\ei(x)\approx d^\frac{\xa_+}{2}(x)$. We endow the space $H^1(\xO,\xf^2_{\xk }dx)$ with the inner product
$$\langle a,b\rangle=\myint{\Gw}{}(\nabla a.\nabla b+\xl_{\xk }ab)\;\xf^2_{\xk }dx.
$$
By a solution $\psi$ of (\ref{doom}) we mean that $\psi\in H_0^1(\xO,\xf^2_{\xk } dx)$ satisfies
\begin{equation}\label{ELin4}
\langle \nabla \psi,\nabla\xz\rangle=\int_\xO\nabla \psi.\nabla\xz \,\ei^2dx+\xl_{\xk }\int_\xO \psi\xz \ei^2 dx=\int_\xO f_0\xz \ei dx\qquad\forall\xz\in H_0^1(\xO,\xf^2_{\xk } dx).
\end{equation}
By Riesz's representation theorem we derive the existence and uniqueness of the solution in this space. Since $H^1(\xO,\xf^2_{\xk }dx)=H_0^1(\xO,\xf^2_{\xk }dx)$ by \cite[Th 2.11]{F.M.T2}, any weak solution $u$ of (\ref{ELin1}) such that $\xf^{-1}_{\xk }u\in H^1(\xO,\xf^2_{\xk }dx)$ is obtained by the above method.
\smallskip

Finally if $f_0\in L^q(\xO),$ where $q>\frac{N+\xa}{2},$ we can apply Moser's approach thanks to (\ref{Sobolev}) and prove first the estimate
\begin{equation}\label{ELin5'}
||\psi||_{L^\infty(\xO)}\leq c_8||f_0||_{L^q(\xO}
\end{equation}
where $c_8=c_8(\Gw,\xk ,q)$, and then to derive the H\"{o}lder regularity up to the boundary
(see e.g. \cite{F.M.T2}).
\end{proof}\medskip

In the next results we make more precise the rate of convergence of a solution of (\ref{ELin1}) to its boundary value.

\begin{prop}\label{reglemma}
Assume $\xk <\frac{1}{4}$. If $f_0\in L^2(\xO)$ and $h\in H^1(\xO)$
there exists a unique weak solution $u$  of (\ref{ELin1}) in $\in H^1_{loc}(\xO)$ and such that
$d^{-\frac{\xa_+}{2}}(u-d^{\frac{\xa_-}{2}}h)\in H^1(\xO,\;d^{\xa_+}(x)dx).$
Furthermore, if $f_0\in L^q(\xO),\;q>\frac{n+\xa}{2}$ and $h\in C^2(\overline{\xO}),$ then there exists $0<\xb<1$ such that
$$\lim_{x\in\xO,\;x\rightarrow y\in\partial\xO}\frac{u(x)}{(d(x))^{\frac{\xa_-}{2}}}=h(y)\qquad\forall y\in\partial\xO,$$
uniformly with respect to $y$,
$$\left|\!\left|\frac{u}{d^{\frac{\xa_-}{2}}}\right|\!\right|_{L^\infty(\xO)}\leq c_9\left(||h||_{C^2(\overline{\xO})}+||f_0||_{L^q(\xO)}\right),$$
and
\be
\sup_{x,y\in\xO,\;x\neq y}|x-y|^{-\xb}\left|\frac{u(x)}{(d(x))^{\frac{\xa_-}{2}}}-\frac{u(y)}{(d(y))^{\frac{\xa_-}{2}}}\right|<c_{10}.\label{regularity}
\ee
with $c_9$ and $c_{10}$ depending on $\xO$, N, q, and $\xk $.
\end{prop}

\noindent{\bf Remark}. By Lemma \ref{space} we already know that $u-d^{\frac{\xa_-}{2}}h\in H_0^1(\xO).$
\begin{proof}
Let $\xb\leq \gb_0$ and  $\eta\in C^2(\xO)$ be a function such that
$\eta=d^{\frac{\xa_-}{2}}(x)$ in $\Gw_\gb$ and $\eta(x)>c>0,$ if $x\in \Gw'_\gb$.
We set $u=\xf_{\xk }v+\eta h.$ Then $v$ is a weak solution of
\be
-\frac{\mathrm{div}(\xf^2_{\xk }\nabla v)}{\ei^2}+\xl_{\xk } v=\frac{1}{\ei}\left(f_0+(\xD\eta+\xk \frac{\eta}{d^2})h+2\nabla\eta.\nabla h+\eta\xD h\right),
\ee\label{ELin6}
in the sense that
\begin{align}\nonumber
\int_\xO\nabla v.\nabla\psi \,\ei^2 dx+\xl_{\xk }\int_{\xO} v\;\psi\,\ei^2 dx&=\int_\xO\left(f_0+(\xD\eta+\xk \frac{\eta}{d^2})h+2\nabla\eta.\nabla h\right)\psi\;\ei dx\\
&\qquad-\int_\xO\nabla h.\nabla\left(\eta\psi\,\ei\right) dx\qquad\forall \psi\in C_0^\infty(\xO).\label{weak}
\end{align}
Let $\psi\in C_0^\infty(\xO_\xb)$. By  an argument similar to the one in the proof of Lemma \ref{space} we have
$$\int_\xO\psi^2dx=\int_{\xO_\gb}\psi^2dx=\int_{\xO_\xb} \mathrm{div} (d\nabla d)|\psi|^2dx-\int_{\xO_\xb} d \xD d|\psi|^2dx,$$
which implies
\be
\int_{\xO_\gb}\psi^2dx\leq c'_{10}\int_{\xO_\gb} d^2|\nabla\psi|^2dx\leq c_{11}\int_{\xO_\gb} d^{\xa_+}|\nabla\psi|^2dx.\label{test1}
\ee
Now
$$\left|\int_{\xO_\gb}\left((\xD\eta+\xk \frac{\eta}{d^2})h+2\nabla\eta.\nabla h\right)\psi\,\ei dx\right|\leq c_{12}\int_{\xO_\gb}\psi^2dx,$$
and
$$\left|\int_{\xO_\gb}\nabla h.\nabla\left(\eta\psi\,\ei\right) dx\right|\leq c_{13}\left(\int_{\xO_\gb}|\nabla h|^2dx+\int\int_{\xO_\gb} d^{\xa_+}|\nabla \psi|^2dx+\int\int_{\xO_\gb}\psi^2dx\right).$$
By (\ref{test1}) we can take $\psi\in H^1(\xO,d^{\xa_+}(x)dx)$ for test function. Thus we can easily obtain that there exists a weak solution $v\in H^1(\xO,d^{\xa_+}(x)dx)$ of (\ref{weak}).

To prove (\ref{regularity}) we first obtain that if $\psi\in C_0^\infty(\xO_\xe)$
$$\int_\xO\psi dx=-\int_{\xO_\xe} d\nabla d.\nabla\psi dx-\int_{\xO_\xe} d \xD d\psi dx,$$
and since
\begin{align}\nonumber
\left|\int_\xO\left((\xD\eta+\xk \frac{\eta}{d^2})h+2\nabla\eta.\nabla h+\eta\xD h\right)\psi\,\ei dx\right|&\leq c_{14}||h||_{C^2(\overline{\xO})}\int_\xO|\psi|dx\\ \nonumber
&\leq \frac{1}{2} \int_{\xO_\xe} d^{\xa_+}|\nabla\psi|^2 dx+ c_{15}(\xO,\xk)||h||_{C^2(\overline{\xO})}.
\end{align}
Using again (\ref{Sobolev}) and Moser's iterative scheeme as in Proposition \ref{exi}, we obtain
$$||v||_{L^\infty(\xO)}\leq c_9\left(||h||_{C^2(\overline{\xO})}+||f_0||_{L^q(\xO)}\right),$$
where $c_9=c_9(\xO, q,\xk )>0$, from which it follows again that $v$ is H\"{o}lder continuous up to the boundary and the uniform convergence holds.
\end{proof}
\begin{prop}\label{reglemma1}
Assume $\xk =\frac{1}{4}$. If $f_0\in L^2(\xO)$ and $h\in H^1(\xO),$
there exists a unique function $u$ in  $H^1_{loc}(\xO)$ weak solution of
$$\CL_{\frac{1}{4}}u=f_0$$
verifying $d^{-\frac{1}{2}}(u-d^{\frac{1}{2}}|\log d|h)\in H^1(\xO,\;d(x)dx).$
Furthermore, if $f_0\in L^q(\xO),\;q>\frac{n+1}{2}$ and $h\in C^2(\overline{\xO}),$ then there exists $0<\xb<1$ such that
$$\lim_{x\in\xO,\;x\rightarrow y\in\partial\xO}\frac{u}{d^{\frac{1}{2}}|\log d|}(x)=h(y)\qquad\forall y\in\partial\xO,$$
uniformly with respect to $y$,
$$\left|\!\left|\frac{u}{\sqrt d\,|\log \frac{d}{D_0}|}\right|\!\right|_{L^\infty(\xO)}\leq c_{16}\left(||h||_{C^2(\overline{\xO})}+||f_0||_{L^q(\xO)}\right)$$
where $D_0=2\sup_{x\in\xO}d(x).$ Finally there holds

\be\label{regularity2}
\sup_{x,y\in\xO,\;x\neq y}|x-y|^{-\xb}\left|\frac{u(x)}{\sqrt{d(x)}|\log \frac{d(x)}{D_0}|}-\frac{u(y)}{\sqrt{d(y)}|\log \frac{d(y)}{D_0}|}\right|<c_{17}\left(||h||_{C^2(\overline{\xO})}+||f_0||_{L^q(\xO)}\right).
\ee
\end{prop}

\begin{proof} Using again Lemma \ref{space}, we know that $u-d^{\frac{1}{2}}|\log d|h\in W_0^{1,p}(\xO),\;\;\forall p<2.$ The proof is very similar to the proof of Proposition \ref{reglemma}. The only differences are we impose $\eta=d^{\frac{1}{2}}|\log d|$ in $\Gw_\gb$ and we use the fact that $|\log d|\in L^p(\xO),\forall p\geq1.$
\end{proof}

In the next result we prove that the boundary Harnack inequality holds, provided the vanishing property of a solution is understood in a an appropriate way.
\begin{prop}\label{maincomp1}
Let $\xd>0$ be small enough, $\xi\in\partial\xO$ and $u\in H^1_{loc}(B_\xd(\xi)\cap\xO)\cap C(B_{\xd}(\xi)\cap\overline{\xO})$ be a positive $\CL_{\frac{1}{4}}$-harmonic function in $B_\xd(\xi)\cap\xO$
vanishing on $\partial\xO\cap B_\xd(\xi)$ in the sense
that
\be\label{zero}\lim_{\dist(x,K)\to 0}\frac{u(x)}{d^\frac{1}{2}(x)|\log d(x)|}=0\qquad\forall K\subset\prt\Gw\cap B_\gd(\xi)\,,\, K\text{ compact}.\ee
Then there exists a constant $c_{18}=c_{18}(N,\xO,\xk)>0$ such that
$$\frac{u(x)}{\xf_{\frac{1}{4}}(x)}\leq c_{18}\frac{u(y)}{\xf_\frac{1}{4}(y)}\qquad\forall x,y\in{\xO}\cap B_\frac{\xd}{2}(\xi).$$
\end{prop}
\begin{proof}
We already know that $u\in C^2(\xO)$. Let $\gd\leq\min(\gb_0,\frac{1}{2})$ such that $B_\xd(\xi)\cap\xO\subset\xO_\xd\subset\xO_{\xb_0}.$\smallskip

By \cite[Lemma 2.8]{1} there exists a  positive supersolution $\xz\in C^2(\xO_\xd)$ of (\ref{IE3}) in $\Gw_\gd$ with the following behaviour
$$\xz(x)\approx d^\frac{1}{2}(x)\log\frac{1}{d(x)}\left(1+\left(\log\frac{1}{ d(x)}\right)^{-\xb}\right),$$
for some $\xb\in (0,1)$
and $c_{19}=c_{19}(\xO)>0$. Set $v=\xz^{-1}u$, then it satisfies
\be
-\xz^{-2}\mathrm{div} (\xz^2\nabla v)\leq0\quad\hbox{in}\;\;B_\xd(\xi)\cap\xO.\label{doom1}
\ee
Let $\eta\in C_0^\infty(B_\xd(\xi))$ such that $0\leq\eta\leq1$ and $\eta=1$ in $B_{\frac{3\xd}{4}}(\xi).$ We set
$v_s=\eta^2 (v-s)_+$
Since by assumption $v_s$ has compact support in $B_\gd(\xi)\cap\Gw$,
 we can use it as a test function in (\ref{doom1}) and we get
 \be
\int_{B_\xd(\xi)\cap\xO}\xz^2\nabla v.\nabla v_sdx=\int_{B_\xd(\xi)\cap\xO}\xz^2\nabla (v-s)_+.\nabla v_s dx\leq0,\label{doom2}
\ee
which yields
$$\int_{B_\xd(\xi)\cap\xO}|\nabla (v-s)_+|^2 \xz^2\eta^2dx\leq4\int_{B_\xd(\xi)\cap\xO}|\nabla\eta|^2(v-s)_+^2\xz^2dx.$$
Letting $s\to 0$ we derive
$$\int_{B_\xd(\xi)\cap\xO}|\nabla v|^2 \xz^2\eta^2dx\leq4\int_{B_\xd(\xi)\cap\xO}|\nabla\eta|^2v^2\xz^2dx.$$
Since
$$|\nabla (v-s)_+|^2 \xz^2\eta^2\uparrow |\nabla v|^2 \xz^2\eta^2\quad\text{as } s\to 0,
$$
and convergence of $\nabla (v-s)_+$ to $\nabla v$ holds a.e. in $\Gw$,  it follows by the monotone convergence theorem
\be\label{cov}
\lim_{s\rightarrow0}\int_{B_\xd(\xi)\cap\xO}|\nabla (v-(v-s)_+)|^2\xz^2\eta^2dx= 0,
\ee
and finally $\xz v_s\to\eta^2 \xz v$ in $H^1(B_\xd(\xi)\cap\xO)$, which yields in particular $\eta^2u=\eta^2\xz v\in H_0^1(B_\xd(\xi)\cap\xO)$.

\smallskip

\noindent{\it Step 2.} By \cite[Lemma 2.8]{1} there exists a positive subsolution $h\in C^2(\xO_\xd)$ of (\ref{IE3}) in $\Gw_\gd$ with the following behaviour
$$h(x)\approx d^\frac{1}{2}(x)\log\frac{1}{d(x)}\left(1-\left(\log\frac{1}{d(x)}\right)^{-\xb}\right),$$
where $\xb\in (0,1)$
and $c_{20}=c_{20}(\xO)>0$. Set $w=h^{-1}u$ and $w_s=\eta^2(w-s)_+$. Then $w_s\to \eta^2w$ in $H^1(B_\xd(\xi)\cap\xO)$ by Step 1. Put $u_s=hw_s$, thus, for $0<s,s'$, we have
\begin{align}
&\int_{B_\xd(\xi)\cap\xO}|\nabla (u_s-u_{s'})|^2dx-\frac{1}{4}\int_{B_\xd(\xi)}\frac{|u_s-u_{s'}|^2}{d^2(x)}dx =\int_{B_\xd(\xi)\cap\xO}h^2|\nabla (w_s-w_{s'})|^2dx\\ \nonumber
&+
\int_{B_\xd(\xi)\cap\xO}|\nabla h|^2|w_s-w_{s'}|^2dx
+\int_{B_\xd(\xi)\cap\xO}h\nabla h.\nabla (u_s-u_{s'})^2dx-\frac{1}{4}\int_{B_\xd(\xi)\cap\xO}\frac{h^2|w_s-w_{s'}|^2}{d^2(x)}dx\\ \nonumber
&\phantom{\int_{B_\xd(\xi)\cap\xO}|\nabla (u_s-u_{s'})|^2dx-\frac{1}{4}\int_{B_\xd(\xi)}\frac{|u_s-u_{s'}|^2}{d^2(x)}dx}\leq \int_{B_\xd(\xi)\cap\xO}h^2|\nabla (w_s-w_{s'})|^2dx,
\end{align}
where, in the last inequality, we have performed by parts integration  and then used the fact that $h$ is a subsolution.
Thus we have by (\ref{cov}) that
\be
\lim_{s,s'\rightarrow0}\int_{B_\xd(\xi)}|\nabla (u_s-u_{s'})|^2dx-\frac{1}{4}\int_{B_\xd(\xi)}\frac{|u_s-u_{s'}|^2}{d^2(x)}dx=0.\label{caushy}
\ee
Let $\mathbf{W}(\Gw)$ denote the closure of $C_0^\infty(\xO)$ in the space of functions $\gf$ satisfying
$$||\gf||_{H}^2:=\int_\xO|\nabla \xF|^2dx-\frac{1}{4}\int_\xO\frac{|\xF|^2}{d^2(x)}dx<\infty.$$
Thus $\eta^2 u\in\mathbf{W}(\Gw)$,
which implies
$$\frac{\eta u}{\xf_{\frac{1}{4}}}\in H_0^1(\xO,\;d(x)dx).$$
Next we set $\tilde v=\xf^{-1}_{\frac{1}{4}}u$; then $\tilde v \in H^1(B_\frac{3\xd}{4}(\xi), d(x)dx)$ and it satisfies
$$-\xf_{\frac{1}{4}}^{-2}\mathrm{div}(\xf^2_{\frac{1}{4}}\nabla \tilde v)+\xl_{\frac{1}{4}} \tilde v=0.$$
By the same approach based on  Moser' iterative scheeme applied to degenerate elliptic operators as the in \cite[Theorem 1.5]{F.M.T2}, we see that $v$ satisfies a Harnack inequality up to the boundary of $\xO$. More precisely there exists a constant $c_{18}=c_{18}(\Gw)>0$ such that
$$v(x)\leq c_{18} v(y)\qquad\forall x,y\in B_{\frac{\xd}{2}}(\xi).$$
And the result follows.
\end{proof}

\medskip
In the case $\xk <\frac{1}{4}$, the boundary Harnack inequality is the following,
\begin{prop}\label{maincomp}
Let $\xd>0$ be small enough $\xi\in\xO,$ $\xk <\frac{1}{4}$ and $u\in H^1_{loc}(B_\xd(\xi)\cap\xO)\cap C(B_{\xd}(\xi)\cap\overline{\xO})$ be a nonnegative $\CL_{\xk }$-harmonic in $B_\xd(\xi)$
vanishing on $\partial\xO\cap B_\xd(\xi)$ in the sense
that
\be\label{zero1}\lim_{\dist(x,K)\to 0}\frac{u(x)}{(d(x))^{\frac{\xa_-}{2}}}=0\qquad\forall K\subset\prt\Gw\cap B_\gd(\xi)\,,\, K\text{ compact}.\ee
Then there exists $c_{21}=c_{21}(\xO,\xk )>0$ such that
$$\frac{u(x)}{\ei(x)}\leq c_{21}\frac{u(y)}{\ei(y)}\qquad\forall x,y\in\overline{\xO}\cap B_\frac{\xd}{2}(\xi).$$
\end{prop}
\begin{proof}
The only difference with the preceding proof is that we take as subsolution  and supersolution (see  \cite[Lemma 2.8]{1})
$C^2(\xO)$ the functions $h$ and $\zeta$ respectively with the boundary behaviour
$$h\approx d^{\xa_-}(1-d^\xb)\qquad \zeta\approx d^{\xa_-}(1+d^\xb),$$
where $\xb\in(0,\sqrt{1-4 \xk }\,).$
\end{proof}


\begin{prop}\label{max1}
Let  $u\in H^1_{loc}(\xO)\cap C(\xO)$ be a $\CL_{\frac{1}{4}}$-subharmonic function
such that
 $$\limsup_{d(x)\to 0}\frac{u(x)}{d^\frac{1}{2}(x)|\log d(x)|}\leq0.$$
Then $u\leq0.$
\end{prop}  
\begin{proof}
We set $v=\max(u,0)$ and we proceed as in Step 1 of the proof  of Proposition \ref{maincomp1} with $\eta=1$. The result follows by letting $s\to 0$.
\end{proof}
Similarly we have
\begin{prop}\label{max}
Let  $u\in H^1_{loc}(\xO)\cap C(\xO)$ be a $\CL_{\xk }$-subharmonic function
such that
$$\limsup_{d(x)\to 0}\frac{u(x)}{(d(x))^{\frac{\xa_-}{2}}}\leq0.$$
then $u\leq0.$
\end{prop}  

The two next statements shows that comparison holds provided comparable boundary data are achieved in way which takes into
account the specific form of the $\CL_{\xk }$-harmonic functions
\begin{prop}\label{comp}
Assume $\xk <\frac{1}{4}$ and $h_i\in H^1(\xO)$ (i=1,2).  Let $u_i\in H^1_{loc}(\xO)$ be two  $\CL_{\xk }$-harmonic functions
such that $d^{-\frac{\xa_+}{2}}\left(u_i-d^{\frac{\xa_-}{2}}h_i\right)\in H^1(\xO,\;d^{\xa_+}(x)dx).$ Then\smallskip

\noindent If $h_1\leq h_2$ a.e. in $\xO$, there holds
$$u_1(x)\leq u_2(x)\qquad \forall x\in\xO.$$
If $h_1-h_2\in H_0^1(\xO),$  there holds
$$u_1(x)=u_2(x)\qquad \forall x\in\xO.$$
\end{prop}  
\begin{proof}
Set $w=\gf_{\xk }^{-1}(u_1-u_2)$, then $w\in H^1(\xO,\ei^2dx)$ and
$$-\mathrm{div}(\xf^2_{\xk }\nabla w)+\xl_{\xk }\ei^{2} w= 0 $$
Since $H^1(\xO,\ei^2dx)=H_0^1(\xO,\ei^2dx)$ by (\ref{test}) we derive that $w$  and $w$ belongs to $H_0^1(\xO,\ei^2dx)$ and, integrating by part, we derive $w_+=0$.
The proof of the second statement is similar.
\end{proof}
In the same way we have in the case $\xk =\frac{1}{4}.$

\begin{prop}\label{comp1}
Assume $\xk =\frac{1}{4}.$ Let $h_i\in H^1(\xO)$ (i=1,2) and let $u_i\in H^1_{loc}(\xO)$ be  two $\CL_{\frac{1}{4}}$-harmonic functions  such that $d^{-\frac{1}{2}}(u_i-d^{\frac{1}{2}}|\log d|h_i)\in H^1(\xO,\;d(x)dx).$\smallskip

\noindent (i) If $h_1\leq h_2$ a.e. in $\xO$, then
$$u_1(x)\leq u_2(x)\qquad \forall x\in\xO.$$\smallskip

\noindent (ii) If $h_1-h_2\in H_0^1(\xO)$, then
$$u_1(x)=u_2(x)\qquad \forall x\in\xO.$$
\end{prop}  

We end with existence and uniqueness results for solving the Dirichlet problem associated to $\CL_{\xk }$.

\begin{prop} \label{mainlemma1}Assume $\xk =\frac{1}{4}.$
For any $h\in C(\partial\xO)$
there exists a unique $\CL_{\frac{1}{4}}$-harmonic function $u$ belonging to $H^1_{loc}(\xO)$  satisfying
$$\lim_{x\in\xO,\;x\rightarrow y\in\partial\xO}\frac{u(x)}{d^{\frac{1}{2}}(x)|\log d(x)|}=h(y)\qquad\text{uniformly for } y\in\partial\xO.$$
Furthermore  there exists a constant $c_{16}=c_{16}(\xO)>0>0$
$$\left|\!\left|\frac{u}{d^{\frac{1}{2}}|\log \frac{d}{D_0}|}\right|\!\right|_{L^\infty(\xO)}\leq c_{24}||h||_{C(\partial\xO)},$$
where $D_0=2\sup_{x\in\xO}d(x).$
\end{prop}  
\begin{proof}
\noindent Uniqueness is a consequence of Proposition \ref{max1}.
For existence let $m\in \mathbb{N}$ and $h_n$ be smooth functions such that $h_m\rightarrow h$ in $L^\infty(\partial\xO).$ Then we can find a function $H_m\in C^2(\overline{\xO})$
with trace $h_m$ on $\partial\xO,$ and $||H_m||_{C^2(\overline{\xO})}\leq c||h_m||_{L^\infty(\partial\xO)},$ where $c$ depends on $\xO.$ By Lemma \ref{reglemma1} there exists a unique weak solution $u_m$ of $\CL_{\frac{1}{4}}u=0$ satisfying
$$\lim_{x\in\xO,\;x\rightarrow y\in\partial\xO}\frac{u_m}{d^{\frac{1}{2}}|\log d|}(x)=h_m(y)\qquad\text{uniformly for } y\in\partial\xO.$$
By Proposition \ref{reglemma1} we have
$$\left|\!\left|\frac{u_m-u_n}{d^{\frac{1}{2}}|\log \frac{d}{D_0}|}\right|\!\right|_{L^\infty(\xO)}\leq c_{16}||h_m-h_n||_{C(\partial\xO)}.$$
Thus  there exists  $u$ such that $$\lim_{m\rightarrow\infty}\left|\!\left|\frac{u_m-u}{d^{\frac{1}{2}}|\log \frac{d}{D_0}|}\right|\!\right|_{L^\infty(\xO)}=0$$ and $u$ is a solution of $L_\frac{1}{4}u=0.$

Let $x\in \xO,$ with $d(x)<\frac{1}{2}$ and $y\in\partial\xO$
\begin{align}\nonumber
\left|\frac{u}{d^{\frac{1}{2}}|\log d|}(x)-h(y)\right|\leq\left|\frac{u}{d^{\frac{1}{2}}|\log d|}(x)-\frac{u_m}{d^{\frac{1}{2}}|\log d|}(x)\right|&+\left|\frac{u_m}{d^{\frac{1}{2}}|\log d|}(x)-h_m(y)\right|\\ \nonumber
&+|h(y)-h_m(y)|.
\end{align}
The result follows by letting successively $x\to y$ and $m\to\infty$.
\end{proof}

Similarly we have
\begin{prop}\label{mainlemma2}
Assume $\xk <\frac{1}{4}$. Then for any $h\in C(\partial\xO)$
there exists a unique $\CL_{\xk }$-harmonic function $u\in H^1_{loc}(\xO)$  satisfying
$$\lim_{x\in\xO,\;x\rightarrow y\in\partial\xO}\frac{u}{d^{\frac{\xa_-}{2}}}(x)=h(y)\qquad\text{uniformly for } y\in\partial\xO.$$
Furthermore there exists a constant $c_9=c_9(\xO,\xa)>0$ such that
$$\left|\!\left|\frac{u}{d^{\xa_-}}\right|\!\right|_{L^\infty(\xO)}\leq c_9||h||_{C(\partial\xO)}.$$
\end{prop}

A useful consequence of \cite[Lemma 2.8]{1} and Propositions \ref{reglemma} and  \ref{reglemma1} is the following local existence result.
\begin{prop}\label{exist} There exists  a positive $\CL_{\xk }$-harmonic  function $Z_{\xk }\in C(\overline{\Gw_{\gb_0}})\cap C^2(\Gw_{\gb_0})$ satisfying
\be\label{zero2}
\lim_{d(x)\to 0} \frac{Z_{\frac{1}{4}}(x)}{\sqrt{d(x)}|\ln d(x)|}=0\ee
if $\xk =\frac{1}{4}$, and
\be\label{zero3}
\lim_{d(x)\to 0} \frac{Z_{\xk }(x)}{(d(x))^{\frac{\ga_-}{2}}}=0
\ee
if $0<\xk <\frac{1}{4}$.
\end{prop}


\subsection{$\CL_{\xk }$-harmonic measure}

Let $x_0\in\xO,$ $h\in C(\partial\xO)$ and denote $L_{\xk ,x}(h):=v_h(x_0)$ where $v_h$ is the solution of the Dirichlet problem (see Propositions \ref{mainlemma1} and \ref{mainlemma2})
\begin{align}\nonumber
\CL_{\xk }v&=0\qquad\mathrm{in}\;\;\xO\\ \label{linear}
v&=h\qquad\mathrm{in}\;\;\partial\xO
\end{align}
where $v$ take the boundary data in the sense of Lemmas \ref{mainlemma1} and \ref{mainlemma2}. By Lemma's \ref{max} and \ref{max1}, the mapping $h\mapsto L_{\xk ,x_0}(h)$ is a linear positive functional on $C(\partial\xO).$ Thus there exists a unique Borel measure on $\partial\xO,$ called {\it $\CL_{\xk }$-harmonic measure} in $\xO,$ denoted by $\xo^{x_0}$, such that
$$v_{h}(x_0)=\int_{\partial\xO}h(y) d\xo^{x_0}(y).$$
Because of Harnack inequality the measures $\xo^x$ and $\xo^{x_0},$ $x_0,\;x\in \xO$ are mutually absolutely continuous. For every fixed $x$ we denote the Radon-Nikodyn derivative by
$$K(x,y):=\frac{dw^x}{dw^{x_0}}(y)\qquad\mathrm{for}\;\xo^{x_0}\text{- almost all }y\in\partial\xO.$$

It is wellknown that the following formula is an equivalent definition of the $\CL_{\xk }$-harmonic measure:
for any closed set $E\subset\partial\xO$
$$\xo^{x_0}(E)=\inf\left\{\psi:\;\psi\in C_+(\xO)\,,\;\CL_{\xk }\text{-superhamornic in } \xO\;\text{ s.t. }\liminf_{x\rightarrow E}\frac{\psi(x)}{W(x)}\geq1\right\},$$
where
$$W(x)=\Bigg\{\begin{array}{lll}&d^{\frac{\xa_-}{2}}(x)\qquad&\text{if}\;\xk <\frac{1}{4},\;\\\\
&d^{\frac{1}{2}}(x)|\log d(x)|\qquad&\text{if}\;\xk =\frac{1}{4}.
\end{array}
$$
The extension to open sets is standard.
Let $\xi\in\partial\xO.$ We set $\xD_r(\xi)=\partial\xO\cap B_r(\xi)$ and
$x_r=x_r(\xi)\in\xO,\;$ such that $d(x_r)=|x_r-\xi|=r.$ Also $x_{r}(\xi)=\xi-r{\bf n}_{\xi}$ where $\bf n_{\xi}$ is the unit outward normal vector to $\prt\Gw$ at $\xi$. We recall that $\xb_0=\xb_0(\xO)>0$ has been defined in Lemma \ref{space}.
\begin{lemma}\label{lem2.1}
There exists a constant $c_{25}>0$ which depends only on $\xO$ and $a$ such that if $0<r\leq \xb_0$ and $\xi\in\prt\Gw$, there holds
\be\label{00}\frac{\xo^x(\xD_r(\xi))}{W(x)}\geq c_{25}\qquad\forall x\in\xO\cap B_{\frac{r}{2}}(\xi).\ee
\end{lemma}  
\begin{proof}
Let $h\in C(\partial\xO)$ be a function with compact support in $\xD_r(\xi),$ $0\leq h\leq1$ and $h=1$ on $\overline{\xD_\frac{3r}{4}(\xi)}.$ And let $v_h,v_1$ the corresponding  $\CL_{\xk }$-harmonic functions with $h$ and $1$ as boundary data respectively (in the sense of Lemmas \ref{mainlemma1} and \ref{mainlemma2}).
Then $v_1(x)\geq v_h(x)\geq0$ and $$\lim_{x\in\xO,\;x\rightarrow x_0}\frac{v_1(x)-v_h(x)}{W(x)}=0\qquad\forall x_0\in\xO\cap B_{\frac{3r}{4}}(\xi).$$
By Lemmas \ref{maincomp} and \ref{maincomp1}, and $\xf_{\xk }\approx d^{\frac{\xa_+}{2}},$ there exists $c_{26}=c_{26}(\xO)>0$ such that
 $$\frac{v_1(x)-v_h(x)}{d^{\frac{\xa_+}{2}}(x)}\leq c_{26}\frac{v_1(y)-v_h(y)}{d^{\frac{\xa_+}{2}}(y)},\qquad\forall x,y\in \xO\cap \overline{B_{\frac{r}{2}}}(\xi).$$
We consider first the case $\xk =\frac{1}{4}$. By Proposition \ref{reglemma1},
we have
$$0\leq\frac{v_1(y)-v_h(y)}{d^{\frac{1}{2}}(y)}\leq\frac{v_1(y)}{d^{\frac{1}{2}}(y)}\leq c_{24}|\log d(y) |.$$
Thus, combining all above we have that
$$\frac{v_1(x)}{d^{\frac{1}{2}}(x)|\log d(x)|}-c_{27}\frac{|\log d(y) |}{|\log d(x) |}\leq \frac{v_h(x)}{d^{\frac{1}{2}}(x)|\log d(x)|}.$$
Now by Lemma \ref{reglemma1}, there exists a $\xe_0>0$ such that $$\frac{v_1(x)}{d^{\frac{1}{2}}(x)|\log d(x)|}>\frac{1}{2}\qquad\forall x\in \xO_{\xe_0}.$$
Thus if we choose $y$ such that $d(y)=\frac{r}{4},$ there exists a constant $c_{27}=c_{27}(\xO)>0$ such that
$$c_{27}\frac{|\log d(y) |}{|\log d(x) |}= c_{27}\frac{|\log \frac{r}{4} |}{|\log d(x) |}\leq c_{27}\frac{|\log \frac{r}{4} |}{|\log \frac{r}{D_0} |}
\leq \myfrac{1}{4}\qquad\forall x\in \Gw_{\frac{r}{D_0}},$$
thus
\be
\frac{v_h(x)}{d^{\frac{1}{2}}(x)|\log d(x)|}\geq\frac{1}{4}\qquad\forall x\in \overline{B_{\frac{r}{2}}}(\xi)\cap\Gw_{\frac{r}{D_0}}.\label{111}
\ee
In particular
\be
\frac{v_h(x_{a^*r}(\xi))}{\sqrt{a^*r}|\log (a^*r)|}\geq\frac{1}{4}.\label{112}
\ee
where $a^*=(\max\{2,D_0\})^{-1}$. If $D_0\leq 2$ we obtain the claim. If not, set $k^*=\BBE[\frac{D_0}{2}]+1$ (recall that $\BBE[x]$ denotes the largest integer less or equal to $x$). If $x\in \overline{B_{\frac{r}{2}}}(\xi)\cap\Gw'_{\frac{r}{D_0}}$ there exists a chain of at most $4k^*$ points $\{z_j\}_{j=0}^{j=j_0}$ such that $z_j\in\overline{B_{\frac{r}{2}}}(\xi)\cap\Gw$, $d(z_j)\geq a^*r$, $ z_0=x_{a^*r}(\xi)$, $z_{j_0}=x$ and  $|z_j-z_{j+1}|\leq \frac{a^*r}{4}$. By  Harnack inequality (applied $j_0$-times)
\be\label{113}
v_h(x_{a^*r}(\xi))\leq c_{28}v_h(x).
\ee
Since
$$W(x_{a^*r}(\xi))\geq \left(a^*\right)^{\frac{1}{2}}W(x),
$$
we obtain finally
\be\label{114}
\frac{1}{4}\leq \frac{\gw^{x_{a^*r}(\xi)}(\Gd_r(\xi))}{\sqrt{a^*r}|\log (a^*r)|}\leq c_{28}\left(\frac{1}{a^*}\right)^{\frac{1}{2}}\frac{\gw^x(\Gd_r(\xi))}{W(x)}\qquad
\forall x\in \Gw\cap B_{\frac{r}{2}}(\xi).
\ee
In the  case $\xk <\frac{1}{4}$, the proof is simpler since no log term appears and we omit it.
\end{proof}

The next result is a Carleson type estimate valid for positive $\CL_{\xk }$-harmonic functions.

\begin{lemma}\label{lemharn}
There exists a constant $c_{29}$ which depends only on $\xO$ such that for any $\xi\in\prt\Gw$ and $0<r\leq s\leq \xb_0.$ ,
\be\label{carle}
\frac{\xo^x(\xD_{r}(\xi))}{W(x)}\leq c_{29}\frac{ \xo^{x_{s}(\xi)}(\xD_{r}(\xi))}{W(x_{s}(\xi))}\qquad\forall x\in\xO\setminus B_{s}(\xi).
\ee
\end{lemma}  

\begin{proof}
Let $h\in C(\partial\xO)$ with compact support in $\xD_r(\xi))$ and $0\leq h\leq 1.$ We denote by $v_h,\;v_1,$ the solutions of \eqref{linear} with boundary data $h$ and $1$ respectively.
 By Propositions \ref{mainlemma1} and \ref{mainlemma2} there exists a constant $c_{30}>0$ such that for $0<r<\gb_0$,
\be\label{glob}
\frac{v_h}{W(x)}\leq\frac{\xo^x(\xD_r(\xi))}{W(x)}\leq \frac{\xo^x(\prt\Gw)}{W(x)}\leq c_{30}\qquad\forall x\in \xO.\ee
By Propositions \ref{mainlemma1} and \ref{mainlemma2}, there holds
\be\label{glob00}
\lim_{d(x)\to 0}\frac{v_1(x)}{W(x)}=1,
\ee
thus we can replace $W$ by $v_1$ in (\ref{carle}). Since $w_h=\frac{v_h(x)}{v_1(x)}$ is H\"older continuous in $\overline\Gw$ and  satisfies
\be\label{glob01}\BA{ll}
-div (v_1^2\nabla w_h)=0\qquad &\text{in }\Gw\setminus \overline {B_{s}}(\xi)\\
\phantom{ (;\nabla w)}
0\leq w_h\leq 1&\text{in }\Gw\setminus \overline {B_{s}}(\xi)\\\phantom{-div (v_1^2\nabla )}
w_h=0\qquad \qquad &\text{in }\prt\Gw\setminus \overline {B_{s}}(\xi)
\EA\ee
the maximum of $w_h$ is achieved on $\Gw\cap\prt B_{s}(\xi)$, therefore it is sufficient to prove the Carleson estimate
\be\label{glob02}
w_h(x)\leq c_{29}w_h(x_{s}(\xi))\qquad\forall x\in\xO\cap\prt B_{s}(\xi).
\ee
If $x$ such that $|x-\xi|=s$ is "far" from  $\prt\Gw$, $w_h(x)$ is "controled" by $w_h(x_{s}(\xi))$ thanks to Harnack inequality, while
if it is close to $\prt\Gw$, $w_h(x)$ is "controled by the fact that it vanishes on $\prt\xO\cap\prt B_{s}(\xi)$. \smallskip

We also note that (\ref{00}) can be written under the form
\be\label{00'}
w_h(x)\geq c_{25}\qquad\forall x\in\xO\cap B_{\frac{r}{2}}(\xi).
\ee

\noindent {\it Step 1}. : $r\leq s\leq 4r$. By Lemma \ref{lem2.1}, \eqref{glob} and the above inequality we have that
$$w_h(x_\frac{r}{2}(\xi))\geq \frac{c_{25}}{c_{30}}w_h(x) \qquad\forall x\in\Gw.
$$
Applying Harnack inequality to $w_h$ in the balls $B_{\frac{(2+j)r}{4}}(x_{\frac{(2+j)r}{4}}(\xi))$ for $j=0,...,j_0\leq14$ we obtain

$$w_h(x_{\frac{(2+j)r}{4}}(\xi))\geq c^{j}_{31}w_h(x_{\frac{r}{2}}(\xi))\qquad\text{for }j=1,...,j_0.
$$
This implies
\be\label{BHI0}
w_h(x_{s}(\xi))\geq c_{32}w_h(x) \qquad\forall x\in\Gw.
\ee
\noindent {\it Step 2}: $\gb_0\geq s> 4r$. We apply Propositions 2.11, 2.12 to $w_h$ in $B_{\frac{s}{2}}(\xi_1)\cap\Gw$ where
$\xi_1\in\prt\Gw$ is such that $|\xi-\xi_1|=s$ and we get
\be\label{BHI1}
w_h(x)\leq c_{18}w_h(x_{\frac{s}{4}}(\xi_1))\qquad\forall x\in B_{\frac{s}{4}}(\xi_1)\cap\Gw
\ee
Then we apply six times Harnack inequality to $w_h$ between $x_{\frac{s}{4}}(\xi_1)$ and $x_{s}(\xi)$ and obtain
\be\label{BHI2}w_h(x_{\frac{s}{4}}(\xi_1))\leq c_{33}w_h(x_{s}(\xi_1)).\ee
Combining (\ref{BHI1}) and (\ref{BHI2}) we derive (\ref{glob02}).
\smallskip

\noindent {\it Step 3}. For $\ge>0$, set $z_h=w_h-c_{33}w_h(x_{s}(\xi))-\ge$. Then  $z_h^+$ has compact support in $\Gw\setminus B_s(\xi)$ and thus belongs to $H^1_0(\Gw\setminus B_s(\xi))$. Integration by parts in (\ref{glob01}) leads to
\be\label{BHI3}
\myint{\Gw\setminus B_s(\xi)}{}v_1^2|\nabla z_h^+|^2dx=0.
\ee
Then $z_h^+=0$ by letting $\ge\to 0$. Combining with  (\ref{BHI0}) and $h\uparrow \chi_{\Gd_r(\xi)}$ implies (\ref{carle}).
\end{proof}

\begin{theorem}\label{lem2.2*}
There exists a constant $c_{34}$ which depends only on $\xO$ such that, for any $0<r\leq\gb_0$ and $\xi\in\prt\Gw$, there holds
\be\label{Rep0}
\frac{1}{c_{34}}r^{N-1-\frac{1}{2}}|\log r|G_{\CL_{\frac{1}{4}}}(x_r(\xi),x)\leq\xo^x(\xD_r(\xi))\leq c_{34}r^{N-1-\frac{1}{2}}|\log r|G_{\CL_{\frac{1}{4}}}(x_r(\xi),x)\qquad\forall x\in\xO\setminus B_{4r}(\xi).
\ee
\end{theorem}  
\begin{proof}
Let $\eta\in C_0^\infty(B_{2r}(\xi))$ such that $0\leq \eta\leq1$ and $\eta=1$ in $B_{r}(\xi).$ We set
$$u=\eta (-\ln d)\sqrt d:=\eta \psi,$$
(we assume that $4r<1$), in order to have
$$\lim_{x\rightarrow x_0}\frac{u(x)}{\psi(x)}=\eta\lfloor_{\prt\Gw}(x_0)=\gz(x_0)\qquad\forall x_0\in\prt\xO,$$
uniformly with respect to $x_0$. Since
$$-\xD \psi-\frac{1}{4}\frac{\psi}{d^2(x)}=\myfrac{2+\ln d}{2\sqrt d}\Gd d=-(N-1)\myfrac{2+\ln d}{2\sqrt d}K
$$
where $K$ is the mean curvature of $\prt\Gw$. Also we have
$$|\nabla\eta|\leq c_0\chi_{\Gw\cap B_{2r}(\xi)}\frac{1}{r}\quad\text{and}\quad|\xD\eta(x)|\leq c_0\chi_{\Gw\cap B_{2r}(\xi)}\frac{1}{r^2}\leq c_0\chi_{\Gw\cap B_{2r}(\xi)}\frac{1}{r}d^{-1}(x),$$

then $u$ satisfies
\begin{align} \nonumber
-\xD u-\frac{1}{4}\frac{u}{d^2(x)}=-\psi\Gd\eta+\myfrac{2+\ln d}{2\sqrt d}\left(2\nabla d.\nabla\eta-(N-1)K\eta\right):=f\qquad&\text{in }\xO\\ \nonumber
u=\gz\qquad&\text{on }\partial\xO.
\end{align}
Then $|f|\leq \frac{c_{35}}{r}(-\frac{\ln d}{\sqrt d})\chi_{\Gw\cap B_{2r}(\xi)}$ since $\eta$ vanishes outside $B_{2r}(\xi)$. We have by the representation formula \cite{F.M.T2}
\be\label{Rep1}
0=u(x)=\int_\xO G_{\CL_{\frac{1}{4}}}(x,y)fdy+\int_{\partial\xO}h(y)d\xo^x(y)\qquad\forall x\in \xO\setminus B_{2r}(\xi).
\ee
By Lemma \ref{green}, we have that for any $x\in\xO\setminus B_{4r}(\xi)$ and $y\in B_{2r}(\xi)$
$$G_{\CL_{\frac{1}{4}}}(x,y)\leq c_{36}G_{\CL_{\frac{1}{4}}}(x,x_r(\xi)),
$$
thus
\be\label{Rep2}\BA {lll}\displaystyle
\displaystyle\xo^x(\xD_r(\xi))\leq \int_{\xO\cap B_{2r}(\xi)} G_{\CL_{\frac{1}{4}}}(x,y)|f(y)|dy\\ [4mm]
\phantom{\xo^x(\xD_r(\xi))}
\displaystyle\leq \frac{c_{37}}{r}G_{\CL_{\frac{1}{4}}}(x,x_r(\xi))\int_{\xO\cap B_{2r}(\xi)}\frac{|\ln d(y)|}{\sqrt {d(y)}}dy\\ [4mm]
\phantom{\xo^x(\xD_r(\xi))}
\displaystyle\leq c_{38}G_{\CL_{\frac{1}{4}}}(x,x_r(\xi))r^{N-1-\frac{1}{2}}|\ln r|,
\EA\ee
since
$$\int_{\xO\cap B_{2r}(\xi)}\frac{|\ln d(y)|}{\sqrt {d(y)}}dy\leq c_{39}r^{N-1}\int_0^{2r}\frac{|\ln t|dt}{\sqrt {t}}
\leq 2c_{39}r^{N-\frac{1}{2}}|\ln r|.
$$
This implies the right-hand side part of (\ref{Rep0}).
For the opposite inequality we observe that if $x\in\prt B_{4r}(\xi)\cap\xO$, there holds by (\ref{00})
 $$\BA {ll}
 r^{N-1-\frac{1}{2}}|\log r|G_{\CL_\frac{1}{4}}(x_r(\xi),x)\leq c_{40}r^{N-1-\frac{1}{2}}|\log r|\min\left\{\myfrac{1}{|x-x_r(\xi)|^{N-2}},\myfrac{\sqrt{d(x)}\sqrt{d(x_r(\xi))}}{|x-x_r(\xi)|^{N-1}}\right\}\\[4mm]
 \phantom{r^{N-\frac{1}{2}}|\log r|G_{\CL_\frac{1}{4}}(x_r(\xi),x)}
 \leq c_{41}\sqrt{d(x)}|\log r|\\[4mm]
 \phantom{r^{N-\frac{1}{2}}|\log r|G_{\CL_\frac{1}{4}}(x_r(\xi),x)}
 \leq c_{42}W(x)
 \\[4mm]
 \phantom{r^{N-\frac{1}{2}}|\log r|G_{\CL_\frac{1}{4}}(x_r(\xi),x)}
 \leq \frac{c_{42}}{c_{25}}\xo^{x_{\frac{r}{8}}(\xi)}(\Gd_r(\xi)).
 \EA$$
We end the proof by Harnack inequality between $\xo^{x_{\frac{r}{8}}(\xi)}(\Gd_r(\xi))$ and $\xo^{x_{4r}(\xi)}(\Gd_r(\xi))$ and by Harnack inequality between $\xo^{x}(\Gd_r(\xi))$ and $\xo^{x_{4r}(\xi)}(\Gd_r(\xi))$ on $\partial B_{4r}(\xi)$ and an argument like in the step 3 in Lemma \eqref{lemharn}.
\end{proof}\medskip

Replacing, in the last proof, the function $\psi=\sqrt d(-\ln d)$ by  $\tilde\psi=d^{\frac{\ga_-}{2}}$, we obtain similarly.y
\begin{theorem}\label{lem2.2}
Assume $\xk <\frac{1}{4}$. There exists a constant $c_{42}$ which depends only on $\xO$ and $\xk $ such that, for any $0<r\leq\gb_0$ and $\xi\in\prt\Gw$, there holds
$$\frac{1}{c_{42}}r^{N-2+\frac{\xa_-}{2}}G_{\CL_{\xk }}(x_r(\xi),x)\leq\xo^x(\xD_r(\xi))\leq c_{42}r^{N-2+\frac{\xa_-}{2}}G_{\CL_{\xk }}(x_r(\xi),x)\qquad\forall x\in\xO\setminus B_{4r}(\xi).$$
\end{theorem}

As a consequence of Theorems \ref{lem2.2*} and \ref{lem2.2} and the Harnack inequality, the harmonic measure for $\CL_{\xk }$ possesses the doubling property.
\begin{theorem}\label{Th10}
Let $0<\xk \leq\frac{1}{4}$. There exists a constant $c_{42}$ which depends only on $\xO,\xk $ such that for any $0<r\leq \gb_0$, there holds
$$\xo^x(\xD_{2r}(\xi))\leq c_{42}\xo^x(\xD_{r}(\xi))\qquad\forall x\in\xO\setminus B_{4r}(\xi).$$

\end{theorem}

\begin{lemma}\label{Lemm10}
Let $0<r\leq \gb_0$ and $u$ be a positive $\CL_{\xk }$-harmonic function such that\smallskip

\noindent (i) $u\in C(\overline{\xO\setminus B_r(\xi)}),$\smallskip

\noindent (ii) $$\lim_{x\rightarrow x_0}\frac{u(x)}{W(x)}=0\qquad\forall x_0\in\xO\setminus \overline{B_{r}(\xi)},$$
uniformy with respect to $x_0$.
\smallskip

\noindent Then
$$c_{42}^{-1}\frac{u(x_r(\xi))}{W(x_r(\xi))}w^x(\xD_r(\xi))\leq u(x)\leq c_{42}\frac{u(x_r(\xi))}{W(x_r(\xi))}w^x(\xD_r(\xi))\qquad\forall x\in\xO\setminus\overline{B_{2r}(\xi)},$$
with $c_{42}$ depends only on $\xk $ and $\xO.$
\end{lemma}  
\begin{proof}
By Propositions \ref{maincomp1}, \ref{maincomp} we have that there exists $C$ such that

$$\frac{1}{C}\frac{u(x_{2r}(\xi))}{w^{x_{2r}(\xi)}(\xD_r(\xi))}\leq \frac{u(x)}{w^{x}(\xD_r(\xi))}\leq C\frac{u(x_{2r}(\xi))}{w^{x_{2r}(\xi)}(\xD_r(\xi))},\quad\forall x\in \xO\cap \partial B_{2r}(\xi).$$
 by Harnack inequality between we have that
$$\frac{1}{C}\frac{u(x_{r}(\xi))}{w^{x_{r}(\xi)}(\xD_r(\xi))}\leq \frac{u(x)}{w^{x}(\xD_r(\xi))}\leq C\frac{u(x_{r}(\xi))}{w^{x_{r}(\xi)}(\xD_r(\xi))},\quad\forall x\in \xO\cap \partial B_{2r}(\xi).$$
Also by Harnack inequality we have that
$$w^{x_{r}(\xi)}(\xD_r(\xi))\geq C w^{x_{\frac{r}{2}}(\xi)}(\xD_r(\xi))> C_0W(x_r(\xi)),$$
where in the last inequality above we have used Lemma \ref{lem2.1}.

Combining all above we have that
$$C^{-1}\frac{u(x_r(\xi))}{W(x_r(\xi))}w^x(\xD_r(\xi))\leq u(x)\leq C\frac{u(x_r(\xi))}{W(x_r(\xi))}w^x(\xD_r(\xi)),\quad\forall x\in \xO\cap \partial B_{2r}(\xi).$$
The result follows by an argument like in step 3 in Lemma \ref{lemharn}.
\end{proof}

\subsection{The Poisson kernel of $\CL_{\xk }$}
In this section we state some properties of the Poisson kernel associated to $\CL_{\xk }.$
\begin{defin}
Fix $\xi\in\partial\xO.$ A function $K$ defined in $\xO$ is called a kernel function at $\xi$ with pole at $x_0\in\xO$ if\smallskip

\noindent(i) $K(\cdot,\xi)$ is $\CL_{\xk }$-harmonic in $\xO,$\smallskip

\noindent(ii) $K(\cdot,\xi)\in C(\overline{\xO}\setminus\{\xi\})$ and for any $\eta\in\prt\Gw\setminus\{\xi\}$
$$\lim_{x\rightarrow \eta}\frac{K(x,\xi)}{W(x)}=0,$$\smallskip
(iii) $K(x,\xi)>0$ for each $x\in\xO$ and $K(x_0,\xi)=1.$
\end{defin}
\begin{prop}
There exists one and only one kernel function for $\CL_{\xk }$ at $\xi$ with pole at $x_0.$
\end{prop}
\begin{proof}
The proof is the same as the one of Theorem 3.1 in \cite{caffa}.
\end{proof}
\begin{prop}
The kernel function $K_{\CL_{\xk }}(x,\xi),$ is continuous in $\xi$ on the boundary of $\xO.$
\end{prop}
\begin{proof}
The proof is the same as the one of  Corollary 3.2 in \cite{caffa}.
\end{proof}

We have proved the uniqueness of Poisson kernel. By (\ref{RD}), Theorems \ref{lem2.2*}, \ref{lem2.2} and Proposition \ref{green} we have the following result.
\begin{theorem}\label{poisson}
Assume $0<\xk \leq \frac{1}{4}$. There exists a positive constant $c_{43}$ such that
\be\label{poissonest}\frac{1}{c_{43}}\frac{d^\frac{\xa_+}{2}(y)}{|\xi-y|^{N+\xa_+-2}}\leq K_{\CL_{\xk }}(y,\xi)\leq c_{43}\frac{d^\frac{\xa_+}{2}(y)}{|\xi-y|^{N+\xa_+-2}}.
\ee
\end{theorem}

\begin{remark} As in \cite[Remark 3.9]{hunt}, the Martin kernel which is classical defined by
 $$\tilde K_{\CL_{\xk }}(x,\xi)=\lim_{x\rightarrow\xi}\frac{G_{\CL_{\xk }}(x,y)}{G_{\CL_{\xk }}(x,x_0)}\qquad\forall \xi\in\partial\xO,$$
 coincides with the Poisson kernel $K_{\CL_{\xk }}$.
 \end{remark}

\begin{theorem}\label{Lemm11}
Let $u$ be a positive $\CL_{\xk }$-harmonic in the domain $\xO.$ Then $u\in L^1_{\gf_\gk}(\Gw)$ and there exists a unique Radon measure $\xm$ on $\partial\xO$ such that
$$u(x)=\int_{\partial\xO}K_{\CL_{\xk }}(x,\xi)d\xm(\xi).$$
\end{theorem}  
\begin{proof}
The proof is the same as the one of  Theorem 4.3 in \cite{hunt}.
\end{proof}

Actually the measure $\xm$ is the boundary trace of $u$. This boundary trace can be achieved in a {\it dynamic way} as in  \cite[Sect 2]{MV}. We consider a increasing sequence of bounded smooth domains $\{\xO_n\}$ such that $\overline{\xO_n}\subset \xO_{n+1}$, $\cup_n\xO_n=\xO$ and $\mathcal{H}^{N-1}(\xO_n)\rightarrow\mathcal{H}^{N-1}(\xO)$. such a sequence is a
{\it smooth exhaustion} of $\Gw$. For each $n$ the operator $\CL_{\xk }^{\Gw_n}$ defined by
\be\label{redu1}
\CL_{\xk }^{\Gw_n}u=-\xD u-\myfrac{\xk }{d^2(x)}u
\ee
is uniformly elliptic and coercive in $H^1_0(\Gw_n)$ and its first eigenvalue $\gl_{\xk }^{\Gw_n}$ is larger than $\gl_{\xk }$.
If $h\in C(\prt\Gw_n)$ the following problem
\be\label{sub}\BA{lll}
\CL_{\xk }^{\Gw_n}v=0\qquad&\text{in}\;\;\Gw_n\\
\phantom{\CL_{\xk }^{\Gw_n}}
v=h\qquad&\text{on}\;\;\prt\Gw_n
\EA
\ee
admits a unique solution which allows to define the $\CL_{\xk }^{\Gw_n}$-harmonic measure on $\prt\Gw_n$
by
\be\label{redu2}
v(x_0)=\myint{\prt \Gw_n}{}h(y)d\gw^{x_0}_{\Gw_n}(y).
\ee
Thus the Poisson kernel of $\CL_{\xk }^{\Gw_n}$ is
\be\label{redu2'}
K_{\CL_{\xk }^{\Gw_n}}(x,y)=\myfrac{d\gw^{x}_{\Gw_n}}{d\gw^{x_0}_{\Gw_n}}(y)\qquad\forall y\in\prt\Gw_n.
\ee

\begin{prop}\label{22222}
Assume $0<\xk \leq \frac{1}{4}$ and let $x_0\in\xO_1$. Then for every $Z\in C(\overline{\xO}),$
\be\label{2.27}
\lim_{n\rightarrow\infty}\int_{\partial\xO_n}Z(x)W(x)d\gw^{x_0}_{\Gw_n}(x)=\int_{\partial\xO}Z(x)d\gw^{x_0}(x).
\ee
\end{prop}  
\begin{proof}
We recall that   $d\in C^2(\overline{\xO}_{\xe})$ for any $0<\xe\leq \xb_0$ and  let $n_0\in \mathbb{N}$ be such that
$$\hbox{dist}(\partial\xO_n,\partial\xO)<\frac{\xb_0}{2},\quad\forall n\geq n_0.$$
For $n\geq n_0$ let $w_n$ be the solution of
\be\label{sub-n}\BA{lll}
\CL_{\xk }^{\Gw_n}w_n=0\qquad&\text{in}\;\;\Gw_n\\
\phantom{\CL_{\xk }^{\Gw_n}}
w_n=W\qquad&\text{on}\;\;\prt\Gw_n
\EA
\ee
It is straightforward to see that the proof of Propositions \ref{mainlemma1} and \ref{mainlemma2} it is inferred  that there exists a positive constant $c_{44}=c_{44}(\xO,\xk )$ such that
$$\|w_n\|_{L^\infty(\xO_n)}\leq c_{44},\quad \forall n\geq n_0.$$
Furthermore
\be
w_n(x_0)=\int_{\partial\xO_n}W(x)d\gw^{x_0}_{\Gw_n}(x)<c_{45}.\label{wn}
\ee
We extend $\gw^{x_0}_{\Gw_n}$ as a Borel measure on $\overline{\xO}$ by setting $\gw^{x_0}_{\Gw_n}(\overline{\xO}\setminus\xO_n) = 0,$ and keep
the notation $\gw^{x_0}_{\Gw_n}$ for the extension. Because of (\ref{wn}) the sequence $\{W\gw^{x_0}_{\Gw_n}\}$ is bounded in the space $\mathfrak M_b(\overline\Gw)$ of bounded Borel measures in $\overline\Gw$. Thus there exists
a subsequence (still denoted by $\{W(x)\gw^{x_0}_{\Gw_n}\}$ which converges narrowly to some
positive measure, say $\widetilde{\gw}$ which is clearly supported by $\partial\xO$ and satisfies
$\|\widetilde{\gw}\|_{\mathfrak M_b}\leq c_{45}$ as in (\ref{wn}). For every
$Z \in C(\overline{\xO})$ there holds
$$\lim_{n\rightarrow\infty}\int_{\partial\xO_n}Z(x)Wd \gw^{x_0}_{\Gw_n}=\int_{\partial\xO}Zd \widetilde{\gw}.$$
Let $\xz:=Z\lfloor_{\prt\Gw}$ and $$z(x):=\int_{\partial\xO}K_{\CL_{\xk }}(x,y)\xz(y)d \gw^{x_0}(y).$$ Then
$$\lim_{d(x)\to 0}\myfrac{z(x)}{W(x)}=\gz\;\;\text{ and }\;\;z(x_0)=\int_{\partial\xO}\gz d \gw^{x_0}.
$$

By Propositions \ref{mainlemma1} and \ref{mainlemma2}, $\frac{z}{W}\in C(\overline{\xO})$. Since $\frac{z}{W}\lfloor_{\prt\Gw_n}$ converges uniformly to $\gz$, there holds
$$z(x_0)=\int_{\partial\xO_n}z\lfloor_{\prt\Gw_n}d \gw^{x_0}_{\Gw_n}=\int_{\partial\xO_n}W\frac{z\lfloor_{\prt\Gw_n}}{W}d \gw^{x_0}_{\Gw_n}\to \int_{\partial\xO}\gz d \tilde\gw\;\text{ as }\;n\to\infty.$$
It follows
$$\int_{\partial\xO}\xz d \widetilde{\gw}=\int_{\partial\xO}\xz d \gw^{x_0},\quad\forall\xz\in C(\partial\xO). $$
Consequently $\widetilde{\gw}=d \gw^{x_0}.$ Because the limit does not depend on the
subsequence it follows that the whole sequence ${W(x)d\gw^{x_0}_{\Gw_n}}$ converges weakly to $w.$ This
implies \eqref{2.27}.
\end{proof}

In the same way we have
\begin{prop}\label{Lemm12}
Let $x_0\in\xO_1$ and $\xm\in\mathfrak{M}(\partial\xO).$ Put $$v:=\int_{\partial\xO}K_{\CL_{\xk }}(x,y)d\xm(y),$$ then for every $Z\in C(\overline{\xO}),$
\be
\lim_{n\rightarrow\infty}\int_{\partial\xO_n}Z(x)v d \gw^{x_0}_{\Gw_n}=\int_{\partial\xO}Z(x)d \xm.\label{2.28}
\ee
\end{prop}
\begin{proof}
The proof is same as the proof of Lemma 2.2 in \cite{MV} and we omit it.
\end{proof}
\medskip

The next result is an analogous of Green formula for positive $\CL_\gk$-harmonic functions.

\begin{prop}\label{trlin}
Let $v$ be a positive $\CL_\gk$-harmonic function in $\Gw$ with boundary trace $\gm$. Let $Z\in C^2(\overline\Gw)$ and
$\tilde G\in C(\Gw)$ which coincides with $G_{\CL_\gk}(x_0,.)$ in $\Gw_\gd$ for some $0<\gd<\gb_0$ and some $x_0\notin
\overline\Gw_{\gb_0}$.  Assume
\be\label{action0}
|\nabla \tilde G.\nabla Z|\leq c'_{45}\gf_\gk.
\ee
Then, if we set $\gz=Z\tilde G$, there holds
\be\label{action}
\int_{\Gw}v\CL_\gk\gz dx=\myint{\prt\Gw}{}Zd\gm.
\ee
\end{prop}
\begin{proof} Let $\{\Gw_j\}$ be a smooth exhaustion of $\Gw$ with Green kernel $G^{\Gw_j}_{\CL_\gk}$ and Poisson kernel
$P^{\Gw_j}_{\CL_\gk}=-\prt_{\bf n}G^{\Gw_j}_{\CL_\gk}$. We assume that $j\geq j_0$ where  $\overline\Gw'_\gd\subset \Gw_j$. Set $\gz_j=Z\tilde G_j$, where the functions $\tilde G_j$ are $C^\infty$ in $\Gw_j$, coincide with $G^{\Gw_j}_{\CL_\gk}(x_0,.)$ in $\Gw_j\cap\overline\Gw_\gd$ and satisfy $\tilde G_j\to \tilde G$ in $C^2(\Gw)$-loc and such that $|\nabla \tilde G_j.\nabla Z|\leq c'_{45}\gf_\gk$.
$$
\int_{\Gw_j}v\CL_\gk\gz_j dx=-\int_{\prt\Gw_j}v\frac{\prt \gz_j}{\prt {\bf n}} dS=
-\int_{\prt\Gw_j}vZ\frac{\prt \tilde G_j}{\prt {\bf n}} dS=\int_{\prt\Gw_n}vZP^{\Gw_j}_{\CL_\gk}(x_0,.)dS=\int_{\prt\Gw_j}vZd\gw^{x_0}_{\Gw_j}.
$$
By (\ref{2.28})
$$\int_{\prt\Gw_j}vZd\gw^{x_0}_{\Gw_j}\to \int_{\partial\xO}Z(x)d \xm\quad\text{as }\,j\to\infty.
$$
Next
$$\CL_\gk\gz_j=Z\CL_\gk\tilde G_j +\tilde G_j\Gd Z+2\nabla \tilde G_j.\nabla Z.
$$
Since $v\in L^1_{\gf_\gk}(\Gw)$, the proof follows .
\end{proof}

Similarly we can prove
\begin{prop}\label{remark2}
Let $v$ be a positive $\CL_\gk$-harmonic function in $\Gw$ with boundary trace $\gm$. Let $0\leq Z\in C^2(\overline\Gw)$ satisfy
\be\nonumber
|\nabla \tilde \ei.\nabla Z|\leq c'_{45}\gf_\gk.
\ee
Then, if we set $\gz=Z\ei$, there holds
\be\nonumber
\int_{\Gw}v\CL_\gk\gz dx\geq c_0\myint{\prt\Gw}{}Zd\gm,
\ee
where the constant $c_0>0$ depends on $\xO,\;N$ and $\xk.$
\end{prop}

\section{The nonlinear problem with measures data}
\setcounter{equation}{0}
\subsection{The linear boundary value problem with $L^1$ data}

In the sequel we denote by $\xo=\xo^{x_0}$ the $\CL_{\xk }$-harmonic measure in $\xO$, for some fixed $x_0\in\xO$ and by $\mathfrak{M}_{\ei}(\xO)$ be the space of Radon measures $\xn$ in $\xO$ such that $\ei d|\xn|$ is a bounded measure. We also denote by  $\mathfrak{M}(\prt\xO)$ the space of Radon measures on $\prt\xO$ with respective norms $\|\gn\|_{\mathfrak{M}_{\ei}(\xO)}$ and
$\|\gm\|_{\mathfrak{M}(\prt\xO)}$. Their respective positive cones are denoted by $\mathfrak{M}^+_{\ei}(\xO)$ and $\mathfrak{M}^+(\prt\xO)$.
By Fubini's theorem and (\ref{greenest}), for any $\xn\in \mathfrak{M}_{\ei}(\xO)$ we can define
$$\mathbb{G}_{\CL_{\xk }}[\xn](x)=\int_{\xO}G_{\CL_{\xk }}(x,y)d\xn(y),$$
and  we have
 \be\label{poi1}
\|\mathbb{G}_{\CL_{\xk }}[\xn]\|_{L_{\ei}^1(\xO)}\leq c_{46}\|\gn\|_{\mathfrak{M}_{\ei}(\xO)}.
\ee
If $\gm\in \mathfrak{M}(\partial\xO),$  we set
$$\mathbb{K}_{\CL_{\xk }}[\gm](x)=\int_{\partial\xO}K_{\CL_{\xk }}(x,y)d\xm(y),$$
 \be\label{poi2}
\|\mathbb{K}_{\CL_{\xk }}[\gm]\|_{L_{\ei}^1(\xO)}\leq c_{47}\|\gm\|_{\mathfrak{M}(\prt\xO)}.
\ee
In the above inequalities $c_{46}$ and $c_{47}$ are positive constants depending on $\Gw$ and $\gk$.

For $0<\xk \leq\frac{1}{4}$, we define the space of test functions $\mathbf{X}(\xO)$ by
\be\label{testXO}\BA {ll}
\mathbf{X}(\xO)=\left\{\eta\in H^1_{loc}(\Gw):\frac{\eta}{d^{\frac{\xa_+}{2}}}\in H^1(\xO,d^{\xa_+}dx)\,,\;(\gf_{\xk })^{-1}\CL_{\xk }\eta\in L^{\infty}(\Omega)\right\}.
\EA\ee
The next statement follows immediately from Propositions (\ref{reglemma}) and (\ref{reglemma1}).
\begin{lemma} \label{ETA}Let $0<\gk\leq \frac{1}{4}$. If $m\in L^\infty(\Gw)$, the solution $\eta_m$ of
\begin{equation}\label{T*1}\BA {lll}
\CL_\gk\eta_m=m\gf_{\xk }\qquad&\text{in }\;\Gw\\
\phantom{\CL_\gk}
\eta_m=0\qquad&\text{on }\;\prt\Gw
\EA\end{equation}
obtained by Propositions (\ref{reglemma}) and (\ref{reglemma1}) with $f_0=m$ and $h=0$ belongs to $\mathbf{X}(\xO)$. Furthermore
\begin{equation}\label{T0}
- \frac{\|m_-\|_{L^\infty(\Gw)}}{\gl_\gk}\gf_\gk\leq -\eta_{m_-}\leq\eta_m\leq \eta_{m_+}\leq
\frac{\|m_+\|_{L^\infty(\Gw)}}{\gl_\gk}\gf_\gk.
\end{equation}
\end{lemma}

In the next Proposition we give some key estimates for the weak solutions of
\begin{equation}\label{T2}\BA {lll}
\CL_\gk u=f\qquad&\text{in }\;\Gw\\
\phantom{\CL_\gk}
u=h\qquad&\text{on }\;\prt\Gw
\EA\end{equation}
\begin{prop}\label{2.6}
For any $(f,h)\in L_{\ei}^1(\xO)\times L^1(\prt\xO,d\gw)$ there exists a unique $u:=u_{f,h}\in L_{\ei}^1(\xO)$ such that
\be\label{2.42}
\int_{\xO}u \CL_{\xk }\eta dx=\int_\xO f\eta dx+\myint{\Gw}{}\mathbb{K}_{\CL_{\xk }}[h\gw]\CL_{\xk }\eta dx
\qquad\forall \eta\in\mathbf{X}(\xO).
\ee
There holds
\be\label{2.43}
u=\mathbb{G}_{\CL_{\xk }}[f]+\mathbb{K}_{\CL_{\xk }}[h\gw]
\ee
and
\be\label{poi3}
\|u\|_{L_{\ei}^1(\xO)}\leq c_{46}\|f\|_{L^1_{\ei}(\xO)}+ c_{47}\|h\|_{L^1(\prt\xO,d\gw)}.
\ee
Furthermore, for any $\eta\in \mathbf{X}(\xO)$, $\eta\geq 0$, we have
\be\label{poi4}
\myint{\Gw}{}|u|\CL_{\xk }\eta dx\leq \myint{\Gw}{}f\eta\rm{sgn} (u)dx+
\myint{\Gw}{}\mathbb{K}_{\CL_{\xk }}[|h|\gw]\CL_{\xk }\eta dx,
\ee
and
\be\label{poi5}
\myint{\Gw}{}u_+\CL_{\xk }\eta dx\leq \myint{\Gw}{}f\eta\rm{sgn}_+ (u)dx+
\myint{\Gw}{}\mathbb{K}_{\CL_{\xk }}[h_+\gw]\CL_{\xk }\eta dx,
\ee
\end{prop}  
\begin{proof} \noindent{\it Step 1: proof of estimate (\ref{poi3})}. Assume $u$ satisfies  (\ref{2.42}). If $\eta=\eta_{\rm{sgn} (u)}$, we have
$$\int_{\xO}|u|\gf_\gk dx=\int_{\xO}u \CL_{\xk }\eta dx= \int_\xO f\eta dx+\myint{\Gw}{}\mathbb{K}_{\CL_{\xk }}[h\gw]\rm{sgn} (u)\gf_\gk dx.
$$
By (\ref{poi1}), (\ref{poi2})
$$\int_\xO f\eta dx\leq \frac{1}{\gl_\gk}\int_\xO |f|\gf_\gk dx,
$$
$$\myint{\Gw}{}\mathbb{K}_{\CL_{\xk }}[h\gw]\rm{sgn} (u)\gf_\gk dx\leq c_{47}\myint{\prt\Gw}{}|h| d\gw,
$$
which implies (\ref{poi3}) and uniqueness. \smallskip

\noindent{\it Step 2: proof of existence}. If $f$ and $h$ are bounded, existence follows from Propositions \ref{reglemma}, \ref{reglemma1}. In the general case let $\{(f_n,h_n)\}$ be a sequence of bounded measurable functions in $\Gw$ and $\prt\Gw$ which converges to $\{(f,h)\}$ in $L^1_{\ei}(\xO)\times L^1(\prt\xO,d\gw)$. Let $\{u_n\}=\{u_{f_n,h_n}\}$ be the sequence weak solutions of (\ref{T2}). By estimate (\ref{poi3}) it is a Cauchy sequence in $L^1_{\ei}(\xO)$ which converges to $u$. Letting $n\to\infty$ in identity
\be\label{2.42n}
\int_{\xO}u_n \CL_{\xk }\eta dx=\int_\xO f_n\eta dx+\myint{\Gw}{}\mathbb{K}_{\CL_{\xk }}[h_n\gw]\CL_{\xk }\eta dx
\ee
where $ \eta\in\mathbf{X}(\xO)$ implies that $u=u_{f,h}$.\smallskip

 \noindent{\it Step 3: proof of estimates (\ref{poi4}), (\ref{poi5})}. We first assume that $f$ is bounded and $h$ is $C^2(\overline{\xO})$.
Set $\xO_n=\xO'_{\frac{1}{n}},$
Let $u_n$ be the unique solution of

\be\label{cor1}\BA{lll}
\CL_{\xk }u_n=f\qquad&\text{in}\;\;\Gw_n\\
\phantom{\CL_{\xk }^{\Gw_n}}
v_n=Wh\qquad&\text{on}\;\;\prt\Gw_n
\EA
\ee

Then $u_n$ can be written in the form $$u_n=\mathbb{G}_{\CL_{\xk }}^{n}[f](x)+w_n,$$
where $$\mathbb{G}_{\CL_{\xk }}^{n}[f](x)=\int_{\xO}G_{\CL_{\xk }}^{n}(x,y)f(y)dy,$$
$G_{\CL_{\xk }}^{n}$ is the Green Kernel of $\CL_{\xk }$ in $\xO_n,$ and $w_n$ satisfies
 \be\label{cor2}\BA{lll}
\CL_{\xk }v=0\qquad&\text{in}\;\;\Gw_n\\
\phantom{\CL_{\xk }^{\Gw_n}}
v=Wh\qquad&\text{on}\;\;\prt\Gw_n.
\EA
\ee

Now note that ${G}_{\CL_{\frac{1}{4}}}^{n}(x,y)\leq {G}_{\CL_{\frac{1}{4}}}(x,y):={G}_{\CL_{\frac{1}{4}}}^{\Gw}$,
and for any $x,y\in \Gw,$ $x\neq y$
\be
{G}_{\CL_{\frac{1}{4}}}^{n}(x,y)\uparrow {G}_{\CL_{\frac{1}{4}}}(x,y).
\ee
Also in view of the proof of Proposition \ref{22222} there exists $c_0>0$ which depends on $\xO,\;N,\;\xk,||h||_{C^2(\overline{\xO})}$ such that
$$\sup_{x\in\xO_n}|w_n|<c_0,\;\;\forall n\in \mathbb{N},$$

and $w_n\rightarrow\mathbb{K}_{\CL_{\xk }}[h\gw].$ Thus by the properties of Green kernel that we described above, we have that there exists a constant $c_{01}$ $\xO,\;N,\;\xk,||h||_{C^2(\overline{\xO})},||f||_{L^\infty(\xO)},$ such that 
$$\sup_{x\in\xO_n}|u_n|<c_0,\;\;\forall n\in \mathbb{N},$$ 
and $$u_n\rightarrow u= \mathbb{G}_{\CL_{\xk }}[f]+\mathbb{K}_{\CL_{\xk }}[h\gw].$$

Let $\eta\in \mathbf{X}(\xO)$ be nonnegative function and let $\eta_n$ be the solution of the problem
$$\CL_{\xk }v=\CL_{\xk }\eta,\qquad v=0\;\;\mathrm{on}\;\partial\xO_n.$$
Then there exists $c_0=c_0(||\xD\eta||_{L^\infty(\xO)},\xk,N,\xO)$ such that $|\eta_n|\leq c_0\ei$ and
$$\CL_{\xk }\eta_n\rightarrow L_{c_0}\eta,\qquad\eta_n\rightarrow\eta.$$
Let $z_n$ be the solution of
$$\CL_{\xk }v=\text{sgn}(\eta_n)\CL_{\xk }\eta,\qquad v=0\;\;\mathrm{on}\;\partial\xO_n.$$
Then  $z_n\geq \max(\eta_n,0)$ since 
$$\CL_{\xk }|\eta_n|\leq \text{sgn}(\eta_n)\CL_{\xk }\eta_n=\text{sgn}(\eta_n)\CL_{\xk }\eta,$$
and $|z_n|\leq c_0\ei,$
$$\CL_{\xk }z_n\rightarrow L_{c_0}\eta,\qquad z_n\rightarrow\eta.$$ 
Now note that $z_n\geq0$ and $z_n\in C^1(\overline{\xO}_n).$
Also, the following inequality holds (see eg. \cite{Ver1}), 
\begin{align}\nonumber
\int_{\xO}|u_n| L_{c_0}z_n dx&\leq\int_\xO f z_n\text{sgn}(u_n)-\int_{\partial\xO}\frac{\partial z_n}{\partial\xn}|h|W dx\\ 
&=\int_\xO f z_n\text{sgn}(u_n)+\int_{\xO}\widetilde{w}_n L_{c_0}z_n dx,
\label{2.51}
\end{align}
where $\widetilde{w}_n$ is the solution of 
\be\label{cor3}\BA{lll}
\CL_{\xk }v=0\qquad&\text{in}\;\;\Gw_n\\
\phantom{\CL_{\xk }^{\Gw_n}}
v=W|h|\qquad&\text{on}\;\;\prt\Gw_n.
\EA
\ee
In view of the proof of Proposition \ref{22222} there exists $c_{02}>0$ which depends on $\xO,\;N,\;\xk,||h||_{C^2(\overline{\xO})}$ such that
$$\sup_{x\in\xO_n}|\widetilde{w}_n|<c_0,\;\;\forall n\in \mathbb{N},$$

\noindent and $\widetilde{w}_n\rightarrow\mathbb{K}_{\CL_{\xk }}[|h|\gw].$
Thus combining all above and taking the limit in \eqref{2.51} we have the proof of (\ref{poi4}) in the case that $f$ is bounded and $h\in C^2(\overline{\xO}).$ We note here that for any $h\in C^2(\partial\xO)$ there exists $H_m\in C^2(\overline{\xO}),$ such that
$||H_m||_{C^2(\overline{\xO})}\leq c_{03}||h||_{L^\infty(\partial\xO)},$ for some constant $c_{03}$ which depends only on $\xO,$ and $H_m\rightarrow h$ in $L^\infty(\partial \xO)$. Thus it is not hard to prove that \eqref{22222} is valid if $f$ is bounded and $h\in C^2(\partial \xO).$ In the general case we consider a sequence $(f_n,h_n)\subset L^\infty(\Gw)\times C^2(\prt\Gw)$ which converges to $(f,h)$ in $L^1(\Gw)\times L^1(\prt\Gw,d\gw)$. Since $u_{f_n,h_n}$ converges to $u_{f,h}$ in $L^1_{\gf_\gk}(\Gw)$ we obtain (\ref{poi4}) from the inequality
verified by any $\eta\in\mathbf{X}(\xO)$
$$\myint{\Gw}{}|u_{f_n,h_n}|\CL_{\xk }\eta dx\leq \myint{\Gw}{}f_n\eta\rm{sgn} (u)dx+
\myint{\Gw}{}\mathbb{K}_{\CL_{\xk }}[|h_n|\gw]\CL_{\xk }\eta dx.$$
\end{proof}
The proof of (\ref{poi5}) is follows by adding \eqref{2.42} and \eqref{poi4}.
\subsection{General nonlinearities}
Throughout this section $\Gw$ is a smooth bounded domain and $\gk$ a real number in $(0,\frac{1}{4}]$. Let $g:\BBR\mapsto\BBR$ be a nondecreasing continuous function, vanishing at $0$ for simplicity. The problem under consideration is the following
 \begin{equation}\label{M1}\BA {lll}
-\Gd u-\myfrac{\xk }{d^2}u+g(u)=\gn\qquad&\text{in }\Gw\\[2mm]\phantom{-\Gd -\myfrac{\xk }{d^2}u+g(u)}
u=\gm&\text{in }\prt\Gw
\EA\end{equation}
where $\gn$ and $\gm$ are  Radon measures respectively in $\Gw$ and $\prt\Gw$. \medskip

\noindent {\bf Definition. }Let $\gn\in\mathfrak M_{\gf_\gk}(\Gw)$ and $\gm\in\mathfrak M(\prt\Gw)$. We say that  $u$ is a solution  of (\ref{M1}) if $u\in L^1_{\gf_\gk}(\Gw)$, $g(u)\in L^1_{\gf_\gk}(\Gw)$ and for any
$\eta\in \mathbf{X}(\xO)$ there holds
 \begin{equation}\label{M2-}\BA {lll}
\myint{\Gw}{}\left(u\CL_\gk\eta+g(u)\eta\right)dx=\myint{\Gw}{}\left(\eta d\gn+\mathbb{K}_{\CL_{\xk }}[\gm]\CL_{\xk }\eta\right)dx
\EA\end{equation}

Our main existence result for {\it  subcritical nonlinearities} is the following.

\begin{theorem}\label{gen} Assume $g$ satisfies
 \begin{equation}\label{M2}\BA {lll}
\myint{1}{\infty}\left(g(s)-g(-s)\right)s^{-2\frac{N-1+\frac{\ga_+}{2}}{N-2+\frac{\ga_+}{2}}}ds<\infty.
\EA\end{equation}
Then for any $(\gn,\gm)\in\mathfrak M_{\gf_\gk}(\Gw)\times\in\mathfrak M(\prt\Gw)$ problem (\ref{M1}) admits a unique solution
$u=u_{\gn,\gm}$. Furthermore the mapping $(\gn,\gm)\mapsto u_{\gn,\gm}$ is increasing and stable in the sense that if
$\{(\gn_n,\gm_n)\}$ converge to $(\gn,\gm)$ in the weak sense of measures, $\{u_{\gn_n,\gm_n}\}$ converges to $u_{\gn,\gm}$ in
$L_{\gf_\gk}^1(\Gw)$.
\end{theorem}

The proof is based upon estimates of $\mathbb{M}_{\CL_{\xk }}$ and $\mathbb{K}_{\CL_{\xk }}$ into Marcinkiewicz spaces.
\begin{lemma}\label{Marcin} Let $\gn\in\mathfrak M^+_{\gf_\gk}(\Gw)$, $\gm\in\mathfrak M^+(\prt\Gw)$ and for $s>0$,
$E_s(\gn)=\{x\in \Gw:\mathbb{G}_{\CL_{\xk }}[\gn](x)>s\}$ and $F_s(\gm)=\{x\in \Gw:\mathbb{K}_{\CL_{\xk }}[\gm](x)>s\}$. If we denote
$$\CE_s(\gn)=\int_{E_s(\gn)}\gf_\gk dx\quad\text{and }\;\;\CF_s(\gm)=\int_{F_s(\gm)}\gf_\gk dx,$$
there holds
 \begin{equation}\label{M3}\BA {lll}
\CE_s(\gn)+\CF_s(\gm)\leq c_{47}\left(\myfrac{\|\gn\|_{\mathfrak M_{\gf_\gk}(\Gw)}+\|\gm\|_{\mathfrak M(\prt\Gw)}}{s}\right)^{\frac{N+\frac{\ga_+}{2}}{N-2+\frac{\ga_+}{2}}}.
\EA\end{equation}
\end{lemma}
\begin{proof}{\it Step 1: estimate of $\CF_s(\gn)$}. By estimate (\ref{poissonest}), for any $\xi\in\prt\Gw$,
$$F_s(\gd_\xi)\subset \tilde F_s(\gd_\xi):=\left\{x\in\Gw:\frac{d^{\frac{\ga_+}{2}}(x)}{|x-\xi|^{N+\ga_+-2}}\geq \frac{s}{c_{_{43}}}\right\}\subset B_{(\frac{c_{43}}{s})^{\gth}}(\xi),
$$
with $\gth=\frac{1}{N-2+\frac{\ga_+}{2}}$. From (\ref{Lin1}), (\ref{Lin2})
$$\CF_s(\gd_\xi)\leq \myint{B_{(\frac{c_{43}}{s})^{\gth}}(\xi)}{}\gf_\gk dx\leq c_{49}
\myint{B_{(\frac{c_{43}}{s})^{\gth}}(\xi)}{}|x-\xi|^{\frac{\ga_+}{2}} dx=c_{50}s^{-\frac{N+\frac{\ga_+}{2}}{N-2+\frac{\ga_+}{2}}}.
$$
Therefore, for any $s_0>0$ and any Borel set $G\subset\Gw$
$$\BA{lll}\myint{G}{}K_{\CL_\gk}(x,\xi)\gf_\gk dx\leq s_0\myint{G}{}\gf_\gk dx+\myint{F_{s_0}(\gd_\xi)}{}K_{\CL_\gk}(x,\xi)\gf_\gk dx\\[4mm]\phantom{\myint{G}{}K_{\CL_\gk}(x,\xi)\gf_\gk dx}
\leq s_0\myint{G}{}\gf_\gk dx-\myint{s_0}{\infty}sd\CF_s(\gd_\xi)
\\[4mm]\phantom{\myint{G}{}K_{\CL_\gk}(x,\xi)\gf_\gk dx}
\leq s_0\myint{G}{}\gf_\gk dx+c_{50}\myint{s_0}{\infty}s^{-\frac{N+\frac{\ga_+}{2}}{N-2+\frac{\ga_+}{2}}}ds
\\[4mm]\phantom{\myint{G}{}K_{\CL_\gk}(x,\xi)\gf_\gk dx}
\leq s_0\myint{G}{}\gf_\gk dx+c_{51}s_0^{-\frac{2}{N-2+\frac{\ga_+}{2}}}
\EA$$
Next we choose $s_0$ so that the two terms in the right part of the last inequality are equal and we get
 \begin{equation}\label{M4}\BA {lll}
\myint{G}{}K_{\CL_\gk}(x,\xi)\gf_\gk dx\leq c_{52}\left(\myint{G}{}\gf_\gk dx\right)^{\frac{2}{N+\frac{\ga_+}{2}}}.
\EA\end{equation}
Henceforth, for any $\gm\in \mathfrak M(\prt\Gw)$, there holds by Fubini's theorem,
 \begin{equation}\label{M5}\BA {lll}
\myint{G}{}\mathbb{K}_{\CL_{\xk }}[|\gm|]\gf_\gk dx= \myint{\Gw}{}\myint{G}{}K_{\CL_{\xk }}(x,\xi)\phi_\gk(x) dx d|\gm|(\xi)
\leq c_{52}\|\gm\|_{\mathfrak M(\prt\Gw)}\left(\myint{G}{}\gf_\gk dx\right)^{\frac{2}{N+\frac{\ga_+}{2}}}.
\EA\end{equation}
If we take in particular $G=F_s(|\gm|)$, we derive
$$s\CF_s(|\gm|)\leq c_{52}\|\gm\|_{\mathfrak M(\prt\Gw)}\left(\CF_s(|\gm|)\right)^{\frac{2}{N+\frac{\ga_+}{2}}},
$$
which yields to (\ref{M3}) with $\gn=0$.\smallskip

\noindent{\it Step 2: estimate of $\CE_s(\gn)$}. By estimate (\ref{greenest}), for any $y\in\Gw$,
$$E_s(\gd_y)\subset \tilde E_s(\gd_y):=\left\{x\in\Gw:\frac{d^{\frac{\ga_+}{2}}(y)d^{\frac{\ga_+}{2}}(x)}{|x-y|^{N+\ga_+-2}}\geq \frac{s}{c_{_{3}}}\right\}\bigcap\left\{x\in\Gw:\frac{1}{|x-y|^{N-2}}\geq \frac{s}{c_{_{3}}}\right\},
$$
A simple geometric verification shows that there exists an open  domain $\CO\subset\overline \CO\subset\Gw$ such that
$y\in \CO$, $\dist (y,\CO^c)>\gl_1d(y)$, $\CO\subset B_{\gl_2d(y)}(y)$ for some $0<\gl_1<\gl_2<1$ independent of $y$ with the following properties
$$\BA {lll}&x\in\CO\Longrightarrow \myfrac{d^{\frac{\ga_+}{2}}(y)d^{\frac{\ga_+}{2}}(x)}{|x-y|^{N+\ga_+-2}}\geq \myfrac{1}{|x-y|^{N-2}}\\[4mm]
& x\in \CO^c\Longrightarrow \myfrac{d^{\frac{\ga_+}{2}}(y)d^{\frac{\ga_+}{2}}(x)}{|x-y|^{N+\ga_+-2}}\leq \myfrac{1}{|x-y|^{N-2}}.
\EA$$
Notice that if
$\Gw=\BBR_+^N$ then $\CO=B_{\frac{\sqrt 5}{2}}(\tilde y)$ where $d(\tilde y)=\frac{3}{2}d(y)$.
Set
$$\tilde E^1_s(\gd_y)=\left\{x\in\Gw:\frac{1}{|x-y|^{N-2}}\geq \frac{s}{c_{_{3}}}\right\}\cap \CO
$$
and
$$\tilde E^2_s(\gd_y)=\left\{x\in\Gw\setminus\CO:\frac{d^{\frac{\ga_+}{2}}(y)d^{\frac{\ga_+}{2}}(x)}{|x-y|^{N+\ga_+-2}}\geq \frac{s}{c_{_{3}}}\right\}
$$

We can easily prove
\begin{align*}
\CE_s(\gd_y)=\myint{ E_s(\gd_y)}{}\gf_\gk dx&\leq \myint{\tilde E_s(\gd_y)}{}\gf_\gk dx\\ 
&\leq\myint{\tilde E_s^1(\gd_y)}{}\gf_\gk dx+\myint{\tilde E_s^2(\gd_y)}{}\gf_\gk dx\leq c_{53}s^{-\frac{N+\frac{\ga_+}{2}}{N-2+\frac{\ga_+}{2}}}(d(y))^{\frac{\ga_+(N+\frac{\ga_+}{2})}{2N-4+\ga_+}}.
\end{align*}
As in step 1, for any Borel subset $\Gth\subset\Gw$, we write
$$\BA{lll}\myint{\Gth}{}G_{\CL_\gk}(x,y)\gf_\gk dx\leq s_0\myint{\Gth}{}\gf_\gk dx+\myint{E_{s_0}(\gd_y)}{}G_{\CL_\gk}(x,y)\gf_\gk dx\\[4mm]\phantom{\myint{\Gth}{}K_{\CL_\gk}(x,y)\gf_\gk dx}
\leq s_0\myint{\Gth}{}\gf_\gk dx-\myint{s_0}{\infty}sd\CE_s(\gd_y)
\\[4mm]\phantom{\myint{\Gth}{}K_{\CL_\gk}(x,y)\gf_\gk dx}
\leq s_0\myint{\Gth}{}\gf_\gk dx+c_{53}(d(y))^{\frac{\ga_+(N+\frac{\ga_+}{2})}{2N-4+\ga_+}}\myint{s_0}{\infty}s^{-\frac{N+\frac{\ga_+}{2}}{N-2+\frac{\ga_+}{2}}}ds
\\[4mm]\phantom{\myint{\Gth}{}K_{\CL_\gk}(x,y)\gf_\gk dx}
\leq s_0\myint{\Gth}{}\gf_\gk dx+c_{54}(d(y))^{\frac{\ga_+(N+\frac{\ga_+}{2})}{2N-4+\ga_+}} s_0^{-\frac{2}{N-2+\frac{\ga_+}{2}}}
\EA$$
 \begin{equation}\label{M6}\BA {lll}
\myint{\Gth}{}G_{\CL_\gk}(x,y)\gf_\gk dx\leq c_{55}(d(y))^{{\frac{\ga_+}{2}}}\left(\myint{G}{}\gf_\gk dx\right)^{\frac{2}{N+\frac{\ga_+}{2}}}
\leq c_{56}\gf_\gk(y)\left(\myint{G}{}\gf_\gk dx\right)^{\frac{2}{N+\frac{\ga_+}{2}}}.
\EA\end{equation}
Thus, for any $\gn\in \mathfrak M_{\gf_\gk}(\Gw)$, we have
\begin{equation}\label{M7}\BA {lll}
\myint{\Gth}{}\mathbb{G}_{\CL_{\xk }}[|\gn|]\gf_\gk dx= \myint{\Gw}{}\myint{\Gth}{}G_{\CL_{\xk }}(x,y)\phi_\gk(x) dx d|\gn|(y)
\leq c_{55}\|\gn\|_{\mathfrak M_{\gf_\gk}(\Gw)}\left(\myint{\Gth}{}\gf_\gk dx\right)^{\frac{2}{N+\frac{\ga_+}{2}}}.
\EA\end{equation}
Thus (\ref{M3}) holds.
\end{proof}\medskip

\noindent{\it Proof of Theorem \ref{gen}}. {\it Step 1: existence and uniqueness}. Let $\{(\gn_n,\gm_n)\}\subset C(\overline\Gw)\times C^1(\prt\Gw)$ which converges to
$(\gn,\gm)$ in the weak sense of measures in $\mathfrak M_{\gf_\gk}(\Gw)\times \mathfrak M(\prt\Gw).$ Set $v_n=\mathbb{K}_{\CL_{\xk }}[\gm_n\xo],$ then $v_n\in L^\infty(\xO)$ and it is $\CL_{\xk}$-harmonic. Set $\widetilde{g}(t,x)=g(t+v_n(x))-g(v_n(x))$ and $\widetilde{f}(x)=\xn_n(x)-g(v_n(x)).$
Let  $J_\gk$ be the functional defined in $L^2(\Gw)$ by the expression
\begin{equation}\label{M9J}\BA {lll}
\CJ_\gk(w)=\myfrac{1}{2}\myint{\Gw}{}\left(|\nabla w|^2-\myfrac{\gk}{d^2}w^2+2J(w)\right)dx-\myint{\Gw}{}\widetilde{f}w\gf_\gk dx
\EA\end{equation}
where $J(w)=\int_{0}^{w}\widetilde{g}(t)dt$ with domain
$$D(\CJ_\gk)=\{w\in {\bf H}_\gk(\Gw):J(w)\in L^1(\Gw)\},
$$
(see definition in 2.1-5). By (\ref{Wnorm3}), $\CJ_\gk$ is a convex lower semicontinuous and coercive functional over $L^2(\Gw)$.
Let $w_n=w_{\gn_n,\gm_n}$ be its minimum, then $u_n=u_{\gn_n,\gm_n}=w_n+v_n$ is the solution of
\begin{equation}\label{M8}\BA {lll}
\CL_{\gk}u_n+g(u_n)=\gn_n\qquad&\text{in }\Gw\\\phantom{\CL_{\gk}+g(u_n)}
u_n=\gm_n\qquad&\text{in }\prt\Gw,
\EA\end{equation}
and for any $\eta\in {\bf X}(\Gw)$, there holds
\begin{equation}\label{M9}\BA {lll}
\myint{\Gw}{}\left(u_n\CL_\gk\eta+g(u_n)\eta\right)dx=\myint{\Gw}{}\left(\gn_n\eta+\BBK_{\CL_\gk}[\gm_n\gw]\CL_\gk\eta\right)dx.
\EA\end{equation}
By Proposition \ref{2.6} (\ref{poi4}), there holds, with $\eta=\phi_\gk$,
\begin{equation}\label{M10}\BA {lll}
\myint{\Gw}{}\left(\gl_\gk|u_n|+|g(u_n)|\right)\gf_\gk dx\leq \myint{\Gw}{}\left(|\gn_n|+\BBK_{\CL_\gk}[|\gm_n|\gw]\right)\gf_\gk dx
\\\phantom{\myint{\Gw}{}\left(\gl_\gk|u_n|+|g(u_n)|\right)\gf_\gk dx}
\leq c_{46}\|\gn_n\|_{\mathfrak M_{\gf_\gk}(\Gw)}+c_{47}\|\gm_n\|_{\mathfrak M(\prt\Gw)}
\\\phantom{\myint{\Gw}{}\left(\gl_\gk|u_n|+|g(u_n)|\right)\gf_\gk dx}\leq c_{57}.
\EA\end{equation}
Moreover
\begin{equation}\label{M11}\BA {lll}
-\BBG_{\CL_\gk}[\gn_n^-]-\BBK_{\CL_\gk}[\gm_n^-\gw]\leq u_n\leq \BBG_{\CL_\gk}[\gn_n^+]+\BBK_{\CL_\gk}[\gm_n^+\gw].
\EA\end{equation}
By using the local $L^1$ regularity theory for elliptic equations we obtain that the sequence $\{u_n\}$ is relatively compact in the
$L^1$-local topology in $\Gw$ and that there exist a subsequence still denoted by $\{u_n\}$ and a function
$u\in L^1_{\gf_\gk}(\Gw)$ such that $u_n\to u$ a.e. in $\Gw$. By (\ref{M11})
\begin{equation}\label{M12}\BA {lll}
|g(u_n)|\leq g\left(\BBG_{\CL_\gk}[\gn_n^+]+\BBK_{\CL_\gk}[\gm_n^+\gw]\right)
-g\left(-\BBG_{\CL_\gk}[\gn_n^-]-\BBK_{\CL_\gk}[\gm_n^-\gw]\right).
\EA\end{equation}
We prove the convergence of $\{g(u_n)\}$ to  $g(u)$ in $L^1_{\gf_\gk}(\Gw)$ by the uniform integrability in the following way: let
$G\subset\Gw$ be a Borel subset.
Then for any $s_0>0$
$$\BA {lll}
\myint{G}{}|g(u_n)|\gf_\gk dx\leq \myint{G}{}\left(g\left(\BBG_{\CL_\gk}[\gn_n^+]\right)+g\left(\BBK_{\CL_\gk}[\gm_n^+\gw]\right)- g\left(-\BBG_{\CL_\gk}[\gn_n^-]\right)-g\left(-\BBK_{\CL_\gk}[\gm_n^-\gw]\right)\right)\gf_\gk dx
\\\phantom{\myint{G}{}|g(u_n)|\gf_\gk dx}
\leq s_0\myint{G}{}\gf_\gk dx+\myint{E_s(\gn^+)}{}g\left(\BBG_{\CL_\gk}[\gn_n^+]\right)\gf_\gk dx+
\myint{F_s(\gm^+)}{}g\left(\BBK_{\CL_\gk}[\gm_n^+]\right)\gf_\gk dx
\\\phantom{\myint{G}{}|g(u_n)|\gf_\gk dx\leq s_0\myint{G}{}\gf_\gk dx}
-\myint{E_s(\gn^-)}{}g\left(-\BBG_{\CL_\gk}[\gn_n^-]\right)\gf_\gk dx-\myint{F_s(\gm^-)}{}g\left(-\BBK_{\CL_\gk}[\gm_n^-]\right)\gf_\gk dx
\\\phantom{\myint{G}{}|g(u_n)|\gf_\gk dx}
\leq s_0\myint{G}{}\gf_\gk dx-\myint{s_0}{\infty}g(s)(d\CE_s(\gn_n^+)+d\CF_s(\gm_n^+))+
\myint{s_0}{\infty}g(-s)(d\CE_s(\gn_n^-)+d\CF_s(\gm_n^-)).
\EA$$
But,
$$\BA {lll}
-\myint{s_0}{\infty}g(s)d\CE_s(\gn_n^+)=g(s_0)\CE_{s_0}(\gn_n^+)+\myint{s_0}{\infty}\CE_s(\gn_n^+)dg(s)
\\\phantom{-\myint{s_0}{\infty}g(s)d\CE_s(\gn_n^+)}
\leq g(s_0)\CE_{s_0}(\gn_n^+)+c_{47}\left(\|\gn_n^+\|_{\mathfrak M_{\gf_\gk}}\right)^{\frac{N+\frac{\ga_+}{2}}{N-2+\frac{\ga_+}{2}}}
\myint{s_0}{\infty}s^{-\frac{N+\frac{\ga_+}{2}}{N-2+\frac{\ga_+}{2}}}dg(s)
\\\phantom{-\myint{s_0}{\infty}g(s)d\CE_s(\gn_n^+)}
\leq \frac{2N+\ga_+}{2N-4+\ga_+}c_{47}\left(\|\gn_n^+\|_{\mathfrak M_{\gf_\gk}(\Gw)}\right)^{\frac{N+\frac{\ga_+}{2}}{N-2+\frac{\ga_+}{2}}}\myint{s_0}{\infty}s^{-2\frac{N-1+\frac{\ga_+}{2}}{N-2+\frac{\ga_+}{2}}}g(s)ds.
\EA$$
All the other terms yields similar estimates which finally yields to
 \begin{equation}\label{M13}\BA {lll}
\myint{G}{}|g(u_n)|\gf_\gk dx\leq s_0\myint{G}{}\gf_\gk dx\\[4mm]
\phantom{\myint{G}{}|g(u_n)|}+c_{58}\left(\|\gn_n\|_{\mathfrak M_{\gf_\gk}(\Gw)}+\|\gm_n\|_{\mathfrak M(\prt\Gw)}\right)^{\frac{N+\frac{\ga_+}{2}}{N-2+\frac{\ga_+}{2}}}\myint{s_0}{\infty}s^{-2\frac{N-1+\frac{\ga_+}{2}}{N-2+\frac{\ga_+}{2}}}(g(s)-g(-s))ds
\EA\end{equation}
Since $\|\gn_n\|_{\mathfrak M_{\gf_\gk}(\Gw)}+\|\gm_n\|_{\mathfrak M(\prt\Gw)}$ is bounded independently of $n$, we obtain easily, using (\ref{M2}) and fixing $s_0$ first, that for any $\ge>0$, there exists $\gd>0$ such that
 \begin{equation}\label{M14}\BA {lll}
\myint{G}{}\gf_\gk dx\leq \gd\Longrightarrow \myint{G}{}|g(u_n)|\gf_\gk dx\leq\ge.
\EA\end{equation}
Since
$$|u_n| \leq \BBG_{\CL_\gk}[|\gn_n|]+\BBK_{\CL_\gk}[|\gm_n|\gw],$$
we have by (\ref{M5}), (\ref{M7})
 \begin{equation}\label{M15}\BA {lll}
\myint{G}{}|u_n|\gf_\gk dx\leq \left(c_{52}\|\gm_n\|_{\mathfrak M(\prt\Gw)}+c_{55}\|\gn_n\|_{\mathfrak M_{\gf_\gk}(\Gw)}
\right)\left(\myint{G}{}\gf_\gk dx\right)^{\frac{2}{N+\frac{\ga_+}{2}}}.
\EA\end{equation}
This implies the uniform integrability of the sequence $\{u_n\}$. Letting $n\to\infty$ in identity (\ref{M9}), we conclude that
(\ref{M2-}) holds. Uniqueness, as well as the monotonicity of the mapping $(\gn,\gm)\mapsto u_{\gn,\gm}$, is an immediate consequence of (\ref{poi4}), (\ref{poi5}) and the monotonicity of $g$. \smallskip

\noindent {\it Step 2: stability}. The stability is a direct consequence of inequalities (\ref{M13}) and (\ref{M15}) which show the uniform integrability of the sequence $(u_n,g(u_n))$ in $L^1_{\gf_\gk}(\Gw)\times L^1_{\gf_\gk}(\Gw)$.
{\hspace{10mm}\hfill $\square$}\medskip

The proof of the following result is similar as the one of \cite[Lemma 3.2, Def. 3.3]{MV2}.

\begin{prop}\label{equiv-def} Let $(\gn,\gm)\in\mathfrak M_{\gf_\gk}(\Gw)\times\in\mathfrak M(\prt\Gw)$ such that problem (\ref{M1}) admits a solution $u_{\gm,\gn}$. Then
 \begin{equation}\label{M16}\BA {lll}
 u_{\gm,\gn}=-\BBG_{\CL_\gk}[g(u_{\gm,\gn})]+\BBK_{\CL_\gk}[\gm].
\EA\end{equation}
Conversely, if $u\in L_{\phi_\gk}^1(\Gw)$ such that $g(u)\in L^1_{\phi_\gk}(\Gw)$ satisfies (\ref{M16}), it coincides with the solution
$u_{\gm,\gn}$ of problem (\ref{M1}).
\end{prop}
\subsection{The power case}

In this section we study in particular the following boundary value problem with $\gm\in\mathfrak M(\prt\Gw)$
 \begin{equation}\label{N1}\BA {lll}
\CL_\gk u+|u|^{q-1}u=0\qquad&\text{in }\Gw\\
\phantom{\CL_\gk u+|u|^{q-1}}
u=\gm\qquad&\text{in }\prt\Gw
\EA\end{equation}
A Radon measure for which this problem has a solution (always unique) is called a {\it good measure}. The solution, whenever it exists, is unique and denoted by $u_\gm$. For such a nonlinearity, the condition (\ref{M2}) is fulfilled if and only if
 \begin{equation}\label{N2}\BA {lll}
0<q<q_c:=\myfrac{N+\frac{\ga_+}{2}}{N-2+\frac{\ga_+}{2}}.
\EA\end{equation}
On the contrary, in the {\it supercritical case} i.e. if $q\geq q_c$, a continuity condition with respect to some Besov capacity is needed in order a measure be good. We recall some notations concerning Besov space. For $\gs>0$, $1\leq p<\infty$, we denote by $W^{\gs,p}(\BBR^d)$ the Sobolev space over $\BBR^d$. If $\gs$ is not an integer the Besov space $B^{\gs,p}(\BBR^d)$ coincides with $W^{\gs,p}(\BBR^d)$. When $\gs$ is an integer we denote $\Gd_{x,y}f=f(x+y)+f(x-y)-2f(x)$ and
$$B^{1,p}(\BBR^d)=\left\{f\in L^p(\BBR^d): \myfrac{\Gd_{x,y}f}{|y|^{1+\frac{d}{p}}}\in L^p(\BBR^d\times \BBR^d)\right\}
$$
with norm
$$\|f\|_{B^{1,p}}=\left(\|f\|^p_{L^p}+\myint{}{}\myint{\BBR^d\times \BBR^d}{}\myfrac{|\Gd_{x,y}f|^p}{|y|^{p+d}}dxdy\right)^{\frac{1}{p}}.
$$
Then
$$B^{m,p}(\BBR^d)=\left\{f\in W^{m-1,p}(\BBR^d): D_x^\ga f\in B^{1,p}(\BBR^d)\;\forall\ga\in \BBN^d\,|\ga|=m-1\right\}
$$
with norm
$$\|f\|_{B^{m,p}}=\left(\|f\|^p_{W^{m-1,p}}+\sum_{|\ga|=m-1}\myint{}{}\myint{\BBR^d\times \BBR^d}{}\myfrac{|D_x^\ga\Gd_{x,y}f|^p}{|y|^{p+d}}dxdy\right)^{\frac{1}{p}}.
$$
These spaces are fundamental because they are stable under the real interpolation method as it was developed by Lions and Petree.
For $\ga\in\BBR$ we defined the Bessel kernel of order $\ga$ by $G_\ga(\xi)=\CF^{-1}(1+|.|^2)^{-\frac{\ga}{2}}\CF(\xi)$, where $\CF$ is the Fourier transform of moderate distributions in $\BBR^d$. The Bessel space $L_{\ga,p}(\BBR^d)$ is defined by
$$L_{\ga,p}(\BBR^d)=\{f=G_\ga\ast g:g\in L^{p}(\BBR^d)\},
$$
with norm
$$\|f\|_{L_{\ga,p}}=\|g\|_{L^p}=\|G_{-\ga}\ast f\|_{L^p}.
$$
It is known that if $1<p<\infty$ and $\ga>0$, $L_{\ga,p}(\BBR^d)=W^{\ga,p}(\BBR^d)$ if $\ga\in\BBN$ and $L_{\ga,p}(\BBR^d)=B^{\ga,p}(\BBR^d)$ if $\ga\notin\BBN$, always with equivalent norms. The Bessel capacity is defined for compact subset
$K\subset\BBR^d$ by
$$C^{\BBR^d}_{\ga,p}=\inf\{\|f\|^p_{L_{\ga,p}}, f\in\CS'(\BBR^d),\,f\geq \chi_K\}.
$$
It is extended to open set and then any set by the fact that it is an outer measure.
Our main result is the following

\begin{theorem}\label{supcr} Assume $0<\gk\leq\frac{1}{4}$. Then $\gm\in\mathfrak M^+(\prt\Gw)$ is a good measure if and only if it is absolutely continuous with respect to the Bessel capacity $C^{\BBR^{N-1}}_{2-\frac{2+\ga_+}{2q'},q'}$ where $q'=\frac{q}{q-1}$, that is
 \begin{equation}\label{N3}\BA {lll}
\forall E\subset\prt\Gw,\,E\text { Borel }, C^{\BBR^{N-1}}_{2-\frac{2+\ga_+}{2q'},q'}(E)=0\Longrightarrow \gm(E)=0.
\EA\end{equation}
\end{theorem}

The striking aspect of the proof is that it is based upon potential estimates which have been developed by Marcus and V\'eron in the study of the supercritical boundary trace problem in polyhedral domains \cite{MV2}.
Before proving this result we need a key potential estimate.
\begin{theorem}\label{poten} Assume $0<\gk\leq\frac{1}{4}$ and $q\geq q_c$. There exists a constant $c_{59}>1$ dependning on $\Gw$, $q$, and $\gk$ such that for any $\gm\in\mathfrak M^+(\prt\Gw)$ there holds
 \begin{equation}\label{N4}\BA {lll}
\myfrac{1}{c_{59}}\|\gm\|^q_{B^{-2+\frac{2+\ga_+}{2q'},q}}\leq \myint{\Gw}{}\left(\BBK_{\CL_\gk}[\gm]\right)^q\phi_\gk dx
\leq c_{59}\|\gm\|^q_{B^{-2+\frac{2+\ga_+}{2q'},q}}
\EA\end{equation}
\end{theorem}
\begin{proof} {\it Step 1: local estimates}. Denote by $\xi=(\xi_1,\xi')$ the coordinates in $\BBR_+^{N}$, $\xi_1>0$, $\xi'\in\BBR^{N-1}$
The ball of radius $R>0$ and center $a$ in $\BBR^{N-1}$ is denoted by $B'_{R}(a)$ (by $B'_{R}$ if $a=0$).
Let $R>0$, $\gn\in \mathfrak M^+(\BBR_+^{N-1})$ with support in $B'_{\frac R2}$ and
 \begin{equation}\label{N5}\BA {lll}
{\bf K}[\gn](\xi)=\myint{B'_{\frac R2}}{}\myfrac{d\gn(\gz')}{(\xi^2_1+|\xi'-\gz'|^2)^{\frac{N-2+\ga_+}{2}}}
\EA\end{equation}
Then, by \cite[Th 3.1]{MV2},
 \begin{equation}\label{N6}\BA {lll}
\frac1{c_{60}}\|\gm\|^q_{B^{-2+\frac{2+\ga_+}{2q'},q}}\leq\myint{0}{R}\myint{B'_{R}}{}\xi_1^{(q+1)\frac{\ga_+}{2}}\left(\myint{B'_{\frac R2}}{}\myfrac{d\gn(\gz')}{(\xi^2_1+|\xi'-\gz'|^2)^{\frac{N-2+\ga_+}{2}}}\right)^{q}d\xi'd\xi_1
\\[4mm]\phantom{-------\frac1{c_{60}}\|\gm\|^q_{B^{-2+\frac{2+\ga_+}{2q'},q}}\leq\myint{0}{R}\myint{B'_{R}}{}}
\leq c_{60}\left(1+R^{(q+1)\frac{\ga_+}{2}}\right)
\|\gm\|^q_{B^{-2+\frac{2+\ga_+}{2q'},q}}.
\EA\end{equation}
There exists $R>0$ such that for  any $y_0\in\prt\Gw$, there exists a $C^2$ diffeomorphism $\Gth:=\Gth_{y_0}$ from
$B_{R}(y_0)$ into $\BBR^N$ such that $\Gth(y_0)=0$, $\Gth_{y_0}(B_{R}(y_0))=B_{R}$  and
$$\Gth(\Gw\cap B_{R}(y_0))=B^+_{R}:=B_{R}\cap \BBR_+^{N}\,,\;
\Gth(\prt\Gw\cap B_{\frac R2}(y_0))=B'_{\frac R2}\,,\;\Gth(\prt\Gw\cap B_{R}(y_0))=B'_{R}.
$$
Moreover, $\Gth$ has bounded distortion, in the sense that since
$$\gf_\gk(x)\myint{\prt\Gw\cap B_R(y_0)}{}\myfrac{d\gm(z)}{|x-z|^{N-2+\ga_+}}=
\gf_\gk\circ\Gth^{-1}(\xi)\myint{B'_{R}}{}\myfrac{d(\gm\circ\Gth^{-1})(\gz)}{|\Gth^{-1}(\xi)-\Gth^{-1}(\gz)|^{N-2+\ga_+}},
$$
there holds

$$\BA {lll}
\myfrac{\xi_1^{\frac{\ga_+}{2}}}{c_{61}}\myint{B'_\frac R2}{}
\myfrac {d(\gm\circ\Gth^{-1})(\gz)}{(\xi^2_1+|\xi'-\gz'|^2)^{\frac{N-2+\ga_+}{2}}}\\[4mm]
\phantom{-------}\leq
\gf_\gk\circ\Gth^{-1}(\xi)\myint{B'_\frac R2}{}\myfrac{d(\gm\circ\Gth^{-1})(\gz)}{|\Gth^{-1}(\xi)-\Gth^{-1}(\gz)|^{N-2+\ga_+}}\\[4mm]
\phantom{-------------------}\leq
c_{61}\xi_1^{\frac{\ga_+}{2}}\myint{B'_\frac R2}{}
\myfrac {d(\gm\circ\Gth^{-1})(\gz)}{(\xi^2_1+|\xi'-\gz'|^2)^{\frac{N-2+\ga_+}{2}}}
\EA$$
Since $\gm\mapsto \gm\circ\Gth^{-1}$ is a $C^2$ diffeomorphism between
$\mathfrak M^+(\prt\Gw\cap B_{\frac{R}{2}}(y_0))\cap B^{-2+\frac{2+\ga_+}{2q'},q}(\prt\Gw\cap B_{\frac{R}{2}}(y_0))$ and
$\mathfrak M^+(B'_{\frac{R}{2}})\cap B^{-2+\frac{2+\ga_+}{2q'},q}(B'_{\frac{R}{2}})$, we derive, using (\ref{poissonest}) and
(\ref{N6})
 \begin{equation}\label{N7}\BA {lll}
\frac{1}{c_{62}}\|\gm\|^q_{B^{-2+\frac{2+\ga_+}{2q'},q}}
\leq \myint{\Gw\cap B_R(y_0)}{}(\BBK_{\CL_\gk}[\gm])^q\phi_\gk dx\leq c_{62}\|\gm\|^q_{B^{-2+\frac{2+\ga_+}{2q'},q}}
\EA\end{equation}
Clearly the left-hand side inequality (\ref{N4}) follows. Combining Harnack inequality and boundary Harnack inequality we obtain
 \begin{equation}\label{N8}\BA {lll}
\myint{\Gw}{}(\BBK_{\CL_\gk}[\gm])^q\phi_\gk dx\leq c_{63}\myint{\Gw\cap B_R(y_0)}{}(\BBK_{\CL_\gk}[\gm])^q\phi_\gk dx
\EA\end{equation}
which implies the  left-hand side inequality (\ref{N4}) when $\gm$ has it support in a ball $B_{\frac R2}(y_0)\cap\partial\xO.$\smallskip

\noindent {\it Step 2: global estimates}. We write $\gm=\sum_{j=1}^{j_0}\gm_j$ where the $\gm_j$ are positive measures on
$\prt\Gw$ with support in some ball $B_{\frac R2}(y_j)$ with $y_j\in\prt\Gw$ and such that
$$\frac{1}{c_{64}}\|\gm\|_{B^{-2+\frac{2+\ga_+}{2q'},q}}\leq \|\gm_j\|_{B^{-2+\frac{2+\ga_+}{2q'},q}}\leq c_{64}\|\gm\|_{B^{-2+\frac{2+\ga_+}{2q'},q}}. $$
Then
$$
\|\BBK_{\CL_\gk}[\gm]\|_{L^q_{\phi_\gk}}\leq \sum_{j=1}^{j_0}\|\BBK_{\CL_\gk}[\gm_j]\|_{L^q_{\phi_\gk}}
\leq c_{59}^{\frac{1}{q}}\sum_{j=1}^{j_0}\|\gm_j\|^q_{B^{-2+\frac{2+\ga_+}{2q'},q}}\leq j_0c_{64}c_{59}^{\frac{1}{q}}\|\gm\|_{B^{-2+\frac{2+\ga_+}{2q'},q}}.
$$
On the opposite side
$$\BA {lll}\|\BBK_{\CL_\gk}[\gm]\|_{L^q_{\phi_\gk}}\geq\max_{1\leq j\leq j_0}\|\BBK_{\CL_\gk}[\gm_j]\|_{L^q_{\phi_\gk}}\\[4mm]
\phantom{\|\BBK_{\CL_\gk}[\gm]\|_{L^q_{\phi_\gk}}}
\geq \frac{1}{c_{59}^{\frac{1}{q}}}\max_{1\leq j\leq j_0}\|\gm_j\|_{B^{-2+\frac{2+\ga_+}{2q'},q}}
\\[4mm]
\phantom{\|\BBK_{\CL_\gk}[\gm]\|_{L^q_{\phi_\gk}}}
\geq \frac{1}{j_0c_{59}^{\frac{1}{q}}}\sum_{j=1}^{j_0}\|\gm_j\|_{B^{-2+\frac{2+\ga_+}{2q'},q}}
\\[4mm]
\phantom{\|\BBK_{\CL_\gk}[\gm]\|_{L^q_{\phi_\gk}}}
\geq \frac{1}{c_{64}c_{59}^{\frac{1}{q}}}\|\gm\|_{B^{-2+\frac{2+\ga_+}{2q'},q}},
\EA$$
which ends the proof.
\end{proof}
\noindent{\it Proof of Theorem \ref{supcr}: The condition is sufficient.} Let $\gm$ be a boundary measure such that $|\BBK_{\CL_\gk}[\gm]|^q\in L^1_{\ei}(\xO).$ For $k>0$ set $g_k(u)=$ sgn$(u)\min\{|u|^q,k^q\}$ and let $u_k$  be the solution of
 \begin{equation}\label{N9}\BA {lll}
\CL_\gk u_k+g_k(u_k)=0\qquad&\text{in }\Gw\\
\phantom{\CL_\gk u_k+g_k()}
u_k=\gm\qquad&\text{in }\prt\Gw,
\EA\end{equation}
which exists a is unique by Theorem \ref{gen}. Furthermore $k\mapsto u_k$ is decreasing,
$$0\leq u_k\leq \BBK_{\CL_\gk}[\gm]$$
and
$$0\leq g_k(u_k)\leq g_k(\BBK_{\CL_\gk}[\gm])\leq (\BBK_{\CL_\gk}[\gm])^q,$$
and the first terms on the right of the two previous inequalities are integrable for the measure $\gf_\gk dx$ by Theorem \ref{poten}. Finally for any $\eta\in {\bf X}_\gk(\Gw)$, there holds
$$\myint{\Gw}{}\left(u_k\CL_\gk\eta+g_k(u_k)\eta\right)dx=\myint{\Gw}{}\BBK_{\CL_\gk}[\gm]\CL_\gk\eta dx.
$$
Since $u_k$ and $g_k(u_k)$ converge respectively to $u$ and $g(u) $ a.e. and in $L^1_{\ei}(\Gw)$; we conclude that
$$\myint{\Gw}{}\left(u\CL_\gk\eta+u^q\eta\right)dx=\myint{\Gw}{}\BBK_{\CL_\gk}[\gm]\CL_\gk\eta dx.
$$
If $\gm$ is a positive measure which vanishes on Borel sets $E\subset\prt\Gw$ with $C^{\BBR^{N-1}}_{2-\frac{2+\ga_+}{2q'},q'}$-capacity zero, there exists an increasing sequence of positive measures in $B^{-2+\frac{2+\ga_+}{2q'},q}(\prt\Gw)$ $\{\gm_n\}$ which converges to $\gm$ (see \cite{DaM}, \cite {FDeP}). Let $u_{\gm_n}$ be the solution of (\ref{N1})  with boundary data
$\gm_n$. The sequence $\{u_{\gm_n}\}$ is increasing with limit $u$. Since, by taking $\phi_\gk$ as test function, we obtain
$$\myint{\Gw}{}\left(\gl_\gk u_{\gm_n}+g(u_{\gm_n})\right)\gf_\gk dx=\gl_\gk\myint{\Gw}{}\BBK_{\CL_\gk}[\gm_n]\gf_\gk dx,
$$
it follows that $u,g(u)\in L^1_{\gf_\gk}(\xO)$. Thus
$$\myint{\Gw}{}\left(u\CL_\gk\eta+g(u)\eta\right)dx=\myint{\Gw}{}\BBK_{\CL_\gk}[\gm]\CL_\gk\eta dx\quad\forall \eta\in {\bf X}_\gk(\Gw),
$$
and therefore $u=u_\gm.$\medskip

\noindent{\bf Definition} A smooth lifting is a continuous linear operator $R[.]$ from  $C^2_0(\prt\Gw)$ to $C^2_0(\overline\Gw)$ satisfying
 \begin{equation}\label{N10}\BA {ll}
(i)\qquad 0\leq\eta\leq 1\Longrightarrow 0\leq R[\eta]\leq 1\,,\; R[\eta]\lfloor_{\prt\Gw}=\eta\\[2mm]
(ii)\qquad  |\nabla \gf_\gk.\nabla R[\eta]|\leq c_{65}\gf_\gk
 \EA
\end{equation}
 where $c_{65}$ depends on the $C^1$-norm of $\eta$.\medskip

Our proof are based upon modification of an argument developed by Marcus and V\'eron in \cite {MV-JMPA01}.
\begin{lemma}\label{lift} Assume there exists a solution $u_\gm$ of (\ref{N1}) with $\gm\geq 0$. For $\eta\in C^2(\Gw)$, $0\leq\eta\leq 1$ set $\gz=\gf_\gk (R[\eta])^{q'}$ where $R$ is a smooth lifting. Then
 \begin{equation}\label{N10-1}
\left(\myint{\prt\Gw}{}\eta d\gm\right)^{q'}\leq c_{67}
\myint{\Gw}{}u^q\gz dx+c_{67}\left(\myint{\Gw}{}u^q\gz dx\right)^{\frac{1}{q}}\left(\left(\myint{\Gw}{}\gf_\gk dx\right)^{\frac{1}{q'}}+q'\left(\myint{\Gw}{}(L[\eta])^{q'}dx\right)^{\frac{1}{q'}}\right)
\end{equation}
where
 \begin{equation}\label{N10-2}
 L[\eta]=(R[\eta])^{q'-1}\left(2\gf_\gk^{-\frac{1}{q}}|\nabla\gf_\gk.\nabla R[\eta]|+\gf_\gk^{\frac{1}{q'}}|\Gd R[\eta]|\right)
\end{equation}
and $c_{67}$ depends on $\xO,\xl_\xk,q,\xk,N.$
\end{lemma}
\begin{proof}  There holds
$$\CL_\gk\gz=\gl_\gk(R[\eta])^{q'}\gf_\gk-2q'(R[\eta])^{q'-1}\nabla\gf_\gk.\nabla R[\eta]-q'(R[\eta])^{q'-2}\gf_\gk\left(R[\eta]\Gd R[\eta]-(q'-1)|\nabla R[\eta]|^2\right).
$$
Then $\gz\in{\bf X}_\gk(\Gw)$ because of (\ref{N10})-(ii) and by Proposition \ref{remark2}
$$c_{66}\myint{\prt\Gw}{}\eta^{q'}d\gm\leq \myint{\Gw}{}\left(u\CL_\gk\gz+u^q\gz\right)dx.
$$
Since
$$u\CL_\gk\gz\leq u\left(\gl_\gk(R[\eta])^{q'}\gf_\gk+2q'(R[\eta])^{q'-1}|\nabla\gf_\gk.\nabla R[\eta]|+q'(R[\eta])^{q'-1}\gf_\gk|\Gd R[\eta]|\right)
$$
we obtain
$$\myint{\Gw}{}u\CL_\gk\gz dx\leq \left(\myint{\Gw}{}u^q\xz dx\right)^{\frac{1}{q}}
\left(\left(\myint{\Gw}{}\gf_\gk dx\right)^{\frac{1}{q'}}+q'\left(\myint{\Gw}{}(L[\eta])^{q'}dx\right)^{\frac{1}{q}}\right),
$$
where $L[\eta]$ is defined by (\ref{N10-2}).

\end{proof}

\begin{lemma}\label{heat} There exist a smooth lifting $R$ such that $\eta\mapsto L[\eta]$ is continuous from $B^{2-\frac{2+\ga_+}{2q'},q'}(\prt\Gw)$ into $L^{q'}(\Gw)$. Furthermore,
 \begin{equation}\label{N10-1'}
\|L[\eta]\|_{L^{q'}(\Gw)}\leq c'_{66}\|\eta\|^{q'-1}_{L^\infty(\prt\Gw)}\|\eta\|_{B^{2-\frac{2+\ga_+}{2q'},q'}(\prt\Gw)}.
\end{equation}
\end{lemma}
\begin{proof} The construction of the lifting is originated into \cite[Sect 1]{MV}. For $0<\gd\leq \gb_0$, we set
$\Gs_\gd=\{x\in\Gw:d(x)=\gd\}$ and we identify $\prt\Gw$ with $\Gs:=\Gs_0$. The set $\{\Gs_\gd\}_{0<\gd\leq \gb_0}$ is a smooth foliation of $\prt\Gw$. For each $\gd\in (0,\gb_0]$ there exists a unique $\gs(x)\in \Gs_\gd$ such that $d(x)=\gd$ and $|x-\gs(x)|=\gd$. The set of couples $(\gd,\gs)$ defines a system of coordinates in $\Gw_{\gb_0}$ called the flow coordinates. The Laplacian obtain the following expression in this system
 \begin{equation}\label{N11}
\Gd=\myfrac{\prt^2}{\prt\gd^2}+b_0\myfrac{\prt}{\prt\gd}+\Gl_\Gs
\end{equation}
where $\Gl_\Gs$ is a linear second-order elliptic operator on $\Gs$ with $C^1$ coefficients. Furthermore
$b_0\to K$ and  $\Gl_\Gs\to \Gd_\Gs$, where $K$ is the mean curvature of $\Gs$ and $\Gd_\Gs$ the Laplace-Beltrami operator on $\Gs$. If $\eta\in B^{-2+\frac{2+\ga_+}{2q'},q}(\prt\Gw)$, we denote by $H:=H[\eta]$ the solution of
 \begin{equation}\label{N12}\BA {lll}
\phantom{(0,.)}
\myfrac{\prt H}{\prt s}+\Gd_\Gs H=0\qquad&\text{in }(0,\infty)\times\Gs\\
\phantom{\myfrac{\prt H}{\prt s}+\Gd_\Gs}
H(0,.)=\eta\qquad&\text{in }\Gs
\EA\end{equation}
Let $h\in C^\infty(\BBR_+)$ such that $0\leq h\leq 1$, $h'\leq 0$, $h\equiv 1$ on $[0,\frac{\gb_0}{2}]$, $h\equiv 0$ on
$[\gb_0,\infty]$. The lifting we consider is expressed by
  \begin{equation}\label{N13}
R[\eta](x)=\left\{\BA {lll}H[\eta](\gd^2,\gs(x))h(\gd)\qquad&\text{if }x\in \overline\Gw_{\gb_0}\\
0&\text{if }x\in \Gw'_{\gb_0},
\EA\right.
\end{equation}
with $x\approx (\gd,\gs):=(d(x),\gs(x)$. Mutatis mutandis, we perform the same computation as the one in \cite[Lemma 1.2]{MV-JMPA01}, using local coordinates $\{\gs_j\}$ on $\Gs$ and obtain
$$\nabla R[\eta]=2\gd h(\gd)\frac{\prt H}{\prt\gd}(\gd^2,\gs)\nabla\gd+\sum_{j=1}^{N-1}h(\gd)\frac{\prt H}{\prt\gs_j}(\gd^2,\gs)\nabla\gs_j+h'(\gd) H(\gd^2,\gs)\nabla\gd
$$
In $\Gw_{\frac{\gb_0}{2}}$ there holds
  \begin{equation}\label{N14}\nabla R[\eta].\nabla\gf_\gk=2\gd h(\gd)\frac{\prt H}{\prt\gd}(\gd^2,\gs)\nabla\gf_\gk.\nabla\gd+\sum_{j=1}^{N-1}h(\gd)\frac{\prt H}{\prt\gs_j}(\gd^2,\gs)\nabla\gs_j.\nabla\ei+h'(\gd) H(\gd^2,\gs)\nabla\gd.\nabla \ei
\end{equation}
Moreover $\gf_\gk(x)\leq c_2(d(x))^{\frac{\ga_+}{2}}=c_2\gd^{\frac{\ga_+}{2}}$ and $|\nabla \gf_\gk(x)|\leq c'_2(d(x))^{\frac{\ga_+}{2}-1}=c'_2\gd^{\frac{\ga_+}{2}-1}$. Similarly as in  \cite[(1.13)]{MV-JMPA01}
$$\nabla\gf_\gk=\frac{\prt\gf_\gk}{\prt\gd}{\nabla d}+\sum_{j=1}^{N-1}\frac{\prt\gf_\gk}{\prt\gs_j}(\gd^2,\gs)\nabla\gs_j,$$
thus
$$|\nabla\gf_\gk.\nabla\gs_j|\leq c_{68}\gd^{\frac{\ga_+}{2}},
$$
$$\gf^{-\frac{1}{q}}_\gk|\nabla R[\eta].\nabla\gf_\gk| \leq c_{69}\gd^{\frac{\ga_+}{2q'}}\left(\left|\frac{\prt H}{\prt\gd}(\gd^2,\gs)\right|+\sum_{j=1}^{N-1}\left|\frac{\prt H}{\prt\gs_j}(\gd^2,\gs)\right|-\frac{h'(\gd)}{\xd} H(\gd^2,\gs)\right).
$$
Thus
$$\BA {ll}\myint{\Gw}{}\gf^{-\frac{q'}{q}}_\gk|\nabla R[\eta].\nabla\gf_\gk|^{q'}dx
\leq c_{70}\myint{\Gw_{\gb_0}}{}\gd^{\frac{\ga_+}{2}}\left|\myfrac{\prt H}{\prt\gd}(\gd^2,\gs)\right|^{q'}dx\\[4mm]
\phantom{\myint{\Gw}{}\gf^{-\frac{q'}{q}}_\gk|\nabla R[\eta].\nabla\gf_\gk|^{q'}dx}
\displaystyle
+ c_{70}\sum_{j=1}^{N-1}\myint{\Gw_{\gb_0}}{}\gd^{\frac{\ga_+}{2}}\left|\frac{\prt H}{\prt\gs_j}(\gd^2,\gs)\right|^{q'}dx
\\[4mm]
\phantom{\myint{\Gw}{}\gf^{-\frac{q'}{q}}_\gk|\nabla R[\eta].\nabla\gf_\gk|^{q'}dx}
+c_{70}\myint{\Gw_{\gb_0}\setminus\Gw_{\frac{\gb_0}{2}}}{}\gd^{\frac{\ga_+}{2}}H^{q'}(\gd^2,\gs)dx
\EA$$
Then
  \begin{equation}\label{N15}\BA {ll}
\myint{\Gw}{}\gf^{-\frac{q'}{q}}_\gk|\nabla R[\eta].\nabla\gf_\gk|^{q'}dx
\leq c_{71}\myint{0}{\gb_0}\gd^{\frac{\ga_+}{2}}\myint{\Gs}{}\left|\myfrac{\prt H}{\prt \xd}(\gd^2,\gs)\right|^{q'}dSd\gd\\[4mm]
\phantom{\myint{\Gw}{}\gf^{-\frac{q'}{q}}_\gk|\nabla R[\eta].\nabla\gf_\gk|^{q'}dx}
\leq c_{71}\myint{0}{\gb^2_0}\myint{\Gs}{}\left(t^{\frac{2+\ga_+}{4q'}}\left\|\myfrac{\prt H}{\prt t}(t,.)\right\|_{L^{q'}(\Gs)}\right)^{q'}\myfrac{dt}{t}
\\[4mm]
\phantom{\myint{\Gw}{}\gf^{-\frac{q'}{q}}_\gk|\nabla R[\eta].\nabla\gf_\gk|^{q'}dx}
\leq c_{72}\|\eta\|^{q'}_{B^{2-\frac{2+\ga_+}{2q'},q'}(\Gs)}
\EA\end{equation}
by using the classical real interpolation identity
  \begin{equation}\label{N16}
  \left[W^{2,q'}(\Gs),L^{q'}(\Gs)\right]_{1-\frac{2+\ga_+}{4q'},q'}=B^{2-\frac{2+\ga_+}{2q'},q'}(\Gs).
  \end{equation}
  Similarly (see \cite[(1.17),(1.19)]{MV-JMPA01})
  \begin{equation}\label{N17}\BA {ll}\displaystyle
\sum_{j=1}^{N-1}\myint{\Gw_{\gb_0}}{}\gd^{\frac{\ga_+}{2}}\left|\frac{\prt H}{\prt\gs_j}(\gd^2,\gs)\right|^{q'}dx
+\myint{\Gw_{\gb_0}\setminus\Gw_{\frac{\gb_0}{2}}}{}\gd^{\frac{\ga_+}{2}}H^{q'}(\gd^2,\gs)dx
  \leq c_{72}\|\eta\|^{q'}_{W^{2-\frac{2+\ga_+}{2q'},q'}(\Gs)}.
  \EA\end{equation}
  Next we consider the second  term. Adapting in a straightforward manner the computation in \cite[p. 886-887 ]{MV-JMPA01} we obtain the following instead of \cite[(1.21)]{MV-JMPA01}
  \begin{equation}\label{N18}\BA {ll}\displaystyle
  \myint{\Gw}{}\gf_\gk|\Gd R[\eta]|^{q'}dx \leq
  c_{72}\myint{0}{\gb_0}\myint{\Gs}{}\left|\gd^{2+\frac{\ga_+}{2q'}}\myfrac{\prt^2H[\eta]}{\prt\gd^2}\right|^{q'}\!\!\!(\gd^2,\gs)d\gs d\gd\\[4mm]
  \phantom{  \myint{\Gw}{}\gf_\gk^{\frac{q'}{q}}|\Gd R[\eta]|^{q'}dx}+
  c_{72}\myint{0}{\gb_0}\myint{\Gs}{}\gd^{\frac{\ga_+}{2}}\left(  \left|\myfrac{\prt H[\eta]}{\prt\gd}\right|^{q'}+|H|^{q'}+|\Gl_\Gd-\Gl_\Gs|^{q'}\right)(\gd^2,\gs)dx
  \EA\end{equation}
  Then
    \begin{equation}\label{N19}\BA {ll}\displaystyle
 \myint{0}{\gb_0}\myint{\Gs}{}\left|\gd^{2+\frac{\ga_+}{2q'}}\myfrac{\prt^2H[\eta]}{\prt\gd^2}\right|^{q'}\!\!\!(\gd^2,\gs)d\gs d\gd
=\myint{0}{\gb^2_0}\myint{\Gs}{}\left|t^{2\left(1-\frac{4q'-\ga_+-2}{8q'}\right)}\myfrac{\prt^2H[\eta]}{\prt t^2}\right|^{q'}d\gs \frac{d t}{t}\\[2mm]
  \phantom{ \myint{0}{\gb_0}\myint{\Gs}{}\left|\gd^{2+\frac{\ga_+}{2q'}}\myfrac{\prt^2H[\eta]}{\prt\gd^2}\right|^{q'}(\gd^2,\gs)d\gs d\gd}
    \leq c_{73}\|\eta\|^{q'}_{B^{2-\frac{2+\ga_+}{2q'},q'}(\Gs)},
    \EA\end{equation}
by using the real interpolation identity
  \begin{equation}\label{N20i}
  \left[W^{4,q'}(\Gs),L^{q'}(\Gs)\right]_{\frac{4q'-\ga_+-2}{8q'},q'}=B^{2-\frac{2+\ga_+}{2q'},q'}(\Gs).
  \end{equation}
  The other term in the right-hand side of (\ref{N18}) yields to the same inequality as in (\ref{N17}).
\end{proof}

\noindent{\it Proof of Theorem \ref{supcr}: The condition is necessary.} Let $K\subset\prt\Gw$ be a compact set
and $\eta\in C^2_0(\prt\Gw)$ such that $0\leq\eta\leq 1$ and $\eta=1$ on $K$. Then, by (\ref{N10-1})
  \begin{equation}\label{N20}\BA {lll}
  (\gm(K))^{q'}\leq c_{67}\myint{\Gw}{}u^q(R[\eta])^{q'}\phi_\gk dx+\\[4mm]
  \phantom{-\myint{\Gw}{}u^q(R[\eta])^{q'}}
  c_{67}\left(\myint{\Gw}{}u^q(R[\eta])^{q'}\phi_\gk dx\right)^{\frac{1}{q}}\left(\left(\myint{\Gw}{}\gf_\gk dx\right)^{\frac{1}{q'}}+c'_{66}q'\|\eta\|_{B^{2-\frac{2+\ga_+}{2q'},q'}(\prt\Gw)}\right).
  \EA\end{equation}
From this inequality, we obtain classically the result since if $C^{\BBR^{N-1}}_{2-\frac{2+\ga_+}{2q'},q'}(K)=0$ there exists a sequence $\{\eta_n\}$ in $C^2_0(\prt\Gw)$ with the following properties:
  \begin{equation}\label{N20-1}\BA {lll}
  0\leq\eta_n\leq 1\,,\; \eta_n=1\; \text{in a neighborhood of $K$ and $\eta_n\to 0$ in } B^{2-\frac{2+\ga_+}{2q'},q'}(\prt\Gw)
  \text{ as } n\to\infty.
\EA\end{equation}
This implies that $u^q(R[\eta_n])^{q'}\to 0$ in $L^1_{\gf_\gk}(\Gw)$. Therefore the right-hand side of (\ref{N20}) tends to $0$ if we substitute  $\eta_n$ to $\eta$ and thus $\gm(K)=0$ for any $K$ compact with zero capacity and this relation holds for any Borel subset.
$\hfill$
$\Box$ \medskip

\noindent {\bf Definition}. We say that a compact set $K\subset\prt\Gw$ is {\it removable} if any positive solution $u\in C(\overline\Gw\setminus K)$ of
 \begin{equation}\label{N21}\BA {lll}
 \CL_\gk u+|u|^{q-1}u=0\qquad\text{in }\Gw\\
   \EA\end{equation}
such that
 \begin{equation}\label{N21-0}\BA {lll}
\myint{\Gw}{} (u\CL_\gk\eta +|u|^{q-1}u\eta)dx=0\quad\forall\eta\in {\bf X}^K_{\gk}(\Gw)
   \EA\end{equation}
where ${\bf X}^K_{\gk}(\Gw)=\{\eta\in {\bf X}_{\gk}(\Gw):\eta=0 \text{ in a neighborhood of } K\}$,
  is identically zero.
\begin{theorem}\label{remov} Assume $0<\gk\leq\frac{1}{4}$ and $q\geq 1$. A compact set $K\subset\prt\Gw$ is removable if and only if
$C^{\BBR^{N-1}}_{2-\frac{2+\ga_+}{2q'},q}(K)=0$.
\end{theorem}
\begin{proof} The condition is clearly necessary since, if a compact boundary set $K$ has positive capacity, there exists a capacitary measure $\gm_k\in \mathfrak M_+(\prt\Gw)\cap B^{-2+\frac{2+\ga_+}{2q'},q}(\prt\Gw)$ with support in $K$ (see e.g. \cite{AH}). For such a measure there exists a solution $u_{\gm_K}$ of (\ref{N1}) with $\gm=\gm_K$ by Theorem \ref{supcr}. Next we assume that $C^{\BBR^{N-1}}_{2-\frac{2+\ga_+}{2q'},q}(K)=0$. Then there exists a sequence $\{\eta_n\}$ in $C^2_0(\prt\Gw)$ satisfying (\ref{N20-1}). In particular, there exists a decreasing sequence $\{\CO_n\}$ of relatively open subsets of $\prt\Gw$, containing $K$ such that $\eta_n=1$ on $\CO_n$ and thus $\eta_n=1$ on $K_n:=\overline\CO_n$.
We set $\tilde\eta_n=1-\eta_n$ and $\tilde \gz_n=\gf_\gk( R[\tilde\eta_n])^{2q'}$ where $R$ is defined by (\ref{N13}).
Then $0\leq\tilde\eta_n\leq 1$  and $\tilde\eta_n=0$ on $K_n$. Therefore
 \begin{equation}\label{N22}\BA {lll}
\tilde \gz_n(x)\leq \gf_\gk\min\left\{1, c_{74}(d(x))^{1-N}e^{-(4d(x))^{-2}(\dist (x,K^c_n))^2}\right\}
   \EA\end{equation}
   Furthermore
    \begin{equation}\label{N22-1}\BA {lll}
(i)\qquad &|\nabla R[\tilde\eta_n]|\leq c_{75}\min\left\{1, (d(x))^{-2-N}e^{-(4d(x))^{-2}(\dist (x,K^c_n))^2}\right\}\\[4mm]
(ii)\qquad &|\Gd R[\tilde\eta_n]|\leq c_{75}\min\left\{1, (d(x))^{-4-N}e^{-(4d(x))^{-2}(\dist (x,K^c_n))^2}\right\}
   \EA\end{equation}
\smallskip

\noindent {\it Step 1}. We claim that
 \begin{equation}\label{N23}\BA {lll}
\myint{\Gw}{}\left(u\CL_\gk\tilde \gz_n+u^q\tilde\gz_n\right) dx=0.
   \EA\end{equation}
  By Proposition \ref{prop19} there exists $c_{74}>0$ such that
   \begin{equation}\label{N24}\BA {lll}
(i)\qquad &\phantom{|\nabla|}u(x)\leq c_{76}(d(x))^{\frac{\ga_+}{2}}(\dist(x,K))^{-\frac{2}{q-1}-\frac{\ga_+}{2}}\\
(ii)\qquad &|\nabla u(x)|\leq c_{76}(d(x))^{\frac{\ga_+}{2}-1}(\dist(x,K))^{-\frac{2}{q-1}-\frac{\ga_+}{2}}
   \EA\end{equation}
   for all $x\in\Gw$. As in the proof of Lemma 3.8,
   \begin{equation}\label{N25-}\BA {lll}
   |u\CL_\gk\tilde\gz_n|\leq c_{77}(R[\tilde\eta_n])^{2q'-2}u\left(\gf_\gk R^2[\tilde\eta_n]+R[\tilde\eta_n]|\nabla\gf_\gk.\nabla R[\tilde\eta_n]|\right.\\[2mm]
   \phantom{----------------}\left.+\gf_\gk(R[\tilde\eta_n]|\Gd R[\tilde\eta_n]|+|\nabla R[\tilde\eta_n]|^2)\right).
   \EA\end{equation}

Let $\CO$ be a relatively open neighborhood of $K$ such that $\overline \CO\subset \CO_n$. We set
$G_{\CO,\gb_0}=\{x\in\Gw_{\gb_0}:\gs(x)\in \CO\}$ and $ G_{\CO^c,\gb_0}=\Gw_{\gb_0}\setminus G_{\CO}$. If $x\in G_{\CO}$,
$\dist (x,K^c_n)\geq \gt>0$. Then, by (\ref{N24})-(i) and (\ref{N22}), $u^q\tilde\gz_n\in L^q( G_{\CO})$. Since $u(x)=\circ (W(x))$ in
$G_{\CO^c}$ it follows that $u^q\tilde\gz_n\in L^1(\Gw_{\gb_0})$ and thus $u^q\tilde\gz_n$ is integrable in $\Gw$ . Similarly, using ({N22-1})-(i) and (ii), $u\CL_\gk\tilde\gz_n\in L^1(\Gw)$. Since $\tilde\gz_n$ does not vanish in a neighborhood of $K$, we introduce a cut-off function  $\gth_\ge\in C^2(\overline\Gw)$ for $0<\ge\leq\frac{\gb_0}{2}$, with the following properties,
$$\BA{lll}0\leq\gth_\ge\leq 1\,,\;\gth_\ge(x)=0 \;\forall x\in G_{\CO,\ge}\,,\;\gth_\ge(x)=1 \;\forall x\in \overline\Gw\,\text{ s.t. } \;
\dist (x,G_{\CO,\ge})\geq\ge\\[2mm]
|\nabla\gth_\ge|\leq c_{78}\ge^{-1}\chi_{G_{\CO_\ge,\ge}\setminus G_{\CO,\ge}}\,\text{ and }\;|D^2\gth_\ge|\leq c_{78}\ge^{-2}
\chi_{G_{\CO_\ge,\ge}\setminus G_{\CO,\ge}},
\EA$$
where we have taken $\ge$ small enough so that
$$G_{\CO_\ge,\ge}:=\{x\in\Gw:\dist (x,G_{\CO,\ge})\leq\ge\}\subset G_{K_n,2\ge}=\{x\in\Gw_{2\ge}:\gs(x)\in K_n\}.$$
Clearly $\gth_\ge\tilde\gz_n\in {\bf X}_\gk^K(\Gw)$, thus
$$\myint{\Gw}{}\left(u\CL_\gk(\gth_\ge\tilde\gz_n)+u^q\gth_\ge\tilde\gz_n\right)dx=0.
$$
Next
$$\BA {lll}\myint{\Gw}{}\left(u\CL_\gk(\gth_\ge\tilde\gz_n)+u^q\gth_\ge\tilde\gz_n\right)dx=\myint{\Gw\setminus G_{\CO_\ge,\ge} }{}\left(u\CL_\gk(\gz_n)+u^q\tilde\gz_n\right)dx+\myint{G_{\CO_\ge,\ge} }{}\left(u\CL_\gk(\gth_\ge\tilde\gz_n)+u^q\gth_\ge\tilde\gz_n\right)dx\\[4mm]
\phantom{\myint{\Gw}{}\left(u\CL_\gk(\gth_\ge\tilde\gz_n)+u^q\gth_\ge\tilde\gz_n\right)dx}
=I_\ge+II_\ge
\EA$$
Clearly
$$\lim_{\ge\to 0}I_\ge=\myint{\Gw}{}\left(u\CL_\gk\tilde\gz_n+u^q\tilde\gz_n\right)dx$$
and
$$\lim_{\ge\to 0}\myint{G_{\CO_\ge,\ge} }{}u^q\gth_\ge\tilde\gz_n dx=0.
$$
Finally, since $\CL_\gk(\gth_\ge\tilde\gz_n)=\gth_\ge\CL_\gk\tilde\gz_n+\tilde\gz_n\Gd\gth_\ge+2\nabla\gth_\ge.\nabla \tilde\gz_n$, $\gth_\ge$ is constant outside $ G_{\CO_\ge,\ge}\setminus G_{\CO,\ge}$ and
$\dist (G_{\CO_\ge,\ge}\setminus G_{\CO,\ge},F_n^c)\geq \gt>0$, independent of $\ge$
 there holds, by (\ref{N22})
$$|\CL_\gk(\gth_\ge\tilde\gz_n)|\leq c_{79}\ge^{-{N+4}}e^{-\frac{\gt}{\ge^2}}.
$$
Using (\ref{N24})-(i) we derive
$$\lim_{\ge\to 0}\myint{G_{\CO_\ge,\ge} }{}u\CL_\gk(\gth_\ge\tilde\gz_n)dx=0,
$$
which yields to (\ref{N23}).\smallskip

\noindent {\it Step 2}. We claim that
 \begin{equation}\label{N26}\BA {lll}
\myint{\Gw}{}u^q\gf_\gk dx<\infty.
   \EA\end{equation}
   Using the expression of $\CL_\gk\gz_n$ in (\ref{N23}) where replace $\eta_n$ by $\tilde \eta_n$, we derive
   \begin{equation}\label{N25}\BA {lll}
   \myint{\Gw}{}u^q\tilde \gz_ndx= \myint{\Gw}{}\left(-\gl_\gk(R[\tilde\eta_n])^{2q'}\gf_\gk+4q'(R[\tilde\eta_n])^{2q'-1}\nabla\gf_\gk.\nabla R[\tilde\eta_n]+\right.
   \\[4mm]\phantom{-----\myint{\Gw}{}u^q\gz_ndx}2q'(R[\tilde\eta_n])^{2q'-2}\gf_\gk\left(R[\tilde\eta_n]\Gd R[\tilde\eta_n]
   \left.+(2q'-1)|\nabla R[\tilde\eta_n]|^2\right)\right)u dx
   \\[4mm]\phantom{\myint{\Gw}{}u^q\gz_ndx}
   \leq c_{79} \left(\myint{\Gw}{}u^q\tilde \gz_ndx\right)^{\frac{1}{q}} \left(\myint{\Gw}{}(\tilde L[\eta_n])^{q'}dx\right)^{\frac{1}{q'}},
   \EA\end{equation}
where we have set
   \begin{equation}\label{N25+}\BA {lll}
\tilde L[\eta]=(\gf_{\gk})^{-\frac{1}{q}}\nabla\gf_\gk.\nabla R[\eta_n]+(\gf_{\gk})^{\frac{1}{q'}}|\Gd R[\tilde\eta_n]|+
(\gf_{\gk})^{\frac{1}{q'}}|\nabla R[\tilde\eta_n]|^2
   \EA\end{equation}
   By Lemma \ref{heat} we know that
      \begin{equation}\label{N26+}\BA {lll}
\myint{\Gw}{}(\gf_{\gk})^{-\frac{q'}{q}}|\nabla\gf_\gk.\nabla R[\eta_n]|^{q'}+\gf_{\gk}|\Gd R[\tilde\eta_n]|^{q'}dx\leq (c_{72}+c_{73})
\|\eta_n\|^{q'}_{B^{2-\frac{2+\ga_+}{2q'},2}(\prt\Gw)}.
   \EA\end{equation}
   The last term is estimated in the following way
         \begin{equation}\label{N27}\BA {lll}
\myint{\Gw}{}\gf_{\gk}|\nabla R[\tilde\eta_n]|^{2q'}dx\leq c_{80}\myint{0}{\gb_0^2}\myint{\Gs}{}s^{q'+\frac{\ga_++2}{4}}\left|\myfrac{\prt H[\eta_n]}{\prt s}\right|^{2q'}dS\myfrac{ds}{s}\\[4mm]
\phantom{\myint{\Gw}{}\gf_{\gk}|\nabla H[\tilde\eta_n]|^{2q'}dx\leq c_{80}}
+c_{80}\myint{0}{\gb_0^2}\myint{\Gs}{}s^{\frac{\ga_++2}{4}}\left(|\nabla_\Gs H[\eta_n]|^{2q'}+(H[\eta_n])^{2q'}\right)dS\myfrac{ds}{s},
      \EA\end{equation}
 where $\nabla_\Gs$ denotes the covariant gradient on $\Gs$.     Since the following interpolation identity holds
      $$\left[W^{2,2q'}(\Gs),L^{2q'}(\Gs)\right]_{1-\frac{\ga_++2}{8q'},2q'}=B^{1-\frac{\ga_++2}{4q'},2q'}(\Gs)
      $$
we obtain
$$\myint{0}{\gb_0^2}\myint{\Gs}{}s^{q'+\frac{\ga_++2}{4}}\left|\myfrac{\prt H[\eta_n]}{\prt s}\right|^{2q'}\myfrac{ds}{s}
\leq c_{81}\|\eta_n\|^{2q'}_{B^{1-\frac{\ga_++2}{4q'},2q'}(\Gs)}
$$
By the Gagliardo-Nirenberg inequality
         \begin{equation}\label{N28}\BA {lll}\|\eta_n\|^{2q'}_{B^{1-\frac{\ga_++2}{4q'},2q'}(\Gs)}\leq c_{82}\|\eta_n\|^{q'}_{B^{2-\frac{\ga_++2}{2q'},q'}(\Gs)}
\|\eta\|^{q'}_{L^\infty(\Gs)}=c_{82}\|\eta_n\|^{q'}_{B^{2-\frac{\ga_++2}{2q'},q'}(\Gs)}.
      \EA\end{equation}
By the same inequality
         \begin{equation}\label{N29}\BA {lll}\myint{\Gs}{}\left(|\nabla_\Gs H[\eta_n]|^{2q'}+(H[\eta_n])^{2q'}\right)dS\leq
      c_{82}\|H[\eta_n]\|^{q'}_{L^\infty(\Gs)}\myint{\Gs}{}\left(|\Gd_\Gs H[\eta_n]|^{q'}+(H[\eta_n])^{q'}\right)dS.
      \EA\end{equation}
 Using the estimates on $L[\eta]$ in Lemma \ref{heat} and the fact that $0\leq H[\eta_n]\leq 1$, we conclude that
 $$\myint{0}{\gb_0^2}\myint{\Gs}{}s^{\frac{\ga_++2}{4}}\left(|\nabla_\Gs H[\eta_n]|^{2q'}+(H[\eta_n])^{2q'}\right)dS\myfrac{ds}{s} \leq c_{83}\|\eta_n\|^{q'}_{B^{2-\frac{\ga_++2}{2q'},q'}.(\Gs)}
 $$
It follows from (\ref{N25})
         \begin{equation}\label{N30}\BA {lll}
         \myint{\Gw_{\frac{\gb_0}{2}}}{}u^q(R[\tilde\eta_n])^{2q'}\ei dx
         \leq c_{84}\myint{\Gw_{\gb_0}}{}(\tilde L{\eta_n})^{q'}dx\leq c_{85}\|\eta_n\|^{q'}_{B^{2-\frac{\ga_++2}{2q'},q'}.(\Gs)}
       \EA\end{equation}
       Letting $n\to\infty$ and using the fact that $\eta_n\to 0$, we obtain by Fatou's lemma that
       $$ \myint{\Gw_{\frac{\gb_0}{2}}}{}u^q\gf_\gk dx=0.$$
       Combined with the fact that $u$ is bounded in $\Gw'_{\frac{\gb_0}{2}}$ we obtain (\ref{N26}). Notice that $\|u\|_{L^q_{\gf_\gk}}(\Gw)$ is bounded independently of $u$.
   \smallskip

\noindent {\it Step 3}. End of the proof. Since $u^q\in L^1_{\gf_\gk}(\Gw)$, by Proposition \ref{2.6} there exists a unique weak solution $v\in L^1_{\gf_\gk}(\Gw)$ of
         \begin{equation}\label{R1}\BA {lll}
\CL_\gk v=u^q\qquad&\text{in }\,\Gw\\[1mm]
\phantom{\CL_\gk v}v=0&\text{in }\,\prt\Gw,
   \EA\end{equation}
 and $v\geq 0$.  Then $w=u+v$ is $\CL_\gk$-harmonic in $\Gw$, and by Theorem \ref{Lemm11} there exists a unique positive    Radon measure $\gt$ on $\prt\Gw$ such that $w=\BBK_{\CL_\gk}[\gt]$. Since $v$ and $u$ vanish respectively on  on $\prt\Gw$ and $\prt\Gw\setminus K$, it follows from Propositions \ref{Lemm12} and \ref{trlin} that the support of $\gt$ is included in $K$. By
 Theorem \ref{supcr}, $\gt$ vanishes on Borel subsets with zero $C^{\BBR^{N-1}}_{2-\frac{2+\ga_+}{2q'},q'}$-capacity. Since
 $C^{\BBR^{N-1}}_{2-\frac{2+\ga_+}{2q'},q'}(K)=0$, $\gt=0$. This implies that $u$ is a weak solution of
          \begin{equation}\label{R2}\BA {lll}
\CL_\gk u+u^q=0\qquad&\text{in }\,\Gw\\[1mm]
\phantom{\CL_\gk +u^q}u=0&\text{in }\,\prt\Gw,
   \EA\end{equation}
   and therefore $u=0$.
\end{proof}

\noindent{\bf Remark}. Using the fact that $u^+$ and $u_-$ are subsolutions of (\ref{N21}), it is easy to check  that Theorem \ref{remov} remains valid for any signed solution of (\ref{N21}).\medskip

 \noindent{\bf Remark}. If $1<q<q_c$ (see (\ref{N2})) it follows from Sobolev imbedding theorem that only the empty set has
 zero $C^{\BBR^{N-1}}_{2-\frac{2+\ga_+}{2q'},q'}$-capacity. only the empty set As a consequence of the previous result, if $q\geq q_c$ any isolated boundary singularity of a solution of (\ref{N21}) is removable.

\section{Isolated boundary singularities}
\setcounter{equation}{0}
We denote by $\{{\bf e}_{_1},...,{\bf e}_{_N}\}$ the canonical basis in $\BBR^N=\{x=(x',x_N)\in \BBR^{N-1}\times\BBR\}$ and by $(r,\gs)$ the spherical coordinates therein. Then $\BBR_+^{N}=\{=(x',x_N):, x'\in\BBR^{N-1}, x_N>0\}$ . We although denote by  $S^{N-1}$ and $S_+^{N-1}$ the unit sphere and the upper hemisphere of $\BBR_+^{N}$, i.e.  $S^{N-1}:\cap \BBR_+^{N}$. In this section we study the behavior near $0$ of solutions of
\begin{equation}\label{Eq1}
-\Gd u-\myfrac{\xk }{d^2}u+|u|^{q-1}u=0
\end{equation}
in a bounded convex domain $\Gw$ of $\BBR^N$  with a smooth boundary containing $0$ where $d$ is the distance function to the boundary, $\xk $ a constant in $(0,\frac{1}{4}]$ and $q>1$. Although it is not bounded, the model case is $\Gw=\BBR_+^{N}=\{=(x',x_N):, x'\in\BBR^{N-1}, x_N>0\}$ which is represented by  $(r,\gs)$, $r>0$, $\gs\in S^{N-1}_+$  in spherical coordinates. Then
\begin{equation}\label{Eq1-2}
\CL_{\gk}u=-u_{rr}-\myfrac{N-1}{r}u_r-\myfrac{1}{r^2}\Gd_{S^{N-1}}u-\myfrac{\gk}{r^2({\bf e}_N.\gs)^2}u+|u|^{q-1}uá.
\end{equation}
We also denote by $\nabla'$ the covariant gradient on $S^{N-1}$ in the metric of $S^{N-1}$ obtained by the imbedding into $\BBR^N$.
\subsection{The spherical $\CL_{\xk }$-harmonic problem}

It is straightforward to check that the Poisson kernel  $K_{\CL_\gk}$ of $\CL_\gk$ in $\BBR^N_+$  has the following expression
 \begin{equation}\label{q1-3}
K_{\CL_\gk}(x,\xi)=c_{N,\gk}\myfrac{x_N^{\frac{\ga_+}{2}}}{|x-\xi|^{N+\ga_+-2}}.
\end{equation}
In spherical coordinates
$$K_{\CL_\gk}(x,0)=c_{N,\gk}r^{2-N-\frac{\ga_+}{2}}\psi(\gs)\qquad r>0\,,\;\gs\in S^{N-1}_+
$$
where $\psi_\gk(\gs)=\frac{x_N}{|x|}\lfloor_{S^{N-1}_+}^{\frac{\ga_+}{2}}=({\bf e}_{_N}.\gs)^{\frac{\ga_+}{2}}$ solves
\begin{equation}\label{Eq7}\BA {ll}
-\Gd_{S^{N-1}}\psi_\gk-\gm_{\xk }\psi_\gk-\myfrac{\xk }{({\bf e}_{_N}.\gs)^2}\psi_\gk=0\quad&\text{in }S^{N-1}_+
\\[2mm]\phantom{-\Gd_{S^{N-1}}\psi_\gk-\gl_{\xk }\psi_\gk-\myfrac{\xk }{({\bf e}_{_N}.\gs)^2}}
\psi_\gk=0\quad&\text{in }\prt S^{N-1}_+,
\EA\end{equation}
and
 \begin{equation}\label{Eq1-4}
\gm_{\xk }=\frac{\ga_+}{2}(N+\frac{\ga_+}{2}-2)
\end{equation}
Notice that equation $(\ref{Eq7})$ admits a unique positive solution with supremum $1$. We could have defined the first eigenvalue  $\gm_{\xk }$ of the operator
$$\gf\mapsto \CL'_{\xk }w:= -\Gd_{S^{N-1}}w-\frac{\xk }{({\bf e}_{_N}.\gs)^2}w$$
by
\begin{equation}\label{Eq4}
\gm_{\xk }=\displaystyle\inf\left\{
\myfrac{\int_{S^{N-1}_+}\left(|\nabla w|^2-\xk ({\bf e}_{_N}.\gs)^{-2}w^2\right)dS}{\int_{S^{N-1}_+}w^2dS}
:w\in H^1_0(S^{N-1}_+),w\neq 0\right\}.
\end{equation}
By \cite{XX} the infimum exists since $\gr(\gs)=x_{N}\lfloor_{S^{N-1}_+}={\bf e}_{_N}.\gs$ is the first eigenfunction of $-\Gd_{S^{N-1}}$ in $H^1_0(S^{N-1}_+)$. The minimizer $\psi_{\gk}$ belongs to $H^1_0(S^{N-1}_+)$ only if $1<\gk<\frac{1}{4}$. Furthermore
\begin{equation}\label{Eq4-1}
\psi_{\gk}\in {\bf Y}(S^{N-1}_+):=\{\gf\in H^1_{loc}(S^{N-1}):\gr^{-\frac{\ga_+}{2}}\phi\in H^1(S^{N-1}_+,\gr^{\ga_+})\}.
\end{equation}
We can also define $\gm_k$ by
\begin{equation}\label{Eq4-2}
\gm_k=\inf\left\{\myfrac{\int_{S^{N-1}_+}|\nabla'(\gr^{-\frac{\ga_+}{2}}\gw)|^2\gr^{\ga_+}dS}{\int_{S^{N-1}_+}\gw^2dS}:\gw\in{\bf Y}(S^{N-1}_+)\setminus\{0\}\right\}.
\end{equation}

We can use the symmetry of the operator to obtain the second eigenvalue and eigenfunction of $\CL'_{\xk }$ on $S^{N-1}_+$. We first notice that for $j=1,...,N-1$, the function
\begin{equation}\label{Eq4-3}
x\mapsto \myfrac{x_N^{\frac{\ga_+}{2}}x_j}{|x|^{N+\ga_+}}
\end{equation}
is $\CL_\gk$-harmonic in $\BBR^{N-1}_+$, positive (resp. negative) on $\{x=(x_1,...,x_N:x_j>0,x_N>0\}$ (resp. $\{x=(x_1,...,x_N:x_j<0,x_N>0\}$) and vanishes on $\{x=(x_1,...,x_N:x_j=0,x_N=0\}$.

\begin{prop}\label{2nd} For any $j=1,..,N-1$ the function
$$\gs\mapsto \psi_{\gk,j}(\gs)=({\bf e}_{N}.\gs)^{\frac{\ga_+}{2}}{\bf e}_{j}.\gs
$$
satisfies
\begin{equation}\label{Eq4-4}
\CL'_{\gk}\psi_{\gk,j}=(\gm_\gk+N-1+\ga_+)\gr_{\gk,j}
\end{equation}
in $S^{N-1}_+$. It is positive (resp. negative) on $S^{N-1}_+\cap\{x=(x_1,...,x_N)=x_j>0\}$ (resp. $S^{N-1}_+\cap\{x=(x_1,...,x_N)=x_j<0\}$) and it vanishes on $\prt S^{N-1}_+\cap \{x=(x_1,...,x_N)=x_j=0\}$. The real number
$$\gm_{\gk,2}=\gm_\gk+N-1+\ga_+=(\frac{\ga_+}{2}+1)(N+\frac{\ga_+}{2}-1)$$
is the second eigenvalue of $\CL'_{\gk}$ in ${\bf Y}(S^{N-1}_+)$.
\end{prop}
\begin{proof} There holds
$$\BA {lll}\CL'_{\gk}\psi_{\gk,j}={\bf e}_{j}.\gs\CL_{\gk}\psi_{\gk}+\psi_{\gk}\Gd_{S^{N-1}}{\bf e}_{j}.\gs+2\nabla'\psi_{\gk}.\nabla'{\bf e}_{j}.\gs\\\phantom{\CL_{\gk}\psi_{\gk,j}}
=(\gm_\gk+N-1)\psi_{\gk,j}-\ga_+({\bf e}_{N}.\gs)^{\frac{\ga_+}{2}-1}\nabla'({\bf e}_{j}.\gs).\nabla'({\bf e}_{N}.\gs).
\EA$$
Now
$$\nabla(\frac{x_j}{r})=(\frac{x_j}{r})_r\frac{x}{r}+\frac{1}{r}\nabla'(\frac{x_j}{r})=\frac{1}{r}\nabla'(\frac{x_j}{r})=
\frac{1}{r}{\bf e}_j-\frac{x_j}{r^3}x,
$$
thus
$$\nabla(\frac{x_j}{r}).\nabla(\frac{x_N}{r})=-\frac{x_jx_N}{r^4}=
\frac{1}{r^2}\nabla'(\frac{x_j}{r}).\nabla'(\frac{x_N}{r})=\frac{1}{r^2}\nabla'({\bf e}_{j}.\gs).\nabla'({\bf e}_{N}.\gs)
$$
which implies
$$\nabla'({\bf e}_{j}.\gs).\nabla'({\bf e}_{N}.\gs)=-\frac{x_jx_N}{r^2}=-({\bf e}_{j}.\gs)({\bf e}_{N}.\gs)
$$
and finally
\begin{equation}\label{Eq4-5}
\CL_{\gk}\psi_{\gk,j}=(\gm_\gk+N-1+\ga_+)\psi_{\gk,j}.
\end{equation}
Since $S^{N-1}_+=\{(\gs'\sin\gth,\cos\gth):\gs'\in S^{N-2},\gth\in[0,\frac{\gp}{2}]\}$,  ${\bf e}_{N}.\gs=\cos\gth$,
${\bf e}_{j}.\gs={\bf e}_{j}.\gs'\sin\gth$ and $dS=(\sin\gth)^{N-2}dS' d\gth$ where
$dS$ and $dS'$ are the volume element of $S^{N-1}$ and $S^{N-2}$ respectively, we derive from the fact that $\gs'\mapsto {\bf e}_{j}.\gs'$ is an odd function on $S^{N-2}$,
$$\BA {lll}\myint{S^{N-1}_+}{}\psi_{\gk,j}\psi_{\gk}dS=\myint{S^{N-1}_+}{}({\bf e}_{N}.\gs)^{\ga_+}{\bf e}_{j}.\gs dS
\\[4mm]\phantom{\myint{S^{N-1}_+}{}\gr_{\gk,j}\gr_{\gk}dS}
=\myint{0}{\frac{\gp}{2}}\left(\myint{S^{N-2}}{}{\bf e}_{j}.\gs' dS'\right)(\cos\gth)^{\ga_+}(\sin\gth)^{N-1}d\gth
\\[4mm]\phantom{\myint{S^{N-1}_+}{}\psi_{\gk,j}\gr_{\gk}dS}
=0.
\EA$$
Hence $\psi_{\gk,j}$ is an eigenvalue of $\CL'_\gk$ in ${\bf Y}(S^{N-1}_+)$ with two nodal domains and the space the $\psi_{\gk,j}$ span is (N-1)-dimensional and any linear combination of the $\psi_{\gk,j}$ has exactely two nodal domains since
$$\sum_{j=1}^{N-1}a_j\psi_{\gk,j}=({\bf e}_{N}.\gs)^{\frac{\ga_+}{2}}(\sum_{j=1}^{N-1}a_j{\bf e}_{j}).\gs.
$$
This implies that $\gm_{\gk,2}$ is the second eigenvalue.
\end{proof}
\subsection{The nonlinear eigenvalue problem}
 If we look for separable solutions under the  form
$$u(x)=u(r,\gs)=r^\ga\gw(\gs)
$$
then necessarily $\ga=-\frac{2}{q-1}$ and $\gw$ is a solution of
\begin{equation}\label{Eq2}\BA {ll}
-\Gd_{S^{N-1}}\gw-\ell_{q,N}\gw-\myfrac{\xk }{({\bf e}_{_N}.\gs)^2}\gw+|\gw|^{q-1}\gw=0\quad&\text{in }S^{N-1}_+
\\[2mm]\phantom{-\Gd_{S^{N-1}}\gw-\ell_{q,N}\gw-\myfrac{\xk }{({\bf e}_{_N}.\gs)^2}\gw+|\gw|^{q-1}}
\gw=0\quad&\text{in }\prt S^{N-1}_+,
\EA\end{equation}
\begin{equation}\label{Eq3}
\ell_{q,N}=\frac{2}{q-1}\left(\frac{2}{q-1}+2-N\right)
\end{equation}
 and $(\ref{Eq4})$ is transformed accordingly.  We denote by
 \begin{equation}\label{Eq3-1}
\CE_\gk=\left\{\gw\in {\bf Y}(S^{N-1}_+)\cap L^{q+1}(S^{N-1}_+)\,\text { s. t. (\ref{Eq2}) holds}\right\}
\end{equation}
 and by $\CE_\gk^+$ the set of the nonnegative ones. We also recall that $q_c:=\myfrac{2N+\ga_+}{2N-4+\ga_+}$ and we define a second critical value $q_e:=\myfrac{2N+2+\ga_+}{2N-2+\ga_+}$.\medskip

 The following result holds
 \begin{theorem}\label{struct} Assume $0<\gk\leq\frac{1}{4}$ and $q>1$, then\smallskip

 \noindent   (i) If $q\geq q_c$, $\CE_\gk=\{0\}$.\smallskip

 \noindent (ii) If $1<q<q_c$, $\CE_\gk^+$ is contains exactly two elements: $0$ and $\gw_{\xk }$. Furthermore $\gw_{\xk }$ depends only on the azimuthal angle $\gth$.\smallskip

 \noindent (iii) If $q_e\leq q<q_c$, $\CE_\gk$ contains three elements: $0$, $\gw_{\xk }$ and $-\gw_{\xk }$.
 \end{theorem}
 \begin{proof} We recall that  $q\geq q_c\Longleftrightarrow \ell_{q,N}\leq \gm_\gk$. Then non-existence follows by multiplying by $\gw$ and integrating on $S^{N-1}_+$. For existence, we consider the functional
 \begin{equation}\label{EX-1}
J_\gk(w)=\myint{S^{N-1}_+}{}\left(|\nabla'(w)|^2+(\gm_\gk-\ell_{q,N})w^2+\frac{2}{q+1}\psi^{q-1}_\gk|w|^{q+1}\right) \psi^2_\gk dS,
\end{equation}
defined in $H^1(S^{N-1}_+,\psi^2_\gk dS)\cap L^{q+1}(S^{N-1}_+,\psi^{q+1}_\gk dS)$. Since $\gm_\gk-\ell_{q,N}<0$, there exists a nontrivial minimum $w\gk>0$, which satisfies
 \begin{equation}\label{EX-2}-{\rm div}(\psi^2_\gk\nabla'w_\gk)+(\gm_\gk-\ell_{q,N})\psi^2_\gk w_\gk+\psi^{q+1}_\gk w_\gk^q=0
\end{equation}
If we set $\gw_\gk=\psi_\gk w_\gk$, then $\gw_\gk$ satisfies
 \begin{equation}\label{Eq3-2}
\CL'_\gk\gw_\gk-\ell_{q,N}\gw_\gk+\gw^q_\gk=0\qquad\text{in }S^{N-1}_+.
\end{equation}
By monotonicity we derive that $\gw_\gk\in L^p(S^{N-1}_+)$ for any $1<p<\infty $ and finally, that $\gw_\gk$ satisfies the regularity estimates
of Lemma \ref{reglemma} and Lemma  \ref{reglemma1}. Moreover $\gw_\gk>0$ by the maximum principle.

In the case $q\geq q_c$ or equivalently $\gm_\gk-\ell_{q,N}\geq 0$, nonexistence of nontrivial solution is clear from (\ref{EX-2}).\smallskip

\noindent{\it Uniqueness}. By Proposition \ref{exi} $\gw_\gk(x)\leq c_{86}(\gr(x))^{\frac{\ga_+}{2}}$ and
by standard scaling techniques $|\nabla\gw_\gk(x)|\leq c_{87}(\gr(x))^{\frac{\ga_+}{2}-1}$. Assume now two different positive solutions
of (\ref{Eq2}) $\gw_\gk$ and $\gw'_\gk$ exist. Since $\max\{\gw_\gk,\gw'_\gk\}$ and $\gw_\gk+\gw'_\gk$ are respectively a subsolution and a supersolution and they are ordered, we can assume that  $\gw'_\gk< \gw_\gk<c\gw'_\gk$ for some $c>1$. Let $\ge>0$ and $\ge'=c^{-1}\ge$, then
$\ge\gw'_\gk\geq\ge'\gw_\gk$. Set
$$\gv_{\ge}= \frac{((\gw'_\gk+\ge')^2-(\gw_\gk+\ge)^2)_+}{\gw_\gk+\ge}\,,\;
\gv_{\ge'}= \frac{((\gw'_\gk+\ge')^2-(\gw_\gk+\ge)^2)_+}{\gw'_\gk+\ge'},
$$
and $S_{\ge,\ge'}=\{\gs\in S^{N-1}_+:\gw'_{\gk}+\ge'>\gw_{\gk}+\ge\}$. The assume that $S_{\ge,\ge'}\neq\emptyset$ for any $\ge>0$. Then
$$\myint{S_{\ge,\ge'}}{}\left(\nabla \gw'_\gk.\nabla \gv_{\ge'}-\nabla \gw_\gk.\nabla \gv_{\ge}
-(\ell_{q,N}+\frac{\gk}{\gr^2})(\gw'_\gk.\gv_{\ge'}-\gw_\gk.\gv_{\ge})+\gw'^q_\gk\gv_{\ge'}-\gw^q_\gk\gv_{\ge}
\right) dS=0
$$
The first integrand on the l.h. side is equal to
$$\myint{S_{\ge,\ge'}}{}\left(\left |\nabla \gw'_\gk-\frac{\gw'_\gk+\ge'}{\gw_\gk+\ge}\nabla \gw_\gk
\right|^2+\left|\nabla \gw_\gk-\frac{\gw_\gk+\ge}{\gw'_\gk+\ge'}\nabla \gw'_\gk
\right|^2\right) dS\geq 0
$$
Since $\ge\gw'_\gk<\ge'\gw_\gk$ and $(\gw'_{\gk}+\ge')^2>(\gw_{\gk}+\ge)^2$,the second integrand on the l.h. side is equal to
$$-\myint{S_{\ge,\ge'}}{}(\ell_{q,N}+\frac{\gk}{\gr^2})\left(\frac{\gw'_\gk}{\gw'_\gk+\ge'}-\frac{\gw_\gk}{\gw_\gk+\ge}\right) ((\gw'_\gk+\ge')^2-(\gw_\gk+\ge)^2)dS\geq 0.
$$
At end, the last integrand is
$$
\myint{S_{\ge,\ge'}}{}\left(\frac{\gw'^q_\gk}{\gw'_\gk+\ge'}-\frac{\gw^q_\gk}{\gw_\gk+\ge}\right)((\gw'_\gk+\ge')^2-(\gw_\gk+\ge)^2)dS
$$
If we let $\ge\to 0$, we derive
$$
\myint{S^{N-1}_+}{}\left(\gw'^{q-1}_\gk-\gw^{q-1}_\gk\right)(\gw'^2_\gk-\gw_\gk^2)_+dS\leq 0
$$
This yields a contradiction. Therefore uniqueness holds.\medskip

\noindent{\it Case $q_e\leq q<q_c$}. Assume $\gw_\gk$ is a solution. Using the representation of $S^{N-1}_+$ already introduced in the proof of Proposition \ref{2nd}, with $\gs=(\gs',\gth)$
and
$$\Gd_{S^{N-1}}\gw_\gk=\frac{1}{(\sin\gth)^{N-2}}\frac{\prt}{\prt\gth}\left((\sin\gth)^{N-2}\frac{\prt\gw_\gk}{\prt\gth}\right)+\frac{1}{\sin^{2}\gth}\Gd_{S^{N-2}}\gw_\gk$$
where $\Gd_{S^{N-2}}$ is the Laplace-Beltrami operator on $S^{N-2}$,  we set
$$\bar\gw_\gk(\gth)=\myfrac{1}{|S^{N-2}|}\myint{S^{N-2}}{}\gw_\gk(\gs',\gth)dS'(\gs').
$$
Then $\bar\gw_\gk$ is independent of $\gs'\in S^{N-2}$ and furthermore
$$\myint{S^{N-1}_+}{}(\gw_\gk-\bar\gw_\gk)\psi_\gk dS=
\myint{0}{\frac{\gp}{2}}\left(\myint{S^{N-2}}{}(\gw_\gk-\bar\gw_\gk)dS'\right)(\sin\gth)^{N-2}(\cos\gth)^{\frac{\ga_+}{2}}d\gth=0,
$$
thus $\bar\gw_\gk$ is the projection of $\gw_\gk$ onto the first eigenspace of $\CL_\gk$ and
$$
\myint{S^{N-1}_+}{}(\gw_\gk-\bar\gw_\gk)\CL_\gk(\gw_\gk-\bar\gw_\gk dS
\geq \gm_{\gk,2}\myint{S^{N-1}_+}{}(\gw_\gk-\bar\gw_\gk)^2dS.
$$
At end, noting that
$$\myint{S^{N-2}_+}{}(\overline {g_q\circ\gw_\gk}-g_q\circ\bar\gw_\gk)(\gw_\gk-\bar\gw_\gk)dS'=0
$$
with $g_q\circ u=|u|^{q-1}u$,
$$\BA {lll}\myint{S^{N-1}_+}{}(g_q\circ\gw_\gk-\overline {g_q\circ\gw_\gk})(\gw_\gk-\bar\gw_\gk)dS
=\myint{0}{\frac{\gp}{2}}\myint{S^{N-2}_+}{}(g_q\circ\gw_\gk-\overline {g_q\circ\gw_\gk})(\gw_\gk-\bar\gw_\gk)dS'
(\sin\gth)^{N-2}d\gth\\[4mm]
\phantom{\myint{S^{N-1}_+}{}(g_q\circ\gw_\gk-\overline {g_q\circ\gw_\gk})(\gw_\gk-\bar\gw_\gk)dS}
=\myint{0}{\frac{\gp}{2}}\myint{S^{N-2}_+}{}(g_q\circ\gw_\gk)-g_q\circ\bar\gw_\gk)(\gw_\gk-\bar\gw_\gk)dS'
(\sin\gth)^{N-2}d\gth
\\[4mm]
\phantom{\myint{S^{N-1}_+}{}(g_q\circ\gw_\gk-\overline {g_q\circ\gw_\gk})(\gw_\gk-\bar\gw_\gk)dS}
\geq 2^{1-q}\myint{S^{N-1}_+}{}|\gw_\gk-\bar\gw_\gk|^{q+1}dS,
\EA$$
we derive that $w=\gw_\gk-\bar\gw_\gk$, satisfies
$$\myint{S^{N-1}_+}{}\left((\gm_{\gk,2}-\ell_{N,q})(\gw_\gk-\bar\gw_\gk)^2+2^{1-q}|\gw_\gk-\bar\gw_\gk|^{q+1}\right)dS,
\leq 0
$$
which implies $\gw_\gk=\bar\gw_\gk$ and it satisfies
 \begin{equation}\label{Eq3-3}\frac{1}{(\sin\gth)^{N-2}}\frac{d}{d\gth}\left((\sin\gth)^{N-2}\frac{d\gw_\gk}{d\gth}\right)
+\left(\ell_{q,N}+\frac{\gk}{\cos^{2}\gth}\right)\gw_\gk-g_q\circ \gw_\gk=0.
\end{equation}
Since $\gm_{\gk,1}<\ell_{q,N}\leq\gm_{\gk,2}$,  by \cite[Th. 4, Corol. 1]{Beres} this equation admits three solutions, $\gw_\gk$, $-\gw_\gk$ and $0$.
 \end{proof}

\noindent{\bf Remark}. For $\ge>0$ small enough the function $\ge\psi_\gk$ is a subsolution for problem (\ref{Eq2}). This implies
\be\label{SS}
\gw_\gk(\gs)\geq \ge\psi_\gk(\gs)\qquad\forall\gs\in S^{N-1}_+.
\ee


\subsection{Isolated boundary singularities}

Throughout this section we assume that $\Gw\subset\BBR^N_+$, $0\in\prt\Gw$ the tangent plane to $\prt\Gw$ at $0$ is $\prt \BBR^N_+$ and that $1<q<q_c$.
\begin{lemma}\label {IS-1} There holds
        \begin{equation}\label{Sg4}\BA {ll}
\lim_{|x|\to 0}\myfrac{\BBG_{\CL_\gk}[(K_{\CL_\gk}(.,0))^q](x)}{K_{\CL_\gk}(x,0)}=0
  \EA \end{equation}
\end{lemma}
\begin{proof}
We recall the following estimates (\ref{alpha-green}), (\ref{poissonest})
$$\BA {lll}
(i)\qquad G_{\CL_\gk}(x,y)\leq c_3\min\left\{\myfrac{1}{|x-y|^{N-2}},\myfrac{(d(x))^{\frac{\ga_+}{2}}(d(y))^{\frac{\ga_+}{2}}}{|x-y|^{N+\ga_+-2}}\right\}
\\ [4mm]
(ii)\qquad c^{-1}_3\myfrac{(d(x))^{\frac{\ga_+}{2}}}{|x|^{N+\ga_+-2}}\leq K_{\CL_\gk}(x,0)\leq c_3\myfrac{(d(x))^{\frac{\ga_+}{2}}}{|x|^{N+\ga_+-2}}.
\EA$$
Then
$$\BA {lll}\myfrac{\BBG_{\CL_\gk}[K_{\CL_\gk}^q(.,0)](x)}{K_{\CL_\gk}(x,0)}
\leq c_3^{q+2}|x|^{N+\ga_+-2}\myint{\Gw}{}\myfrac{(d(y))^{\frac{(q+1)\ga_+}{2}}dy}{|x-y|^{N+\ga_+-2}|y|^{q(N+\ga_+-2)}}\\[4mm]
\phantom{\myfrac{\BBG_{\CL_\gk}[K_{\CL_\gk}(.,0)](x)}{K_{\CL_\gk}(x,0)}}
\leq c_3^{q+2}|x|^{N+\frac{\ga_+}{2}-q(N+\frac{\ga_+}{2}-2)}\myint{\BBR^N}{}\myfrac{d\eta}{|e_x-\eta|^{N+\ga_+-2}|\eta|^{q(N+\ga_+-2)}}
\EA$$
where $e_x=|x|^{-1}x$. This last integral is finite and independent of $x$. Since $q<q_c$, (\ref{Sg4}) follows.
\end{proof}

\begin{coro}\label {IS-2} Let $u_{k\gd_0}$ be the unique solution of
        \begin{equation}\label{Sg5}\BA {ll}
\CL_{\gk}u+|u|^{q-1}u=0\qquad&\text{in }\Gw\\
\phantom{\CL_{\gk}u+|u|^{q-1}}
u=k\gd_0\qquad&\text{in }\prt\Gw.
  \EA \end{equation}
  Then
          \begin{equation}\label{Sg6}\BA {ll}
\displaystyle\lim_{x\to 0}\myfrac{u_{k\gd_0}}{K_{\CL_\gk}(x)}=k.
  \EA \end{equation}
\end{coro}

\begin{proof} This is a consequence of (\ref{Sg4}) and the inequality
\be
k\BBK_{\CL_\gk}[\gd_0](x)-k^q\BBG[(\BBK_{\CL_\gk}[\gd_0])^q](x)\leq u_{k\gd_0}(x)\leq k\BBK_{\CL_\gk}[\gd_0](x).\label{cor4}
\ee
\end{proof}

\begin{prop}\label {IS-3} There exists $u_{\infty,0}=\lim_{k\to\infty}u_{k\gd_0}$ and there holds
          \begin{equation}\label{Sg7}\BA {ll}
\displaystyle\lim_{\tiny\BA {cc}x\to 0,x\in\Gw\\
x|x|^{-1}\to\gs\EA}| x|^{\frac{2}{q-1}}u_{\infty,0}(x)=\gw_{\xk }(\gs),
  \EA \end{equation}
  uniformly on compact subsets of $S^{N-1}_+$.
\end{prop}

\begin{proof} The correspondence $k\mapsto u_{k\gd_0}$ is increasing and by the Keller-Osserman estimate, it converges, when $k\to\infty$ to some smooth function $u_{\infty,0}$ defined in $\Gw$ where it satisfies (\ref{IE1}). By Proposition \ref{barr}, for any $0<R<R_0$, $u_{k\gd_0}$, and therefore $u_{\infty,0}$, vanishes on any compact subset of $\prt\Gw\setminus\{0\}$ and furthermore
$$u_{\infty,0}(x)\leq \left\{\BA{lll}c_{K,\gamma,\gk}(\dist(x,K))^{\gamma}\qquad\forall \gamma\in (\frac{\ga_-}{2},\frac{\ga_+}{2})\,&\text{ if }0<\gk<\frac{1}{4}\\[2mm]
c_{K}\sqrt{\dist(x,K)}\sqrt{\ln \left(\frac{{\rm {diam}}(\Gw)}{\rm {dist}(x,K)}\right)}&\text{ if }\gk=\frac{1}{4}
\EA\right.
$$
for all compact set $K\subset \prt\Gw\setminus\{0\}$. Combining this estimate with Propositions \ref{prop19} we obtain
          \begin{equation}\label{Sg8}\BA {ll}
  u_{\infty,0}(x)\leq c_{90}(d(x))^{\frac{\ga_+}{2}}|x|^{-\frac{2}{q-1}-\frac{\ga_+}{2}}\qquad\forall x\in\Gw,
  \EA \end{equation}
  and
            \begin{equation}\label{Sg9}\BA {ll}
|\nabla u_{\infty,0}(x)|\leq c_{90}(d(x))^{\frac{\ga_+}{2}-1}|x|^{-\frac{2}{q-1}-\frac{\ga_+}{2}}\qquad\forall x\in\Gw.
  \EA \end{equation}
Let $\ell_0>0$ be small enough such that $\ell\mathbf{e}\in\xO$ for any $0<\ell<\ell_0,$ where $\mathbf{e}=(0,...,0,1).$ Then by (\ref{alpha-green}), (\ref{poissonest}) and \eqref{cor4} we can easily prove that there exist positive constants $c_{01}$ and $c_{02}$ such that
$$\ell^{\frac{2}{q-1}}u_{\infty,0}(\ell\mathbf{e})\geq c_{01}k\ell^{\frac{2}{q-1}-N-\frac{\xa_+}{2}+2}-c_{02}k^q\ell^{2-q(N+\frac{\xa_+}{2}-2)+\frac{2}{q-1}},\quad\forall k>0.$$
Now we set $k=\frac{1}{M\ell^{\frac{2}{q-1}-N-\frac{\xa_+}{2}+2}},$
then we have that

$$\ell^{\frac{2}{q-1}}u_{\infty,0}(\ell\mathbf{e})\geq\frac{c_{01}}{M}-\frac{c_{02}}{M^q}.$$
Thus if we choose $M$ big enough, we can easily show that there exists $c_{03}>0$ which depends on $\xk,\xO,q,N$ such that
\be
\ell^{\frac{2}{q-1}}u_{\infty,0}(\ell\mathbf{e})\geq c_{03}>0,\quad\forall 0<\ell<\ell_0.\label{cor5}
\ee 
For $\ell>0$, we put $T_\ell[v](x)=\ell^{\frac{2}{q-1}}v(\ell x)$, $\Gw_\ell=\ell^{-1}\Gw$, $d_\ell(y)=\dist (y,\prt\Gw_\ell)$. If $v$ satisfies (\ref{Eq1}) in $\Gw$ and vanishes on $\prt\Gw\setminus\{0\}$, $T_\ell[v]$ vanishes on $\prt\Gw_\ell\setminus\{0\}$ and satisfies
          \begin{equation}\label{Sg9+}\BA {ll}
-\Gd T_\ell[v]-\myfrac{\gk}{d^2_\ell}T_\ell[v]+|T_\ell[v]|^{q-1}T_\ell[v]=0\qquad\in \Gw_\ell.
  \EA \end{equation}
In order to avoid ambiguity, we set $u_{k\gd_0}=u_{k\gd_0}^\Gw$, $v_{k\gd_0}=v_{k\gd_0}^\Gw$, $u_{\infty,0}=u_{\infty,0}^\Gw$ and $v_{\infty,0}=v_{\infty,0}^\Gw.$ Since inequalities (\ref{Sg8}) and (\ref{Sg9}) are invariant under the scaling transformation, the standard elliptic equations regularity theory yields the following estimates

            \begin{equation}\label{Sg10}\BA {ll}
              u_{\infty,0}^{\Gw_\ell}(y)\leq c_{92}(d_\ell(y))^{\frac{\ga_+}{2}}|y|^{-\frac{2}{q-1}-\frac{\ga_+}{2}}\qquad\forall y\in\Gw_\ell,
  \EA \end{equation}
  and
              \begin{equation}\label{Sg11}\BA {ll}
|\nabla u_{\infty,0}^{\Gw_\ell}(y)|\leq c_{92}(d_\ell(y))^{\frac{\ga_+}{2}-1}|y|^{-\frac{2}{q-1}-\frac{\ga_+}{2}}\qquad\forall y\in\Gw_\ell,
  \EA \end{equation}
  valid for any $0<\ell\leq 1$.
  If we let $k\to\infty$, we obtain
  $T_\ell[u_{\infty,0}^\Gw]=u_{\infty,0}^{\Gw_\ell}$ and because of the group property of the transformation $T_\ell$,
  $T_{\ell'}[u_{\infty,0}^{\Gw_{\ell}}]=u_{\infty,0}^{\Gw_{\ell'\ell}}$ for any $\ell,\ell'>0$. Estimates (\ref{Sg10}) and (\ref{Sg11}) imply that $\{u_{\infty,0}^{\Gw_\ell}\}$ is relatively compact for the  topology of convergence on compact subsets of $\BBR^N_+$. Therefore there exists a sequence $\{\ell_n\}$ tending to $0$ and a function $U$ such that $\{u_{\infty,0}^{\Gw_{\ell_n}}\}$ converges to $U$ uniformly on any compact subset of $\BBR^N_+$. By \eqref{cor5} this function is identically equal to zero. Therefore $U$ is a weak solution of
                \begin{equation}\label{Sg12}\BA {ll}
-\Gd U-\myfrac{\gk}{y_N^2}U+U^q=0\qquad\text{in }\BBR^N_+
  \EA \end{equation}
  Furthermore
              \begin{equation}\label{Sg13}\BA {ll}
  u_{\infty,0}^{\BBR^N_+}(y)\leq c_{92}y_N^{\frac{\ga_+}{2}}|y|^{-\frac{2}{q-1}-\frac{\ga_+}{2}}\qquad\forall y\in\BBR^N_+.
  \EA \end{equation}
  Since $T_{\ell'}[u_{\infty,0}^{\Gw_{\ell_n}}]=u_{\infty,0}^{\Gw_{\ell'\ell_n}}$ we derive $T_{\ell'}[U]=U$ for any $\ell'>0$, thus $U$ is self similar. Set $\gw(\frac{y}{|y|})=U(\frac{y}{|y|})$. If we set $\gs=\frac{y}{|y|}$ then there holds
                \begin{equation}\label{Sg14}\BA {ll}
 \gw(\gs)\leq c_{92}\psi_\gk(\gs)\qquad\forall \gs\in S^{N-1}_+.
  \EA \end{equation}
Therefore $\gw$ satisfies (\ref{Eq2}) and it coincides with the unique positive element $\gw_\gk$ of $\CE_\gk$, since by \eqref{cor5} $U(\mathbf{e})\geq c_{03}>0.$ Thus $u_{\infty,0}^{\Gw_{\ell}}$ converges to $U$ on compact subsets of $\BBR^N_+$. In particular (\ref{Sg7}) holds on compact subsets of $S^{N-1}_+$.
\end{proof}

\setcounter{equation}{0}
\section{The boundary trace of positive solutions}
As before we assume that $0<\gk\leq \frac{1}{4}$, $q>1$ and $\xO$ is a bounded smooth domain, convex if $\gk=\frac{1}{4}$. Although the construction of the boundary trace can be made in a more general framework, we restrict ourselves to the class $\CU_+(\Gw)$ of  positive smooth functions $u$ satisfying
        \begin{equation}\label{T1}\BA {ll}
\CL_\gk u+|u|^{q-1}u=0
  \EA \end{equation}
  in $\Gw$.
  \begin{lemma}\label{mode} Let $f\in L^1_{\phi_\gk}(\Gw)$. If $u$ is a nonnegative solution of
          \begin{equation}\label{T2*}\BA {ll}
\CL_\gk u=f\qquad\text{in }\Gw
  \EA \end{equation}
  there exists $\gm\in\mathfrak M_+(\prt\Gw)$ such that $u$ admits $\gm$ for boundary trace and
            \begin{equation}\label{T3}\BA {ll}
u=\BBG_{\CL_\gk}[f]+\BBK_{\CL_\gk}[\gm].
  \EA \end{equation}
  \end{lemma}
  \begin{proof} Let $v=\BBG_{\CL_\gk}[f]$, then $u-v$ is $\CL_\gk$-harmonic and positive thus the result follows.
  \end{proof}

\noindent{\bf Definition} Let $G\subset\Gw$ be a domain. A function $u\in L^q_{loc}(G)$ is a supersolution (resp. subsolution) of (\ref{T1}) if
        \begin{equation}\label{T4}\BA {ll}
\CL_\gk u+|u|^{q-1}u\geq 0\quad (\text{resp. }\;\;\CL_\gk u+|u|^{q-1}u\leq 0\;)
  \EA \end{equation}
  in the sense of distributions in $G$.\medskip

  The following comparison principle holds \cite[Lemma 3.2]{1}
    \begin{prop}\label{comp2} Let $G\subset\Gw$ be a smooth domain and $\bar u,\underline u$ a pair of nonnegative supersolution and subsolution respectively in $G$.\smallskip

    \noindent (i) If there holds
            \begin{equation}\label{T5}\BA {ll}
\displaystyle\limsup_{\dist(x,\prt G)\to 0}(\bar u(x)-\underline u(x))<0,
  \EA \end{equation}
  then $\underline u <\bar u$ in $G$.\smallskip

    \noindent (ii) Assume $\overline G\subset \Gw$ and $\bar u$ and $\underline u$ belong to $H^1(G)\cap C(\overline G)$. If
$\underline u \leq \bar u$ in $\prt G$, then $\underline u \leq\bar u$ in $G$.
      \end{prop}
      \subsection{Construction of the boundary trace}
      We use the notations of \cite{MV-CONT}
      \begin{prop} Let $\gu$ be a non-negative function in $C(\xO).$\\
(i) If $\gu$ is a subsolution of (\ref{T1}), there exists a minimal solution $u_*$ dominating $\gu$,
i.e. $\gu\leq u_*\leq U$ for any solution $U\geq \gu.$\\
(ii) If $\gu$ is a supersolution of (\ref{T1}), there exists a maximal solution $u^*$ dominated by $\gu$,
i.e. $U\leq u^*\leq \gu$ for any solution $U\leq \gu.$
\end{prop}
\begin{proof}

(i) Let $\{\xO_n\}$ be a smooth exhaustion $\Gw$ and for each $n\in\BBN$, $u_n$ the positive solution of
            \begin{equation}\label{T6}\BA {lll}
\CL_{\xk } u+ |u|^{q-1}u=0\qquad&\text {in }\xO_n\\
\phantom{\CL_{\xk } u+ |u|^{q-1}}
u=\gu\qquad&\text {in }\prt\xO_n.
  \EA \end{equation}
By the comparison principle $u_n\geq \gu,$ which implies  $u_{n+1}(x)\geq u_n(x)\;\forall x\in \Gw_n.$ Since $\{u_n\}$ is uniformly bounded on compact subsets of $\Gw$ and thus in $C^2$ by standard regularity arguments that $u_n\uparrow u_*$ which is a positive solution of (\ref{T1}). Furthermore, if $U$ is any solution of (\ref{T1}) dominating $\gu$, it dominates  $u_n$ in $\Gw_n$ and thus $ u_*\leq U$. \\
The proof of (ii) is similar: we construct a decreasing sequence $\{u'_n\}$ of nonnegative solutions of (\ref{T1}) in $\Gw_n$
coinciding with $\gu$ on $\prt\Gw_n$ and dominated by $\gu$. It converges to some $u^*$ which satisfies  $U\leq u^*\leq \gu$ for any solution $U$ dominated by $\gu$.
\end{proof}

\begin{prop}
Let $0\leq u,v\in C(\xO).$\\
(i) If $u$ and $v$ are subsolutions (resp. supersolutions) then $\max(u,v)$ is a subsolution (resp. $\min(u,v)$ is a supersolution).\\
(ii) If $u$ and $v$ are supersolutions then $u+v$ is a supersolution.\\
(iii) If $u$ is a subsolution and $v$ is a supersolution then $(u-v)_+$ is a subsolution.\label{maxsub}
\end{prop}
\begin{proof}   The first two statements follow Kato's inequality. The last statement is verified using that
$$\BA {lll}
-\xD(u-v)_+\leq sign_+(u-v)(-\xD(u-v))&\leq-sign_+(u-v)(u^q-v^q)+\xk \frac{(u-v)_+}{d^2(x)}\\
&\leq-(u-v)^q_++\xk \frac{(u-v)_+}{d^2(x)}.
\EA
$$
\end{proof}

\begin{notation}
Let $u,v$ be nonnegative continuous functions in $\xO.$\\
(a) If $u$ is a subsolution, $[u]_\dag$ denotes the smallest solution dominating $u.$\\
(b) If $u$ is a supersolution, $[u]^\dag$ denotes the largest solution dominated by $u.$\\
(c) If $u,\;v$ are subsolutions then $u\vee v:=[\max(u,v)]_\dag.$\\
(d) If $u,\;v$ are supersolutions then $u\wedge v:=[\inf(u,v)]^\dag$ and $u\oplus v=[u+v]^\dag.$\\
(e) If $u$ is a subsolution and $v$ is a supersolution then $u\ominus v:=[(u-v)_+]_\dag.$\label{34}
\end{notation}
The next result based upon local uniform estimates is due to Dynkin \cite{Dyn}.
\begin{prop}
(i) Let $\{u_k\}\subset C(\Gw)$ be a sequence of positive subsolutions (resp. supersolutions) of (\ref{T1}). Then $U:=\sup u_k$ (resp. $U:=\inf u_k$) is a subsolution (resp. supersolution).\\
(ii) Let $\mathcal{T}\subset C(\Gw)$ be a family of positive solutions of (\ref{T1}). Suppose that, for every pair $u_1,u_2\in\mathcal{T}$ there exists $v\in\mathcal{T}$ such that
$$\max(u_1,u_2)\leq v\qquad \mathrm{resp.}\;\min(u_1,u_2)\geq v.$$
Then there exists a monotone sequence $\{u_n\}\subset \mathcal{T}$ such that
$$u_n\uparrow\sup\mathcal{T}\qquad\mathrm{resp.}\;u_n\downarrow\inf T.$$
Furthermore $\sup \mathcal{T}$ (resp. $\inf\mathcal{T}$) is a solution.\label{sygklish}
\end{prop}

\begin{defin}
Let $F\subset\partial\xO$ be a closed set. We set
\begin{equation}\label{T7}
U_F:=\sup\left\{u\in\CU_+(\xO):\;\lim_{x\rightarrow \xi}\frac{u(x)}{W(x)}=0,\;\forall \xi\in \partial\xO\setminus F\right\},
\end{equation}
and
\begin{equation}\label{T8}
[u]_F=\sup\left\{v\in\CU_+(\xO):\;v\leq u,\;\lim_{x\rightarrow \xi}\frac{v(x)}{W(x)}=0,\;\forall \xi\in \partial\xO\setminus F\right\}
\end{equation}
\end{defin}
Notice that $F\mapsto U_F$ and $F\mapsto [u]_F$ are increasing with respect to the inclusion order relation in $\prt\Gw$, $[u]_F=u\wedge U_F$. As a consequence of Proposition \ref{prop19}, $U_F$ satisfies
\begin{equation}\label{T8'}\lim_{x\rightarrow \xi}\frac{U_F(x)}{W(x)}=0,\;\forall \xi\in \partial\xO\setminus K.
\end{equation}
\begin{prop}\label{sygklish1}
Let $E,F\subset\partial\xO$ be closed sets. Then\\
(i) $U_E\wedge U_F=U_{E\cap F}.$\\
(ii) If $F_n\subset\partial\xO$ is a decreasing sequence of closed sets there holds
$$\lim_{n\to\infty} U_{F_n}=U_F\;\;\text{where}\;\;F=\cap F_n.$$
\end{prop}
\begin{proof}
(i) $U_E\wedge U_F$ is the largest solution dominated by $\inf(U_E, U_F)$ and therefore, by definition, it is the largest solution which vanishes outside $E\cap F.$\\
(ii) If $V:=\lim U_{F_n}$ then $U_F\leq V.$ But $\hbox{supp}\,(V)\subset F_n$ for each $n\in\mathbb{N}$ and consequently $V\leq U_F.$
\end{proof}
For $\xb>0,$ we recall that $\xO_\xb$, $\xS_\xb$ and the mapping $x\mapsto(d(x),\xs(x))$ have been defined in the proof of Lemma \ref{heat}. We also set $\Gw'_\xb=\xO\setminus\overline{\xO}_\xb$ and, if $Q\subset\partial\xO$, $\xS_\xb(Q)=\{x\in\xO_\xb:\;\xs(x)\in Q\}.$

\begin{prop}\label{17} Let $u\in\mathcal{U}(\xO).$\smallskip

\noindent (i) If $A,B\subset\partial\xO$ are closed sets. Then
\be
[[u]_A]_B=[[u]_B]_A=[u]_{A\cap B}.\label{14}
\ee\smallskip

\noindent (ii) If $\{F_n\}$ is a decreasing sequence of closed subsets of $\prt\Gw$ and $F=\cap F_n$, then
$$[u]_{F_n}\downarrow [u]_F.$$

\noindent (iii) If $A,B\subset\partial\xO$ are closed sets. Then
\be
[u]_A\leq[u]_{A\cap B}+[u]_{\overline{A\setminus B}}.\label{15}
\ee
\end{prop}

\begin{proof}
(i) It follows directly from definition that,
$$[[u]_A]_B\leq\inf(u,U_A,U_B).$$
The largest solution dominated by $u$ and vanishing on $A^c\cup B^c$ is $[u]_{A\cap B}.$ Thus
$$[[u]_A]_B\leq[u]_{A\cap B}.$$
On the other hand $$[u]_{A\cap B}=[[u]_{A\cap B}]_B\leq[[u]_A]_B,$$
this proves (\ref{14}).\\
(ii) If $F_n\downarrow F,$ it follows by Proposition \ref{sygklish1}-(ii) that $U_{F_n}\rightarrow U_F,$ thus
$$\displaystyle [u]_F\leq\lim_{n\to\infty}[u]_{F_n}=\lim_{n\to\infty} u\wedge U_{F_n}\leq\lim_{n\to\infty}\,\inf(u,U_{F_n})\leq\inf(u,U_F).$$
Since $[u]_F$ is the largest solution dominated by $\inf(u,U_F)$, $[u]_{F_n}$ is the largest solution dominated by $\inf(u,U_{F_n})$ and $U_{F_n}\downarrow U_{F}$ by Proposition \ref{sygklish1}, the function $v=\lim_{n\to\infty}[u]_{F_n}$ is a solution of (\ref{T1}) dominated by $\inf(u,U_F)$, thus $v\leq [u]_F$ and the proof of (ii) is complete.\\
(iii) Without loss of generality we assume that $A\cap B\neq\emptyset.$ Let $O,O'\subset\prt\xO$ be a relatively open set such that $A\cap B\subset O$ and $\overline{A\cap B^c}\subset O'$ Set $v=[u]_A$
and let $v_\xb^1$ be the solution of
$$
\BA {lll}
\CL_{\xk }w +|w|^{q-1}w=0\qquad&\mathrm{in}\; \xO'_\xb\;\\
\phantom{\CL_{\xk }w +|w|^{q-1}}
w=\chi_{\xS_\xb(\overline{O})}v\qquad&\mathrm{on}\;\xS_\xb.
\EA
$$
Also we denote by $v_\xb^2$ and $v_\xb^3$ the solutions of the above problem with respective boundary data  $\chi_{\xS(\overline{O'})}v$ and  $\chi_{\xS(O^c\cap O'^c)}v$.
Then $v_\xb^i\leq v\lfloor_{\Gw'_\gb}\leq v_\xb^1+v_\xb^2+v_\xb^3$, $i=1,2,3$. Let now $\{\xb_j\}$ be a decreasing sequence converging to $0$ and such that
$$v_{\xb_j}^i\rightarrow v^i\leq v\leq v^1+v^2+v^3,\;i=1,2,3\;\text{locally uniformly in $\xO$}.$$
By definition of $v^i$ and Proposition \ref{barr}, we have that $v^1\leq [v]_{\overline{O}}$, $v^2\leq [v]_{\overline{O'}}$ and
$v^3\leq [v]_{O^c\cap O'^c}$.
But by (i) we have
$$[v]_{O^c\cap O'^c}=[[u]_A]_{O^c\cap O'^c}=[u]_{A\cap O^c\cap O'^c}=0.
$$
Thus
$$v\leq [v]_{\overline{O}}+ [v]_{\overline{O'}}$$
We can consider decreasing sequences $\{O_n\}$ and $\{O'_n\}$ such that $\cap\overline{O_n}=A\cap B$ and
$\cap\overline{O'_n}=\overline{A\cap B^c}$. By (ii) we obtain
$$v\leq [[u]_A]_{A\cap B}+[[u]_A]_{\overline{A\cap B^c}}\leq  [u]_{A\cap B}+[u]_{\overline{A\cap B^c}}
$$
which is (iii).
\end{proof}

\noindent{\bf Remark}.  Since any $u\in\CU_+(\Gw)$ is dominated by $u_{\prt\Gw}$, it follows from (iii)
 that for any set $A\subset\prt\Gw$, there holds
 \begin{equation}\label{T7'}
 u=[u]_{\prt\Gw}\leq [u]_{\overline A}+[u]_{\overline{\prt\Gw\setminus \overline A}}\leq [u]_{\overline A}+[u]_{\overline{\prt\Gw\setminus A}}.
 \end{equation}

\begin{prop}
Let $u$ be a positive solution of \eqref{T1}. If $u\in L^q_{\ei}(\xO)$
 it possesses a boundary trace $\xm \in \mathfrak{M}(\partial\xO),$ i.e., $u$ is the solution of
the boundary value problem \eqref{N1} with this measure $\xm.$\label{pro1}
\end{prop}
\begin{proof}
If $v := \BBG_{\CL_{\xk }}[u^q]$ then $v\in L^1_{\ei}(\xO)$ and $u + v$ is a positive $\CL_{\xk }$-harmonic function.
Hence $u + v \in L^1_{\ei}(\xO)$ and there exists a non-negative measure $\xm \in \mathfrak{M}(\partial\xO)$ such
that $u + v = \po[\xm].$ By Proposition \ref{equiv-def} this implies the result.
\end{proof}
\begin{prop} \label{pro2}Let $u$ be a positive solution of \eqref{T1} and $\xm \in \mathfrak M(\partial\xO).$ If for an exhaustion $\{\xO_n\}$ of $\xO,$ we have
$$
\lim_{n\rightarrow\infty}\int_{\partial\xO_n}Z(x)u d\gw^{x_0}_{\Gw_n}=\int_{\partial \xO}Z(x)d\xm,\quad\forall Z \in C(\overline{\xO}),
$$
where  $\gw^{x_0}_{\Gw_n}$ is the $\CL_{\xk }$-harmonic measure of $\xO_n$ relative to a point $x_0\in \xO_1$, then
$u$ and $|u|^p$ belong to  $L^1_{\ei}(\xO)$. Furthermore $u$ possesses the boundary trace $\xm \in \mathfrak{M}(\partial\xO),$ i.e., $u$ is the solution of
the boundary value problem \eqref{N1} with this measure $\xm.$
\end{prop}
\begin{proof}
Let $\gre^n$ be the green function of $\CL_{\xk }$ in $\xO_n,$ then
$$\gre^n(x,y)\leq \gre^{n+1}(x,y),\quad\forall x,y\in \xO_n$$
 and
 $$\gre^n\uparrow\gre.$$
Since
$$\int_{\partial\xO_n}u d\gw^{x_0}_{\Gw_n}=u(x_0)+\int_{\xO_n}\gre^n(x,x_0)|u(x)|^qdx,$$
we derive, as $n\to\infty$,
$$\xm(\partial\xO)=u(x_0)+\int_{\xO_n}\gre(x,x_0)|u(x)|^qdx.$$
By Proposition \ref{green} this implies $|u|^q\in L^1_{\ei}(\xO),$ and the result follows by Proposition \ref{pro1}.
\end{proof}
\begin{prop}
If $F\subset\partial\xO$ is a closed set and $u$ a positive solution of (\ref{T1}) with boundary trace $\xm\in\mathfrak{M}(\partial\xO)$, then $[u]_F$ has boundary trace $\xm\chi_F.$\label{prop}
\end{prop}
\begin{proof}
The function $[u]_F$ belongs to $\CU_+(\Gw)$ and is dominated by $u$ which satisfies (\ref{T1}), thus $[u]_F\in L^q_{\Gf_\gk}(\Gw)$ and $[u]_F$ admits a boundary trace $\xm_F\leq\xm$ by Proposition \ref{pro1}. Let $v$ be the solution of (\ref{N1}) with boundary data $\xm\chi_F$.  Let $O\subset\partial\xO$ relatively open such that $F\subset O.$ By \ref{T7'} we have
$$v\leq [v]_{\overline{O}}+[v]_{\overline{O^c}}.$$
Let  $A$ be an open set such that $F\subset A\subset\overline{A}\subset O,$ and for exhaustion we take $\Gw_n=\Gw'_{\frac{1}{n}}$ which is smooth for $n$ large enough, and $\prt\Gw_n=\xS_{\frac{1}{n}}$. Then
$$
\int_{\partial{\xO_n}}[v]_{\overline{O^c}}d\gw^{x_0}_{\Gw_n}=\int_{\xS_{\frac{1}{n}}(A)}[v]_{\overline{O^c}}d\gw^{x_0}_{\Gw_n}
+\int_{\partial{\xO_n}\setminus\xS_{\frac{1}{n}}(A)}[v]_{\overline{O^c}}d\gw^{x_0}_{\Gw_n}$$
But
$$\int_{\xS_{\frac{1}{n}}(A)}[v]_{\overline{O^c}}d\gw^{x_0}_{\Gw_n}\leq \int_{\xS_{\frac{1}{n}}(A)}vd\gw^{x_0}_{\Gw_n}\rightarrow 0$$
and
$$\int_{\partial{\xO_n}\setminus\xS_\frac{1}{n}(A)}[v]_{\overline{O^c}}d\gw^{x_0}_{\Gw_n}\leq
\int_{\partial{\xO_n}\setminus\xS_\frac{1}{n}(A)}U_{\overline{O^c}}d\gw^{x_0}_{\Gw_n}\rightarrow0,$$
as $n\to\infty$, thus $[v]_{\overline{O^c}}=0$ by Proposition \ref{pro2} and therefore $v\leq [v]_{\overline{O}}\leq [u]_{\overline{O}}$.
 Since $O$ be an arbitrary open set, take a sequence of open set $\{O_n\}$ such that $F\subset O_n\subset \overline{O}_{n}\subset O_{n-1}$ and $\cap O_n=F.$ Using Proposition \ref{17} we derive
$$v\leq [u]_F,$$ and thus $\xm\chi_F\leq \xm_F.$
Conversely, let $Z\in C(\overline{\xO})$, $Z\geq 0$,
\begin{align}\nonumber
\int_{\partial{\xO_n}}Z[u]_{F}d\gw^{x_0}_{\Gw_n}&=\int_{\partial{\xO_n}\cap\xS_\frac{1}{n}(A)}Z[u]_{F}d\gw^{x_0}_{\Gw_n}
+\int_{\partial{\xO_n}\setminus\xS_\frac{1}{n}(A)}Z [u]_{F}d\gw^{x_0}_{\Gw_n}\\ \nonumber
&\leq\int_{\partial{\xO_n}\cap\xS_\frac{1}{n}(A)}Z ud\gw^{x_0}_{\Gw_n}
+\int_{\partial{\xO_n}\setminus\xS_\frac{1}{n}(A)}Z U_{F}d\gw^{x_0}_{\Gw_n}
\\ \nonumber
&\leq I_n+II_n.
\end{align}
Because of (\ref{T8'}), $II_n\to 0$ as $n\to\infty$, thus
$$\int_{\partial{\xO}}Zd\xm_F\leq \int_{\partial{\xO}}Z\chi_Fd\xm\Longrightarrow \xm_F\leq\xm\chi_O,
$$
and the result follow by regularity since $O$ is arbitrary.
\end{proof}

The next result shows that the boundary trace has a local character.
\begin{prop}\label{pro3}
Let $u\in\mathcal{U}_+(\xO)$ and $\xi\in\partial\xO.$ We assume that there exists $\xr>0$ such that
$$\int_{B_\xr(\xi)\cap\xO}(u(x))^q\ei(x)dx<\infty.$$

\noindent (i) Then
$$[u]^q_F\in L^1_{\ei}(\xO)\quad\forall F\subset\prt\xO\cap B_\xr(\xi),\, F\text{ closed}.$$
Thus $[u]_F$ possesses a boundary trace $\xm_F \in \mathfrak{M}(\partial\xO),$ and $\mathrm{supp}\,(\xm_F)\subset F.$\smallskip

\noindent (ii) There exists a nonnegative Radon measure $\xm_\xr$ on $B_\xr(\xi)$ such that for any closed set $F\subset B_\xr(\xi)\cap\prt\Gw$
$$\xm_F=\xm_\xr\chi_F,$$
and for any  exhaustion $\{\xO_n\}$ of $\xO$ and any $Z\in C(\overline{\xO})$ such that $\text{supp}(Z)\cap\partial\xO\subset \partial\xO \cap B_\xr(\xi)$
\be
\lim_{n\rightarrow\infty}\int_{\partial\xO_n}u(x)Z(x)d\gw^{x_0}_{\Gw_n}= \int_{\partial\xO}u(x)Z(x)d\xm_\xr.\label{5.26}
\ee

\end{prop}  
\begin{proof}
(i) Let $F$ be a closed set and $0<\xr'<\xr$ be such that $$F\subset\partial\xO\cap B_{\xr'}(\xi).$$
Since $[u]_F\leq\inf(u,U_F)$ and $U_F\in C(\overline{\xO}\setminus F),$ we have
$$\int_{\xO}[u]_F^q\ei(x)dx\leq\int_{B_\xr(\xi)\cap\xO}|u|^p\ei(x)dx+\int_{\xO\setminus B_\xr(\xi)}|U_F|^p\ei(x)dx<\infty.$$
(ii) Let $0<\xr_1<\xr_2<\xr,$
then $$[u]_{\overline{B}_{\xr_2}(\xi)\cap\partial\xO}\leq u\leq[u]_{\overline{B}_{\xr_2}(\xi)\cap\partial\xO}+U_{\overline{\partial\xO\setminus\overline{B}_{\xr_2}(\xi)}}.$$
The function $[u]_{\overline{B}_{\xr_2}(\xi)\cap\partial\xO}$ which belongs $L^q_{\ei}(\xO)$ admits a boundary trace $\xn\in\mathfrak M(\prt\Gw)$ and
$$\lim_{n\rightarrow\infty}\int_{\partial\xO_n}U_{\overline{\partial\xO\setminus\overline{B}_{\xr_2}(\xi)}}Z(x)d\gw^{x_0}_{\Gw_n}= 0,$$
for any $Z\in C(\overline{\xO})$ such that $\text{supp}(Z)\cap\partial\xO\subset \partial\xO \cap B_{\xr_1}(\xi).$
Combined with Proposition \ref{prop} it follows identity (\ref{5.26}) and finally statement (ii).
\end{proof}
\begin{defin} The set $\mathcal{R}_u$ of boundary points a such that there exists $r > 0$ such
that (\ref{5.26}) holds is relatively open. Using a partition of unity there exists a positive Radon measure $\xm_u$ on $\mathcal{R}_u$ such that
\begin{equation}\label{REG1}\lim_{n\rightarrow\infty}\int_{\partial\xO_n}u(x)Z(x)d\gw^{x_0}_{\Gw_n}= \int_{\partial\xO}u(x)Zd\xm_u
\end{equation}
for any $Z\in C(\overline{\xO})$ such that $\text{supp}(Z) \cap\partial\xO \subset\mathcal{R}_u.$ The set $\mathcal{S}_u :=\partial\xO\setminus \mathcal{R}_u$ is closed. The couple $(\mathcal{S}_u, \xm_u)$ is the boundary trace of $u,$ denoted by $\text{Tr}_{\partial\xO}(u).$ The measure $\xm_u$ is the regular part of $\text{Tr}_{\partial\xO}(u),$ the set $(\mathcal{S}_u)$ is its singular part.
\end{defin}
\begin{prop}\label{Lemm16}
Let $u$ be a positive solution in $\xO$ and let $\{\xO_n\}$ be an exhaustion of $\xO.$ If $y\in\mathcal{S}_u$ then for every nonnegative $Z\in C(\overline{\xO})$ such that $Z(y)\neq0$ we have
$$\lim_{n\rightarrow\infty}\int_{\partial\xO_n}Zud\gw^{x_0}_{\Gw_n}=\infty.$$
\end{prop}  
\begin{proof}
Let  $Z\in C(\overline{\xO})$, $Z\geq 0$, such that $Z(y)\neq0$  and
$$\liminf_{n\rightarrow\infty}\int_{\partial\xO_n}Zud\gw^{x_0}_{\Gw_n}<\infty.$$
There exists a subsequence $n_j$ such that
$$\lim_{j\rightarrow\infty}\int_{\partial\xO_{n_j}}Zud\gw^{x_0}_{\Gw_{n_j}}=M<\infty.$$
Let $r$ be such that $Z(x)>\frac{Z(y)}{2}$, $\forall x\in B_r(y)\cap\overline{\xO}$, then for any $r'<r$ we have that
$$\limsup_{j\rightarrow\infty}\int_{\partial\xO_{n_j}}[u]_{\overline{B_{r'}(y)\cap\partial\xO}}d\gw^{x_0}_{\Gw_n}<\infty.$$
In view of the proposition of \ref{pro2} the last fact implies that $|[u]_{\overline{B_{r'}(y)}}|^q\in L_{\ei}(\xO),$ which implies that
$$|u|^q\in L_{\ei}({B_{r''}(y)})\quad\forall r''<r'.$$
Which is clearly a contradiction, by Proposition \ref{pro3}.
\end{proof}
\begin{prop}\label{prop17}
Let $u$ be a positive solution of \eqref{T1} in $\xO$ with  boundary trace $(\mathcal{S}_u, \xm_u).$ Then
$$
\int_{\xO}(u \CL_{\xk }\gz+u^q\gz) dx=\int_\xO \mathbb{K}_{\CL_{\xk }}[\xm_u\chi_F]\CL_{\xk }\gz dx,
$$
for any $\gz\in\mathbf{X}(\xO)$ such that $\text{supp}(\gz)\cap\partial\xO\subset F.$
\end{prop}  
\begin{proof}

Consider the function $\xz\in\mathbf{X}(\xO)$ such that $\text{supp}(\xz)\cap\partial\xO\subset F.$ Set $K=\text{supp}(\xz)$,
$$O_\xe=\{x\in\mathbb{R}^N:\;\mathrm{dist}(x,K)<\xe\}$$
and $\xe_0>0$ small enough such that
$$\overline{O_\xe}\cap\partial\xO\subset\mathcal{R}_u,\qquad\forall \,0<\xe<\xe_0.$$
Let $\xe<\frac{\xe_0}{4}$ and $\eta$ be a cut off function such that $\eta\in C_0^\infty(O_\xe)$, $0\leq \eta\leq 1$
and $\eta\equiv1$ on $\overline{O}_{\frac{\xe}{2}}$.
For $0<\gb\leq\gb_0$, let $v_{\xb}$ be the solution of
$$\BA {ll}
\CL_{\xk }w +|w|^{q-1}w=0\qquad&\mathrm{in}\; \xO'_\xb\;\\
\phantom{\CL_{\xk }w +|w|^{q-1}}
w=\eta u\qquad&\mathrm{on}\;\xS_\xb.
\EA$$
Then there exists a sequence $\{\xb_j\}$ decreasing to $0$ such that $v_{\xb_j}\rightarrow v$ locally uniformly, and
$$v\leq [u]_{\partial\xO\cap\overline{O_\xe}}.$$ Thus $v$ has boundary trace $\xm_0$ such that
$$\xm_0\leq\xm_u\chi_{\partial\xO\cap\overline{O_\xe}}.$$
Let $v^1_\xb$ and $v^2_\xb$ be the solutions of
$$\BA {ll}
\CL_{\xk }w +|w|^{q-1}w=0\qquad&\mathrm{in}\; \xO'_\xb\;\\
\phantom{\CL_{\xk }w +|w|^{q-1}}
w=\eta[u]_{\partial\xO\cap\overline{O}_{2\xe}}\qquad&\mathrm{on}\;\xS_\xb.
\EA$$
and
$$\BA {ll}
\CL_{\xk }w +|w|^{q-1}w=0\qquad&\mathrm{in}\; \xO'_\xb\;\\
\phantom{\CL_{\xk }w +|w|^{q-1}}
w=\eta U_{\partial\xO\setminus O_{2\xe}}\qquad&\mathrm{on}\;\xS_\xb,
\EA$$
respectively.
Notice that $u\leq [u]_{\partial\xO\cap\overline{O}_{2\xe}}+U_{\partial\xO\setminus O_{2\xe}}$ we have that
$$v_\xb\leq v^1_\xb+v_\xb^2\leq[u]_{\partial\xO\cap\overline{O}_{2\xe}}+v_\xb^2.$$
Since $[u]_{\partial\xO\cap\overline{O}_{2\xe}}^q\in L^1_{\ei}(\xO).$ By \eqref{3.4.24} we have that
$$\eta (x)U_{\partial\xO\setminus O_{2\xe}}(x)\leq c_{90}d^{\frac{\xa_+}{2}}(x)\quad\forall x\in\xO.$$
where $c_{90}>0$ depends on $N,q,\gk$ and $\dist$(supp$(\eta),\prt\Gw\setminus O_\ge)$. Thus $v^2_\xb(x)\leq c_{90}d^{\frac{\xa_+}{2}}(x)$ and
\be
v_\xb\leq [u]_{\partial\xO\cap\overline{O}_{2\xe}}+c_{90}d^{\frac{\xa_+}{2}}(x),\;\;\forall x\in \xO'_\xb.
\ee
Let $w_\xb$ be the solution of
$$\BA {ll}
\CL_{\xk }w +|w|^{q-1}w=0\qquad&\mathrm{in}\; \xO'_\xb\;\\
\phantom{\CL_{\xk }w +|w|^{q-1}}
w=\chi_{\xS_\xb(\overline{\partial\xO\setminus O_{\frac{\xe}{2}}})}[u]_F\qquad&\mathrm{on}\;\xS_\xb.
\EA$$
Then  $$[u]_F\leq v_{\xb}+w_\xb\qquad\text {in } \xO'_\xb.$$
We have that $w_{\xb_j}\rightarrow0$ locally uniformly in $\xO$ which implies that
$$[u]_F\leq v.$$
Thus we have
\be
\xm_u\chi_F\leq\xm_0\leq\xm_u\chi_{\partial\xO\cap\overline{O_\xe}}.\label{measure}
\ee
Set $Z=\eta\xz_\xb$ where $\xz_\xb$ is the solution of
$$\BA {ll}
\CL_{\xk }w =\CL_{\xk } \xz\qquad&\mathrm{in}\; \xO'_\xb\\
\phantom{\CL_{\xk }}
w=0\qquad&\mathrm{on}\;\xS_\xb.
\EA$$
Since $\gz\in \mathbf{X}(\xO),$ there exists a constant $c_{91}$ such that $\xz_\xb\leq c_{91}\ei$ in $\xO'_\xb.$ Thus there exists a decreasing sequence $\{\xb_j\}$ converging to $0$ such that $\xz_{\xb_j}\rightarrow \xz$ locally uniformly.
 Now,
\begin{align}
\nonumber
\int_{\xO'_\xb}u \CL_{\xk }Zdx+\int_{\xO'_\xb} u^{q}Z dx&=-\int_{\partial \xO'_\xb} \frac{\partial Z}{\partial {\bf n}} udS\\ \nonumber
&=-\int_{\partial \xO'_\xb} \frac{\partial \xz_\xb}{\partial {\bf n}} \eta udS.\\ \nonumber
&=\int_{\xO'_\xb}v_\xb \CL_{\xk }\xz_\xb dx+\int_{\xO'_\xb}v_\xb^q \xz_\xb dx\\
&=\int_{\xO'_\xb}v_\xb \CL_{\xk }\xz dx+\int_{\xO'_\xb}v_\xb^q \xz_\xb dx,\label{ineq2}
\end{align}
We note here that in view of the proof of (\ref{3.4.24*}), we have
$$|\nabla\xz_\xb|\leq c_{92} d^{\frac{\xa_+}{2}-1},\quad\forall x\in \xO'_\xb,$$
where the constant $c_{92}>0$ does not depend on $\xb.$
Also by remark \ref{remark 1} and our assumptions we have
$$\int_{\xO\cap O_{2\xe}}[u]_{\partial\xO\cap\overline{O}_{3\xe}}d^{\frac{\xa_+}{2}-1}dx<\infty.$$
By \eqref{3.4.24}
$$
\int_{\xO\cap O_{2\xe}}U_{\partial\xO\setminus O_{3\xe}}d^{\frac{\xa_+}{2}-1}dx<\infty.
$$
The last two inequalities above implies that
$$
\int_{\xO\cap O_{2\xe}}ud^{\frac{\xa_+}{2}-1}dx\leq\int_{\xO\cap O_{2\xe}}U_{\partial\xO\setminus O_{3\xe}}d^{\frac{\xa_+}{2}-1}dx+\int_{\xO\cap O_{2\xe}}[u]_{\partial\xO\cap\overline{O}_{3\xe}}d^{\frac{\xa_+}{2}-1}dx<\infty.
$$
Combining all above we can choose a decreasing subsequence $\{\xb_j\}$ to the origin such that if we take the limit in \eqref{ineq2} to obtain
$$
\int_{\xO}u \CL_{\xk }\xz dx+\int_{\xO} u^q\xz dx=\int_{\xO}v \CL_{\xk }\xz dx+\int_{\xO}v^q \xz dx=\int_{\xO}\po[\xm_0] \CL_{\xk }\xz dx
$$
Be \eqref{measure} we have the desired result if we send $\xe$ to zero.
\end{proof}

\subsection{Subcritical case}
We recall that
$$q_{c } = \frac{N+\frac{\xa_+}{2}}{N+\frac{\xa_+}{2}-2}$$ is the critical exponent for the equation.
If $1<q<q_{c },$ we have seen in section 4 that for any $a\in\prt\Gw$  and $k\geq 0$ there exists $u_{k\gd_a}$
and $\lim_{k\to\infty}u_{k\gd_a}=u_{\infty,a}$. Furthermore, by Proposition \ref{Lemm16}, $Tr_{\prt\Gw}(u_{\infty,a})=(\{a\},0)$.

\begin{theorem}\label{singpro1}
Assume $1 < q < q_{c}$ and $a\in \mathcal{S}_u.$ Then
\be\label{T9}u(x)\geq u_{\infty,a}(x),\quad\forall x\in \xO.\ee
\end{theorem}
For proof of the above uses some ideas of the proof of Theorem 7.1 in \cite{MV-CPAM} and needs several intermediate lemmas.
\begin{lemma}\label{7.3}
Assume $1 < q < q_{c}.$
Let $\{\xi^n\}$ be a sequence of
points in $\xO$ converging to $a\in\partial\xO$ and let $l\in (0, 1).$ We define the sets
\be\label{TX}\Gw_n:=\Gw'_{d(\xi^n)}=\{x\in\xO:\;d(x)>d(\xi^n)\} \quad\text{and}\quad \xS_n:=\partial \Gw_n.\ee
Let $x_0\in \Gw'_1$ and denote by $\gw_n:=\gw^{x_0}_{\Gw_n}$ the $\CL_{\xk }$-harmonic measure
in $\xO_n$ relative to $x_0.$
Put
$$V_n=B_{lr_n}(\xi^n)\cap\partial \Gw_n,\quad r_n=d(\xi_n).$$
Let $h_n \in L^\infty(\xS_{n})$ $n = 1, 2, . . . ,$ and suppose that there exist numbers $c$ and $k$
such that
\be
\text{supp}\,(h_n)\subset V_n\quad\text{and}\quad0\leq h_n\leq cr_n^{-N-\frac{\xa_+}{2}+2}\label{7.8}
\ee
and
$$\lim_{n\rightarrow\infty}\int_{\xS_{n}} h_n\xf d\gw^{x_0}_{\Gw_n}=k\xf(a),\quad\forall\xf\in C(\overline{\xO}).$$
Let $w_n$ be the solution of the problem
\begin{align}\nonumber
\CL_{\xk }w_n+|w_n|^{q-1}w_n&=0\qquad\;\;\mathrm{in}\;\;\Gw_n\\ \nonumber
w_n&=h_n\qquad\mathrm{on}\;\;\partial\xS_n.
\end{align}
Then $$w_n\rightarrow u_{k,a}\quad\text{locally uniformly in $\xO$}.$$
\end{lemma}  
\begin{proof}
Let $\eta^n\in\partial\xO$ be such that $d(\xi^n)=|\xi^n-\eta^n|.$ By Corollary \ref{poisson} we have
\be\label{7.13}
\po(x, \eta^n)\geq \frac{1}{c_{43}}r_n^{-N-\frac{\ga_+}{2}+2}\geq \frac{1}{c_{43}}h_n(x),\qquad \forall x \in \xS_n,
\ee
by the maximum principle,
\be
\po(x, \eta^n) \geq \frac{1}{c_{43}} w_n(x),\quad \forall x \in \Gw_n.\label{7.14}
\ee
Moreover
$$\int_\xO \po^q(x,y)d^{\frac{\xa_+}{2}}(x)dx\leq c(q,\xO),\quad\forall 1<q<q_{c },$$
where $c(q,\xO)$ is a constant independent of $y.$ Since $q$ is subcritical, it
follows that the sequences $\{\po^q ( \cdot , \eta^n)\}$ and $\{\po( \cdot, \eta^n\}$ are uniformly integrable
in $L^1_{\ei}(\xO).$ Let $\overline{w}_n$ denotes the extension of $w_n$ to $\xO$ defined by $\overline{w}_n = 0$ in $\xO \setminus \Gw_n.$ In
view of (\ref{7.13}) we conclude that the sequences $\{ \overline{w}_n^q\}$ and $\{\overline{w}_n\}$ are uniformly
integrable in $L^1_{\ei}(\xO),$ and locally uniformly bounded in $\Gw$
By regularity results for elliptic equations there exists a subsequence of $\{\overline{w}_n\},$ say again $\{\overline{w}_n\}$ that converges locally
uniformly in $\xO$ to a solution $w$ of (\ref{T1}). This fact and the uniform integrability
mentioned above imply that
$$w_{n}\rightarrow w\quad\text{in}\;\;L^q_{\ei}(\xO)\cap L^1_{\ei}(\xO).$$
Since $w\in L^q_{\ei}(\xO)$ by Proposition \ref{pro1} there exists $\xm\in\mathfrak{M}(\xO)$ such that
\be\nonumber
\int_{\xO}w \CL_{\xk }\eta dx+\int_\xO |w|^{q-1}w\eta dx=\int_\xO \mathbb{K}_{\CL_{\xk }}[\xm]\CL_{\xk }\eta dx\qquad\forall \eta\in\mathbf{X}(\xO).
\ee
Furthermore, using \eqref{7.13} we prove below that measure $\xm$ is concentrated at $a.$
Let $\xf_{\xk ,n}$ be the first eigenfunction of $\CL_{\xk }$ in $\xO_n$ normalized by $\xf_{\xk ,n}({x_0})=1$ for some $x_0\in\xO_1.$ Let $\eta\in \mathbf{X}(\xO)$ be nonnegative function and let $\eta_n$ be the solution of the problem
$$\BA {lll}\CL_{\xk }\eta_n=\frac{\xf_{\xk ,n}}{\ei}\CL_{\xk }\eta\qquad&\text{in } \Gw_n\\
\phantom{\CL_{\xk }}
\eta_n=0\qquad&\text{in }\partial \Gw_n.\EA$$
Then $\eta_n\in C^2(\overline\xO_n)$ and since $\xf_{\xk ,n}\rightarrow\ei,$
\be\nonumber
\CL_{\xk }\eta_n\rightarrow \CL_{\xk }\eta\;\text{ and }\;\eta_n\rightarrow\eta\;\text{ as }\; n\to\infty.
\ee
Then we have
\be
\int_{\Gw_n}w_n \CL_{\xk }\eta_n dx+\int_\xO |w_n|^{q-1}w\eta dx=\int_\xO v_n\CL_{\xk }\eta_n dx,\label{ww}
\ee
where $v_n$ solves
\begin{align}\nonumber
\CL_{\xk }v_n&=0\qquad\,\text{ in}\;\;\Gw_n\\ \nonumber
v_n&=h_n\qquad\text{on}\;\;\partial\xS_n.
\end{align}
By the same arguments as above there exists a subsequence of $\{v_n\chi_{\Gw_n}\},$ for simplicity $\{v_n\chi_{\Gw_n}\}$, converging  in $L^1_{\ei}(\xO)$ to a a nonnegative $\CL_{\xk }$-harmonic function $v$.
By (\ref{7.13}) we have
\be
cc_{43}K_{\CL_\gk}(x, a) \geq v(x),\quad \forall x \in \xO.\label{7.13*}
\ee
Thus there exists a measure $\xn\in \mathfrak{M}(\partial\xO),$ concentrated at $a$ such that $v$ solves
 \begin{align}\nonumber
\CL_{\xk }v&=0\qquad\mathrm{in}\;\;\xO\\ \nonumber
v&=\xn\qquad\mathrm{on}\;\;\partial\xO.
\end{align}
But
$$k=\lim_{n\rightarrow\infty}\int_{\xS_{n}} h_nd\gw^{x_0}_{\Gw_n}=\lim_{n\rightarrow\infty}v_n(x_0)=v(x_0)=\int_{\prt\Gw} d\xn,$$
the results follows if we sent $n$ to $\infty$ in \eqref{ww}.
\end{proof}
\begin{lemma}\label{harnack}
For every $l \in (0, 1)$ there exists a constant $c_l = c(N, \xk , q,l)$ such
that, for every positive solution u of $(\ref{T1})$ in $\xO$ and every $x_0\in\xO,$

\be\label{T10}u(x) \leq c_lu(y),\quad \forall x, y \in B_{lr_0}(x_0) ,\quad r_0 = d(x_0) .\ee
\end{lemma}  
\begin{proof}
Put $r_1=\frac{1+l}{2}r_0.$ Then u satisfies
$$\CL_{\xk }u+u^q=0,\quad\text{in}\;\; B_{r_1}(x_0).$$
Denote by
$\xO_{r_0}$ the domain
$$\xO_{r_0}=\{y\in\mathbb{R}^n:r_0y\in\xO\}.$$
Set $v(y)=u(r_0y),$ and $y_0=\frac{x_0}{r},$ then $v(y)$ satisfies
 $$-\xD v-\xk \frac{v}{\text{dist}^2(y,\partial\xO_{y_0})}+r_0^2|v|^{q-1}v=0,\quad\text{in}\;\; B_{\frac{1+l}{2}}(y_0).$$
Now note that
$$\frac{1}{\text{dist}^2(y,\partial\xO_{y_0})}\leq \frac{4}{(1-l)^2},\quad\forall y\in B_{\frac{1+l}{2}}(y_0)$$
and by Keller Osserman condition
$$r_0^2|v(y)|^{q-1}=r_0^2|u(r_0y)|^{q-1}\leq C(\xO,\xk ,N)r_0^2\frac{1}{d^2(r_0y)}\leq C(\xO,\xk ,N)B_{\frac{1+l}{2}}(y_0).$$
Thus by Harnack inequality there exists a constant $c_l>0$ such that
$$v(z) \leq c_lv(y),\quad \forall z, y \in B_{l}(y_0),$$
and the results follows.
\end{proof}

For the proof of the next lemma we need some notations. Let $\xb>0$ and $\xi\in\Gs_\gb=\prt\Gw'_\gb.$ We set $\xD_r^\xb(\xi)=\Gs_\gb\cap B_r(\xi)$ and, for $0<r<\gb< 2r$,
$x_r^\xb=x_r^\xb(\xi)\in\overline\xO_\xb,\;$ such that $d(x_r^\xb)=|x_r^\xb-\xi|=r.$
Also we denote by $\gw^{x}_{\Gw'_\xb}$ the $\CL_{\xk }$-harmonic measure in $\xO'_\xb:=\xO\setminus \overline{\xO}_\xb$ relative to $x$
\begin{lemma}\label{lem2.1*}
Let $r_0=r_0(\xO)>0$ be small enough and $0<r\leq \frac{r_0}{4}.$ Then there exists a constant $c_{95}$ which depends only on $\xO,N$ such that
\be\label{T11}\gw^{x}_{\Gw'_\xb}(\xD_r(\xi))>c_{95}\qquad\forall x\in\xO\cap B_{\frac{r}{2}}(\xi).\ee
\end{lemma}  
\begin{proof}
Since $x\mapsto \gw^{x}_{\Gw'_\xb}$ is a positive $\CL_{\xk }$-harmonic in $\xO'_\xb,$ it is a positive superharmonic function (relative to  the Laplacian) in $\xO'_\xb.$ Thus
$$\gw^{x}_{\Gw'_\xb}\geq \gu^{x}_{\Gw'_\xb},\quad\forall x\in \xO'_\xb, $$
where $\gu^{x}_{\Gw'_\xb}$ is the standard harmonic measure in $\xO'_\xb$ relative to $x\in\Gw'_\gb$
The result follows by Lemma 2.1 in \cite{caffa}.
\end{proof}
\begin{lemma}\label{vol1}
Let $\xk =\frac{1}{4},$ $\xe\in(0,1)$ and $x_0\in \Gw_1$. Let $\{\xi^n\}$ be a sequence of
points in $\xO$ converging to $a\in\partial\xO.$ Then there exist $n_0=n_0(\xe,\xO)\in\BBN$ and $c_{96}=c_{96}(\xO,N,\xe)$ such that
\be\label{T12}\gw^{x_0}_{\Gw_n}(B_{d(\xi^n)}(\xi^n)\cap \partial \Gw_n)\geq c_{96}d(\xi^n)^{N+\frac{1}{2}-2}(-\log d(\xi^n))^{1-\xe}\qquad\forall n\geq n_0.\ee
\end{lemma}  
\begin{proof}
We recall that for any $n\in\BBN$ $\Gw_n$ is defined by (\ref{TX}), ${G}_{\CL_{\frac{1}{4}}}^{\Gw_n}\leq {G}_{\CL_{\frac{1}{4}}}:={G}_{\CL_{\frac{1}{4}}}^{\Gw}$,
and for a fixed point $y_0\in \Gw_1$
\be
{G}_{\CL_{\frac{1}{4}}}^{\Gw_n}\chi_{\Gw_n}(x)\uparrow {G}_{\CL_{\frac{1}{4}}}(x,y_0),\quad\text{locally uniformly in }\xO\setminus{y_0}.\label{gr2}
\ee
Set $x(\xi^n)=x_{\frac{r_n}{2}}^{2r_n}(\xi^n),$ with $r_n=\frac{d(\xi^n)}{2}.$
By \eqref{greenest} we have
$$r_n^{N-2}{G}_{\CL_{\frac{1}{4}}}^n(x,x(\xi^n))<c_{97},\quad\forall x\in \Gw_n\cap \partial B_{r_n}(\xi^n),$$
and by Lemma \ref{lem2.1*} there exists $r_0=r_0(\xO)>0$ such that for any $r_n\leq \frac{r_0}{4}$
$$r_n^{N-2}{G}^{\Gw_n}_{\CL_{\frac{1}{4}}}(x,x(\xi^n))\leq c_{98} \gw^x_{\Gw_n}(\partial \Gw_n\cap B_{r_n}(\xi^n)),\quad\forall x\in \Gw_n\cap \partial B_{r_n}(\xi^n).$$
Since if $|x-y|>\xe>0$ there holds
$${G}^{\Gw_n}_{L_{\frac{1}{4}}}(x,y)\approx c_{99}(\xe,\Gw_n)\mathrm{dist}(x,\partial \Gw_n)\mathrm{dist}(y,\partial \Gw_n).$$
Thus we have by maximum principle and properties of Green function
\be
r_n^{N-2}{G}^{\Gw_n}_{\CL_{\frac{1}{4}}}(x,x(\xi^n))\leq c_{100} \gw^x_{\Gw_n}(\partial \Gw_n\cap B_{r_n}(\xi^n)),\quad\forall x\in \Gw_n\setminus B_{r_n}(\xi^n).\label{gr3}
\ee
By \cite[Lemma 2.8]{1} there exists $\xb_0=\xb_0(\xO,\xe)>0$ such that the function
$$h_1(x)=d^{\frac{1}{2}}(x)(-\log d(x))\left(1+\left(-\log d(x)\right)^{-\xe}\right),$$
is a supersolution in $\xO_{\xb_0}$ and the function
$$h_2(x)=d^{\frac{1}{2}}(x)(-\log d(x))\left(1-\left(-\log d(x)\right)^{-\xe})\right),$$
is a subsolution in $\xO_{\xb_0}.$ Set
$$c_{101}=\frac{1-\left(-\log d(\xi_n)\right)^{-\xe}}{1+\left(-\log d(\xi_n)\right)^{-\xe}}$$
and
$$H(x)=h_2(x)-c_{101}h_1(x).$$
Let $n_0\in \mathbb{N}$ such that $r_n\leq \frac{\xb_0}{4},\;\forall n\geq n_0.$
then the function $H(x)$ is a nonnegative subsolution in $\Gw_n\setminus\xO'_{\xb_0},$
and $H(x)=0,\;\forall x\in \partial \Gw_n.$
By \eqref{gr2} we can choose $n_1\in \mathbb{N}$ such that
$${G}_{\CL_{\frac{1}{4}}}^{\xO_n}(x_0,x)\geq c(\xO,N,\xk)\xb_0^{\frac{1}{2}},\quad\forall x \in \partial \Gw'_{\xb_0}.$$
Thus we can find  a constant $c_{102}=c_{102}(\xb_0)>0$ such that
$$c_{102} H(x)\leq {G}_{\CL_{\frac{1}{4}}}^{\xO_n}(x_0,x),\quad\forall x \in \partial \Gw'_{\xb_0}. $$
Since $H$ vanishes on $\prt\Gw_n$ it follows by the by maximum principle that
\be
c_{102} H(x)\leq {G}_{\CL_{\frac{1}{4}}}(x_0,x),\quad\forall x \in \overline\Gw_n\setminus\xO'_{\xb_0}.\label{gr4}
\ee
But $$H(x(\xi^n))\geq c_{103}(\xb_0)\geq c_{104}(\xO,N)r_n^{\frac{1}{2}}(-\log r_n)^{1-\xe}$$
and the result follows by the above inequality and inequalities \eqref{gr4} and \eqref{gr3}.
\end{proof}

\begin{lemma}\label{vol2}
Let $\xk <\frac{1}{4},$ $\xe\in\left(0,\sqrt{1-4\xk }\right)$ and $x_0\in \Gw_1$. Let $\{\xi^n\}$ be a sequence of
points in $\xO$ converging to $a\in\partial\xO.$ Then there exists $n_0=n_0(\xe,\xO)\in\BBN$ such that
$$\gw_{\Gw_n}^{x_0}\left(B_{d(\xi^n)}(\xi^n)\cap \partial \Gw'_n\right)\geq c_{105}(\xO,N,\xk ,\xe)d(\xi^n)^{N+\frac{\xa_-}{2}+\xe-2},\quad\forall n\geq n_0,$$
where $\Gw_n$ is defined by (\ref{TX})
\end{lemma}  
\begin{proof}
The proof is same as in Lemma \ref{vol1}. The only difference is that we use $d^{\ga_-}(1-d^{\xe})$ and the supersolution
$d^{\ga_-}(1+d^{\xe})$ as a subsolution.
\end{proof}

\noindent{\it{Proof of Theorem \ref{singpro1}.}}
{\it Step 1}: if
\be
\limsup_{x\in\xO,\;x\rightarrow a}(d(x))^{N+\frac{\xa_+}{2}-2}u(x)<\infty,\label{assert1}
\ee
then $a \in \mathcal{R}_u$. Thus we have to prove that there exists $r_0>0$ such that $u\in L^q_{\ei}(\xO\cap B_{r_0}(a)).$
By \eqref{assert1} there exists $r_1>0$ such that
$$\sup_{x\in\xO\cap B_{r_1}(a)}d^{N+\frac{\xa_+}{2}-2}(x)u(x)=M<\infty.$$
Let $U$ be a smooth open domain such that
$$\xO\cap B_{\frac{r_1}{2}}(a)\subset U\subset\xO\cap B_{r_1}(a),$$
and
$$\overline{U}\cap\partial\xO\subset\partial\xO\cap B_{r_1}(a).$$
For $\gb>0$, set
$$d_U(x)=\text{dist}(x,\partial U)\;\;\;\forall x\in U,\quad U_\xb=\{x\in U:\;d_U(x)>\xb\},\;\quad V_\xb= U\setminus U_\xb.$$
Let $\xb_0>0$ be small enough such that $d_U\in C^2(\overline{U}_{\xb_0}).$ Let $0<\xb<\xb_0$ and $\xz(x)=d_U(x)-\xb.$ Then $u$ satisfies
$$\int_{\partial V_\xb}u dS=\int_{V_\xb\setminus V_{\xb_0}}(u \CL_{\xk }\xz+u^q\xz) dx-\int_{\partial V_{\xb_0}}\frac{\partial u}{\partial {\bf n}}\xz dS.$$
Now
$$\left|\int_{\partial V_{\xb_0}}\frac{\partial u}{\partial {\bf n}}\xz dS\right|\leq c_{106}(\gb_0-\gb),$$
where $c_{106}$ depends on $q,\gk,\Gw,\gb_0$,
$$\int_{V_\xb\setminus V_{\xb_0}}u \CL_{\xk }\xz dx \leq-\int_{V_\xb\setminus V_{\xb_0}}u \xD\xz dx\leq c_{107}\int_{V_\xb\setminus V_{\xb_0}}udx.
$$
and by \eqref{assert1}
$$u^{q-1}(x)\leq c_{108}(d(x))^{-(q-1)(N+\frac{\xa_+}{2}-2)}\leq c_{108}(d_U(x))^{-(q-1)(N+\frac{\xa_+}{2}-2)}\quad\forall x\in U.$$
Combining the above inequalities,  we derive
$$\int_{\partial V_\xb}u dS\leq c_{109}\left(\int^{\xb_0}_\xb (\xs^{1-(q-1)(N+\frac{\xa_+}{2}-2)}+1)\int_{\partial V_\xs}u(x)dSd\xs+1\right).$$
Multiplying the above inequality by $\xb^{\frac{\xa_+}{2}}$ we get
$$\int_{\partial V_\xb}ud_U^{\frac{\xa_+}{2}} dS\leq c_{109}\left(\int^{\xb_0}_\xb (\xs^{1-(q-1)(N+\frac{\xa_+}{2}-2)}+1)\int_{\partial V_\xs}d_U^{\frac{\xa_+}{2}}(x)u(x)dSd\xs+1\right).$$
Set $$U(\xs)=\int_{\partial V_\xs}d_U^{\frac{\xa_+}{2}}(x)u(x)dS,$$
Then we have
\be
U(\xb)\leq c_{110}\left(\int^{\xb_0}_\xb (\xs^{1-(q-1)(N+\frac{\xa_+}{2}-2)}+1) U(\xs)d\xs+1\right),\label{in11}
\ee
Set $$W(\xb)=\int^{\xb_0}_\xb (\xs^{1-(q-1)(N+\frac{\xa_+}{2}-2)}+1) U(\xs)d\xs+1,$$
then
$$W'(\xb)=-(\xb^{1-(p-1)(N+\frac{\xa_+}{2}-2)}+1) U(\xb)=-h(\xb)U(\xb).$$
Thus inequality \eqref{in11} becomes
$$-W'(\xb)\leq c_{110}h(\xb)W(\xb)\Longleftrightarrow \left(H(\xb)W(\xb)\right)'\geq0,$$
where
$$H(\xb)=e^{-c_{110}\int_\xb^{\xb_0}h(s)ds}.$$
Thus we have
$$W(\xb)\leq \frac{1}{H(\xb)}W(\xb_0),\quad\forall 0<\xb< \xb_0.$$
But $$\frac{1}{H(\xb)}=e^{c_{110}\int_\xb^{\xb_0}h(s)ds}=e^{c_{110}\int_\xb^{\xb_0}\xs^{1-(q-1)(N+\frac{\xa_+}{2}-2)}+1ds}<\infty$$
if and only if $$2-(q-1)(N+\frac{\xa_+}{2}-2)>0\Longleftrightarrow q<q_{c }.$$
Thus we have proved that
$$\int_Uu^q(d_U(x))^{\frac{\xa_+}{2}}dx<\infty,$$
which implies the existence of a $r_2>0$ such that
$$\int_{\xO\cap B_{r_2}(a)}u^q(d(x))^{\frac{\xa_+}{2}}dx<\infty,$$
i.e. $a\in \mathcal{R}_u,$ which is the claim.\smallskip

\noindent{\it Step 2}. Since $a\in \CS_u$ the previous statement implies that there exists a sequence $\{\xi^n\}\subset \xO$ such that
\be
\xi^n\rightarrow a\quad \text{and}\quad \limsup_{n\rightarrow \infty}(d(\xi^n))^{N+\frac{\xa_+}{2}-2}u(\xi^n)=\infty.\label{7.27}
\ee
By Lemma \ref{harnack}, there exists a constant $c_l$ such that
\be
u(x) \leq c_lu(y),\quad \forall x, y \in B_{\frac{r_n}{2}}(\xi^n) ,\quad r_n =d(\xi^n) .\label{7.28}
\ee
Put
$V_n := B_{\frac{r_n}{2}}(\xi^n)\cap \partial \Gw'_{r_n},$ and, for $k > 0$, $h_{n,k} := \frac{k}{b_n}u\chi_{V_n}$.

\noindent{\it Case 1}: $\xk =\frac{1}{4}.$
By (\ref{7.28}) and Lemma \ref{vol1} there exists a constant $c_{111} > 0$ such that
$$
b_n :=\int_{V_n} u dS\geq c_{111}A_n r_n^{N+\frac{1}{2}-2}(-\log r_n)^{1-\xe},\quad A_n := \sup_{x\in B_{\frac{r_n}{2}}(\xi^n)}u(x).
$$
Then
\be
\int_{\partial \Gw'_n}h_{n,k} dS = k ,\quad h_{n,k} \leq \frac{k}{c_2}r_n^{2-\frac{\xa_+}{2}-N}\chi_{V_n},\quad\forall n\geq n_0.\label{7.29}
\ee
By \eqref{7.27},
\be
b_n \rightarrow\infty,\quad r_n \rightarrow 0 .\label{7.30}
\ee
Hence, for every $k > 0$ there exists $n_k$ such that
\be
u \geq h_{n,k}\qquad \text{on } \;\partial \Gw'_n\quad \forall n \geq n_k .\label{7.31}
\ee
Let $w_{n,k}$ be defined as in Lemma \ref{7.3} with $h_n$ replaced by $h_{n,k} .$ By (\ref{7.29}) and
(\ref{7.30}), the sequence $\{h_{n,k}\}_{n=1}^\infty$
satisfies (\ref{7.8}) for every fixed $k > 0.$ Therefore by
Lemma \ref{7.3}
$$\lim_{n\rightarrow\infty} w_{n,k} = u_{k\gd_a} \text{ locally uniformly in}\; \xO.$$
By (\ref{7.31}), $u \geq w_{n,k}$ in ${x\in\xO : d(x) > r_n}.$ Hence $u\geq u_{k\gd_a}$ for every
$k > 0.$ The proof in the case $0<\gk<\frac{1}{4}$ is similar.  \hfill$\Box$\medskip

As a consequence we provide a full classification of positive solution of (\ref{Eq1}) with a boundary isolated singularity.
\begin{theorem}\label {class} Assume $1<q<q_{c }$ and $u\in C(\overline\Gw\setminus\{0\})$ is a positive solution of $(\ref{Eq1})$ 
which satisfies 
$$\lim_{x\in\xO,\;x\rightarrow\xi}\frac{u(x)}{W(x)}=0,\quad\forall \xi\in\partial\xO\setminus \{0\}.$$
 Then the following alternative holds\smallskip

\noindent (i) Either there exists $k\geq 0$ such that
      \begin{equation}\label{Sg1}\BA {ll}
\displaystyle\lim_{\tiny{\BA {cc}x\to 0,x\in\Gw\\
x|x|^{-1}\to\gs\EA}}|x|^{N+\frac{\ga_+}{2}-2}u(x)=k\psi_1(\gs)
  \EA \end{equation}
  and $u$ solves
        \begin{equation}\label{Sg2}\BA {ll}
-\Gd u-\myfrac{\xk }{d^2}u+u^q=0\qquad&\text{in }\Gw\\[2mm]
\phantom{-\Gd-\myfrac{\xk }{d^2}u+u^q}
u=k\gd_0\qquad&\text{in }\prt\Gw.
  \EA \end{equation}
  \smallskip

\noindent (ii) Or
        \begin{equation}\label{Sg3}\BA {ll}
\displaystyle\lim_{\tiny\BA {cc}x\to 0,x\in\Gw\\
x|x|^{-1}\to\gs\EA}| x|^{\frac{2}{q-1}}u(x)=\gw_{\xk }(\gs)
  \EA \end{equation}
  locally uniformly on $S^{N-1}_+$.
  \smallskip
\end{theorem}

The result is a consequence of the following result

\begin{lemma}\label{U_0} Assume $1<q<q_c$, $a\in\prt\Gw$ and $F_\ge(a)=\prt\Gw\cap\overline {B_\ge(a)}$. Then
        \begin{equation}\label{SX1}\BA {ll}
       \displaystyle \lim_{\ge\to 0}U_{F_\ge(a)}=u_{\infty,a}.
  \EA \end{equation}
\end{lemma}
\begin{proof} Without loss of generality, we can assume $a=0$. Clearly, $U_{\{0\}}:=\lim_{\ge\to 0}U_{F_\ge(0)}$ is a solution of (\ref{T1}) which satisfies
$$\lim_{x \to \xi}\frac{U_{\{0\}}}{W(x)}=0\qquad\forall \xi\in\prt\Gw\setminus\{0\}
$$
locally uniformly on $\prt\Gw\setminus\{0\}$. By (\ref{3.4.24}) it verifies
        \begin{equation}\label{SX2}\BA {ll}
U_{\{0\}}(x)\leq c|x|^{-\frac{2}{q-1}}\left(\frac{d(x)}{|x|}\right)^{\frac{\ga_+}{2}}.
  \EA \end{equation}

 By Proposition \ref{IS-3} and \eqref{cor6}, we can follow the same argument like in the proof of Theorem 3.4.6-(ii) in \cite{book} to prove that: there exists $c_0=c_{112}(N,\gk,q)>1$ such that
$$ 
 \frac{1}{c_0}|x|^{-\frac{2}{q-1}}\left(\frac{d(x)}{|x|}\right)^{\frac{\ga_+}{2}}\leq u_{\infty,0}(x)\leq U_{\{0\}}(x)\leq c_0|x|^{-\frac{2}{q-1}}\left(\frac{d(x)}{|x|}\right)^{\frac{\ga_+}{2}}
$$ 
Which implies 
        \begin{equation}\label{SX3}\BA {ll}
U_{\{0\}}(x)\leq cu_{\infty,0}(x)\qquad\forall x\in\Gw,
  \EA \end{equation}
 where $c=c_{122}(N,\gk,q)>1.$
  
Assume $U_{\{0\}}\neq u_{\infty,0}$, thus $U_{\{0\}}(x)> u_{\infty,0}(x)$ for all $x\in\Gw$  and put
$\tilde u=u_{\infty,0}-\frac{1}{2c}(U_{\{0\}}- u_{\infty,0})$. By convexity $\tilde u$ is a supersolution of (\ref{T1}) which is smaller than $u_{\infty,0}$. Now $\frac{c+1}{2c}u_{\infty,0}$ is a subsolution, thus there exists a solution $\underline u$ of (\ref{T1}) in
$\Gw$ which satisfies
        \begin{equation}\label{SX4}\BA {ll}
\myfrac{c+1}{2c}u_{\infty,0}(x)\leq \underline u(x)\leq \tilde u(x)<u_{\infty,0}(x)\qquad\forall x\in\Gw.
  \EA \end{equation}
  This implies that $Tr_{\prt\Gw}(\underline u)=(\{0\},0)$, and by Theorem \ref{singpro1}, $\underline u\geq u_{\infty,0}$, which is a contradiction.
\medskip

\noindent{\it Proof of Theorem \ref{class}} Assume $a=0$ without loss of generality. If $a\in \CS_u$, then for any $\ge>0$,
$u\leq U_{F_\ge(0)}$ which is a maximal solution which vanishes on $\prt\Gw\setminus F_\ge(0)$. Thus, using (\ref{SX1})
$$u\leq      \displaystyle \lim_{\ge\to 0}U_{F_\ge(0)}=U_{\{0\}}=u_{\infty,0}.$$
If $0\in \CR_u$, this implies that $Tr_{\prt\Gw}( u)=(\emptyset,k\gd_0)$ for some $k\geq 0$ and we conclude with Corollary
\ref{IS-2}.
\end{proof}\medskip

The next result can be proven by using the same approximation methods as in \cite[Th 9.6]{MV-CPAM}.

\begin{theorem}\label{Borel}. Assume $\CS\subset\prt\Gw$ is closed and $\gn$ is a positive Radon measure on $\CR=\prt\Gw\setminus\CS$. Then there exists a positive solution of (\ref{Eq1}) in $\Gw$ with boundary trace $(\CS,\gm)$.
\end{theorem}

\section{Appendix I: barriers and a priori estimates}
\setcounter{equation}{0}

\subsection{Barriers}\label{Barrier}
Following a localization principle introduced in \cite{MV-CPAM} we the following lemma is at the core of the a priori estimates construction
\begin{prop}\label{barr} Let $\Gw\subset\BBR^N$ be a $C^2$ domain $0<\xk \leq\frac{1}{4}$ and $p>1$.Then there exists $R_0>0$ such that for any $z\in \prt\Gw$ and $0<R\leq R_0$, there exists a super solution $f:=f_{R,z}$ of $(\ref{Eq1})$ in $\Gw\cap B_R(z)$ such that $f\in C(\overline\Gw\cap B_R(z))$, $f(x)\to \infty$ when $\dist (x,K)\to 0$, for any compact subset $K\subset \Gw\cap \prt B_R(z)$ and which vanishes on $\prt\Gw\cap B_R(z)$, and more precisely
\begin{equation}\label{BAR1}
f(x)= \left\{\BA{lll}c_{\gb,\gamma,\gk,q}(R^2-|x-z|^2)^{-\gb}d^{\gamma}(x)\quad\forall \gamma\in (\frac{\ga_-}{2},\frac{\ga_+}{2})\,&\text{ if }0<\gk<\frac{1}{4}\\[2mm]
c_{\gb,\gamma,q}(R^2-|x-z|^2)^{-\gb}\sqrt{d(x)}\sqrt{\ln \left(\frac{{\rm {diam}}(\Gw)}{d(x)}\right)}&\text{ if }\gk=\frac{1}{4}
\EA\right.
\end{equation}
for $\gb\geq \max\{\frac{2}{q-1}+\gg,\frac{N-2}{2},1\}$.
\end{prop}  
\begin{proof} We assume  $z=0$
\smallskip

\noindent {\it Step 1: $\xk <\frac{1}{4}$. } Set
$f(x)=\Gl(R^2-| x|^2)^{-\gb} (d(x))^{\gamma}$ where $\gb,\gamma>0$ to be chosen later on. Then, with $r=|x|$,
$$\BA {lll}
\Gl^{-1}\CL_{\xk }f\\[4mm]
=-(R^2-r^2)^{-\gb}\left(\Gd d^{\gamma}+\xk d^{\gamma-2}\right)-d^{\gamma}\Gd(R^2-r^2)^{-\gb}-2
\nabla(R^2-r^2)^{-\gb}\nabla d^{\gamma}\EA$$
Since $\Gd d(x)=(N-1)H_d$ where $H_d$ is the mean curvature of the foliated set $\Gs_d:=\{x\in\Gw:d(x)=d\}$ and $|\nabla d|^2=1$,

$$\BA {ll}\Gd d^{\gamma}=(N-1)\gamma H_d  d^{\gamma-1}+\gamma(\gamma-1)  d^{\gamma-2}
\EA$$
$$\Gd d^{\gamma}+\xk d^{\gamma-2}=(N-1)\gamma H_d  d^{\gamma-1}+\left(\gamma(\gamma-1)+\xk \right) d^{\gamma-2}
$$
$$\nabla d^{\gamma}=\gamma d^{\gamma-1}\nabla d,
$$

$$\nabla(R^2-r^2)^{-\gb}=2\gb(R^2-r^2)^{-\gb-1}x,
$$

thus
$$\nabla(R^2-r^2)^{-\gb}\nabla d^{\gamma}=2\gb\gamma  d^{\gamma-1}(R^2-r^2)^{-\gb-1}x\nabla d
$$

$$\BA {ll}\Gd(R^2-r^2)^{-\gb}=2N\gb(R^2-r^2)^{-\gb-1}+4\gb(\gb+1)(R^2-r^2)^{-\gb-2}r^2\\[4mm]
\phantom{\Gd(R^2-r^2)^{-\gb}}
=2\gb(R^2-r^2)^{-\gb-2}\left(NR^2+(2\gb+2-N)r^2\right)
\EA$$
Then

$$\BA {lll}
\Gl^{-1}\CL_{\xk }f=-(R^2-r^2)^{-\gb-2} d^{\gamma-2}\left[(R^2-r^2)^{2}\left((N-1)\gamma H_d  d+\gamma(\gamma-1)+\xk \right)\right.
\\[4mm]\phantom{\Gl^{-1}\CL_{\xk }f}\left.
+2\gb d^2\left(NR^2+(2\gb+2-N)r^2\right)+4\gb\gamma d(R^2-r^2)x\nabla d^{\phantom{^4}}
\right]
\EA$$
Therefore
\begin{equation}\label{F3}\BA {ll}\CL_{\xk }f+f^q=\Gl(R^2-r^2)^{-\gb-2}d^{\gamma-2}\left[\Gl^{q-1}(R^2-r^2)^{-(q-1)\gb+2}d^{(q-1)\gamma+2}
\right.
\\[4mm]\phantom{\CL_{\xk }(f)+f^q}\left.
-(R^2-r^2)^{2}\left((N-1)\gamma H_d  d+\gamma(\gamma-1)+\xk \right)
\right.
\\[4mm]\phantom{\CL_{\xk }(f)+f^q}\left.
-2\gb d^2\left(NR^2+(2\gb+2-N)r^2\right)+4\gb\gamma d(R^2-r^2)x\nabla d^{\phantom{^4}}
\right]\EA\end{equation}
If we fix $\gb\geq \max\{\frac{2}{q-1}+\gamma, \frac{N-2}{2},1\}$, there holds
$$2\gb d^2\left(NR^2+(2\gb+2-N)r^2\right)+4\gb\gamma d(R^2-r^2)x\nabla d
\leq 4d^2\gb(\gb+1)NR^2+4\gb\xg dR(R^2-r^2)
$$
We choose $\frac{\ga_-}{2}<\gamma<\frac{\ga_+}{2}$ so that $\gamma(\gamma-1)+\xk <0$. There exist $\gd_0,\ge_0>0$ such that
$$(N-1)\gamma H_d  d+\gamma(\gamma-1)+\xk <-\ge_0<-1
$$
provided $d(x)\leq\gd_0$. We set
$$A=\left\{x\in\Gw\cap B_R:d(x)\leq \myfrac{\ge_0(R^2-r^2)}{16\gb R}\right\}\quad\text{and }\,B:=A\cap\left\{x\in\Gw\cap B_R:d(x)\leq \gd_0^{\phantom{^4}}\right\}
$$
Then, if $x\in B$, there holds
$$\BA {ll}-(R^2-r^2)^{2}\left((N-1)\gamma H_d  d+\gamma(\gamma-1)+\xk \right)-2\gb d^2\left(NR^2+(2\gb+2-N)r^2\right)\\[4mm]\phantom{-------------}+4\gb\gamma d(R^2-r^2)x\nabla d\geq \myfrac{(R^2-r^2)^{2}\ge_0}{2}
\EA$$
Finally, assume $x\in A^c\cap \left\{x\in\Gw\cap B_R:d(x)\leq \gd_0^{\phantom{^4}}\right\}$ and thus
$$d\geq c_1\myfrac{R^2-r^2}{ R}$$
In order to have
\begin{equation}\label{F4}\BA {llllll}
(i)\; &\Gl^{q-1}(R^2-r^2)^{2-(q-1)\gb}d^{(q-1)\gg+2}\geq d^2R^2&\\[2mm]
(ii)\qquad &\Gl^{q-1}(R^2-r^2)^{2-(q-1)\gb}d^{(q-1)\gg+2}\geq dR(R^2-r^2)&
\EA\end{equation}
or equivalently
\begin{equation}\label{F5}\BA {llllll}
(i)\Longleftrightarrow \Gl^{\frac{1}{\gg}}d\geq (R^2-r^2)^{\frac{\gb}{\gg}}\\[2mm]
(ii)\Longleftrightarrow
\Gl^{\frac{q-1}{(q-1)\gg+1}}d\geq R^{\frac{1}{(q-1)\gg+1}}(R^2-r^2)^{\frac{(q-1)\gb-1}{(q-1)\gg+1}}
\EA\end{equation}
it is sufficient to have, for (i)
\begin{equation}\label{F6}\BA{ll}
c_1\Gl^{\frac{1}{\gg}}\myfrac{R^2-r^2}{ R}\geq (R^2-r^2)^{\frac{\gb}{\gg}}\quad\forall r\in (0,R)\Longleftrightarrow
\Gl\geq c_2R^{2\gb-\gg}
\EA\end{equation}
and for (ii)
\begin{equation}\label{F7}\BA {lll}
c_1\Gl^{\frac{q-1}{(q-1)\gg+1}}\myfrac{R^2-r^2}{ R}\geq
R^{\frac{1}{(q-1)\gg+1}}(R^2-r^2)^{\frac{(q-1)\gb-1}{(q-1)\gg+1}}\quad\forall r\in (0,R)\\
\phantom{\myfrac{1}{ R}}\Longleftrightarrow \phantom{\myfrac{1}{ R}}
\Gl\geq c_2R^{2\gb-\gg-\frac{2}{q-1}}
\EA\end{equation}
where $c_2=c_2(N,\gamma,\beta)>0$ since $\gb>\gg+\frac{2}{q-1}$.

\noindent At end, in the set $C:=\{x\in\Gw:d(x)\geq\gd_0\}$, it suffices that
\begin{equation}\label{F8}\BA {lll}
\Gl\geq c_3\max\left\{R^{2\gb},R^{2\gb-\frac{1}{q-1}}\right\}
\EA\end{equation}
for some $c_3=c_3(N,\gamma,\beta,\max|H_d|,\gd_0)>0$ in order to insure
\begin{equation}\label{F9}\BA {llllll}
(i)&\Gl^{q-1}(R^2-r^2)^{-(q-1)\gb+2}d^{(q-1)\gamma+2}\geq (R^2-r^2)^{2}(N-1)\gamma |H_d|  d
\\[2mm]
(ii)&\Gl^{q-1}(R^2-r^2)^{-(q-1)\gb+2}d^{(q-1)\gamma+2}\geq 4d^2\gb(\gb+1)NR^2
\\[2mm]
(iii) &\Gl^{q-1}(R^2-r^2)^{-(q-1)\gb+2}d^{(q-1)\gamma+2}\geq4\gb dR(R^2-r^2).
\EA\end{equation}

\noindent Noticing that $2\gb>2\gb-\frac{1}{q-1},2\gb-\gg>2\gb-\gg-\frac{1}{q-1}$,
we conclude that there exists a constant $c_4=c_4(N,\gamma,\beta,\max|H_d|,\gd_0)>0$ such that if
\begin{equation}\label{F10}\BA {lll}
\Gl\geq c_4
\max\left\{R^{2\gb},R^{2\gb-\gg-\frac{1}{q-1}}\right\}
\EA\end{equation}
there holds
\begin{equation}\label{F11}\BA {lll}
\CL_{\xk }(f)+f^q\geq 0\qquad\text{in }\;\Gw.
\EA\end{equation}
\noindent {\it Step 2: $\xk =\frac{1}{4}$. } Set $f(x)=\Gl(R^2-r^2)^{-\gb}\sqrt d(\ln \frac{eR}{d})^\frac{1}{2}$ for some $\Gl,\gb$ to be fixed.
Then
$$\BA{ll}\Gd\sqrt d(\ln \frac{eR}{d})^\frac{1}{2}=\frac{1}{\sqrt d}\left(\frac{1}{2}(\ln \frac{eR}{d})^\frac{1}{2}-\frac{1}{2}(\ln \frac{eR}{d})^{-\frac{1}{2}}\right)\Gd d\\[4mm]
\phantom{\Gd\sqrt d(\ln \frac{eR}{d})^\frac{1}{2}=}+
\frac{1}{ d^{\frac{3}{2}}}\left(-\frac{1}{4}(\ln \frac{eR}{d})^\frac{1}{2}-\frac{1}{4}(\ln \frac{eR}{d})^{-\frac{3}{2}}\right)\\[4mm]
\phantom{\Gd\sqrt d(\ln \frac{eR}{d})^\frac{1}{2}}
=\frac{N-1}{\sqrt d}\left(\frac{1}{2}(\ln \frac{eR}{d})^\frac{1}{2}-\frac{1}{2}(\ln \frac{eR}{d})^{-\frac{1}{2}}\right)H_d\\[4mm]
\phantom{\Gd\sqrt d(\ln \frac{eR}{d})^\frac{1}{2}=}+
\frac{1}{ d^{\frac{3}{2}}}\left(-\frac{1}{4}(\ln \frac{eR}{d})^\frac{1}{2}-\frac{1}{4}(\ln \frac{eR}{d})^{-\frac{3}{2}}\right)
\EA$$
Thus
$$\BA{ll}\Gd\sqrt d(\ln \frac{eR}{d})^\frac{1}{2}+\frac{\xk }{d^2}\sqrt d(\ln \frac{eR}{d})^\frac{1}{2}=
\frac{N-1}{\sqrt d}\left(\frac{1}{2}(\ln \frac{eR}{d})^\frac{1}{2}-\frac{1}{2}(\ln \frac{eR}{d})^{-\frac{1}{2}}\right)H_d
-\frac{1}{4 d^{\frac{3}{2}}}(\ln \frac{eR}{d})^{-\frac{3}{2}}\\[2mm]\phantom{(\ln \frac{eR}{d})^\frac{1}{2}+\frac{\xk }{d^2}\sqrt d(\ln \frac{eR}{d})^\frac{1}{2}}
=\frac{1}{ d^{\frac{3}{2}}}(\ln \frac{eR}{d})^{-\frac{3}{2}}\left[(N-1)dH_d\left(\frac{1}{2}(\ln \frac{eR}{d})^2-\frac{1}{2}(\ln \frac{eR}{d})\right)-\frac{1}{4})\right]
\EA$$
Further
$$\BA{ll}
\nabla (R^2-r^2)^{-\gb}\nabla \sqrt d(\ln \frac{eR}{d})^\frac{1}{2}=\frac{\gb (R^2-r^2)^{-\gb-1}(\ln \frac{eR}{d})^{-\frac{1}{2}}}{\sqrt d}\left((\ln \frac{eR}{d})-1\right)x\nabla d.
\EA$$
Therefore
$$\BA{ll}
\Gl^{-1}\CL_{\xk }f=-(R^2-r^2)^{-\gb-2}d^{-\frac{3}{2}}(\ln \frac{eR}{d})^{-\frac{3}{2}}\\[2mm]\phantom{--}\left[\phantom{\myfrac{1}{1}}\!\!\!\!(R^2-r^2)^{2}\left[(N-1)dH_d\left(\frac{1}{2}(\ln \frac{eR}{d})^2-\frac{1}{2}(\ln \frac{eR}{d})\right)-\frac{1}{4}\right]\right.\\[2mm]
\left.
+2\gb(R^2-r^2)d\left[(\ln \frac{eR}{d})^2-(\ln \frac{eR}{d})\right]x\nabla d+2\gb d^2(\ln\frac{eR}{d})^2\left[NR^2+(2\gb+2-N)r^2\right]
\phantom{\myfrac{1}{1}}\!\!\!\!\right]\EA$$
Finally
\begin{equation}\label{F12}\BA {lllll}
\CL_{\xk }f+f^q=\Gl(R^2-r^2)^{-\gb-2}d^{-\frac{3}{2}}(\ln \frac{eR}{d})^{-\frac{3}{2}}
\left[\phantom{\myfrac{1}{1}}\!\!\!\!\Gl^{q-1}(R^2-r^2)^{(1-q)\gb+2}d^{\frac{q+3}{2}}(\ln \frac{eR}{d})^{\frac{1}{2}(q-1)+2}\right.\\[4mm]\phantom{--}\left.
-(R^2-r^2)^{2}\left[(N-1)dH_d\left(\frac{1}{2}(\ln \frac{eR}{d})^2-\frac{1}{2}(\ln \frac{eR}{d})\right)-\frac{1}{4}\right]\right.\\[4mm]
\left.
-2\gb(R^2-r^2)d\left[(\ln \frac{eR}{d})^2-(\ln \frac{eR}{d})\right]x\nabla d-2\gb d^2(\ln\frac{eR}{d})^2\left[NR^2+(2\gb+2-N)r^2\right]
\phantom{\myfrac{1}{1}}\!\!\!\!\right].
\EA\end{equation}
Notice that $\frac{eR}{d}\geq e$ thus $-\frac{1}{2}\leq (\ln \frac{eR}{d})^2-(\ln \frac{eR}{d})\leq (\ln \frac{eR}{d})^2$
If $\gb$ is large enough, as in Step 1, there holds
$$\BA {ll}
\left|2\gb(R^2-r^2)d\left[(\ln \frac{eR}{d})^2-(\ln \frac{eR}{d})\right]x.\nabla d+2\gb d^2(\ln\frac{eR}{d})^2\left[NR^2+(2\gb+2-N)r^2\right]\right|\\[2mm]
\phantom{-------------}\leq 4N\gb(\gb+1)(\ln\frac Rd)^2\left((R^2-r^2)dR+d^2R^2\right).
\EA$$
There exists $\gd_0>0$ such that
$$\BA {ll}(N-1)dH_d\left(\frac{1}{2}(\ln \frac{eR}{d})^2-\frac{1}{2}(\ln \frac{eR}{d})\right)-\frac{1}{4}\leq-\frac{1}{8}<-1
\EA$$
if $d(x)\leq \gd_0$. If we define $A, B$ by
$$A=\left\{x\in\Gw\cap B_R:d(x)\leq \myfrac{\ge_0(R^2-r^2)}{16\gb R(\ln\frac{eR}{d})^2}\right\}\quad\text{and }\,B:=A\cap\left\{x\in\Gw\cap B_R:d(x)\leq \gd_0^{\phantom{^4}}\right\}
$$
there holds if $x\in B$
$$\BA {ll}
-2\gb(R^2-r^2)d\left[(\ln \frac{eR}{d})^2-(\ln \frac{eR}{d})\right]x.\nabla d-2\gb d^2(\ln\frac{eR}{d})^2\left[NR^2+(2\gb+2-N)r^2\right]\\[4mm]\phantom{---}
-(R^2-r^2)^{2}\left[(N-1)dH_d\left(\frac{1}{2}(\ln \frac{eR}{d})^2-\frac{1}{2}(\ln \frac{eR}{d})\right)-\frac{1}{4}\right]
\geq \frac{(R^2-r^2)^2}{16}.
\EA$$
If $x\in A^c\cap\left\{x\in \Gw\cap \Gw:d(x)\leq \gd_0\right\}$, then
\begin{equation}\label{F13}d(x)\geq c_1\myfrac{R^2-r^2}{ R(\ln\frac{eR}{d})^2}.
\end{equation}
In order to have
\begin{equation}\label{F14}\BA {lllll}
(i)\quad \Gl^{q-1}(R^2-r^2)^{(1-q)\gb+2}d^{\frac{q+3}{2}}(\ln \frac{eR}{d})^{\frac{q+3}{2}2}\geq (\ln\frac{eR}{d})^2(R^2-r^2)dR\\[2mm]
(ii)\quad \Gl^{q-1}(R^2-r^2)^{(1-q)\gb+2}d^{\frac{q+3}{2}}(\ln \frac{eR}{d})^{\frac{q+3}{2}}\geq (\ln\frac{eR}{d})^2d^2R^2
\EA\end{equation}
or equivalently
\begin{equation}\label{F15}\BA {lllll}
(i)\quad \Gl^{\frac{2q-2}{q+1}}d(\ln \frac{eR}{d})^{\frac{q-1}{q+1}}\geq (R^2-r^2)^{\frac{2(q-1)\gb-2}{q+1}}R^{\frac{2}{q+1}}\\[2mm]
(ii)\quad \Gl^{2}d\ln \frac{eR}{d}\geq R^{\frac{4}{q-1}}(R^2-r^2)^{2\gb-\frac{4}{q-1}}
\EA\end{equation}
Up to taking $c_1$ small enough, $(\ref{F13})$ is fulfilled if
\begin{equation}\label{F16}\BA{ll}
\myfrac{eR}{d}\leq \myfrac{R^2}{R^2-r^2}\left(\ln(\frac{R^2}{R^2-r^2})\right)^2\Longleftrightarrow
d\geq\myfrac{e(R^2-r^2)}{R}\left(\ln(\frac{R^2}{R^2-r^2})\right)^{-2}.
\EA\end{equation}
Inequality (\ref{F14})-(i) will be insured if
$$\BA {lllll}\Gl^{\frac{2q-2}{q+1}}\geq \frac{1}{e}(R^2-r^2)^{2\frac{(q-1)\gb-1}{q+1}-1}R^{\frac{2}{q+1}+1}(\ln(\frac{R^2}{R^2-r^2})^{\frac{2}{q+1}}
\EA$$
which holds if, for any $\ge>0$, we have for any $r\in (0,R)$
$$
\Gl^{\frac{2q-2}{q+1}}\geq {C_\ge}(R^2-r^2)^{2\frac{(q-1)\gb-1}{q+1}-1}R^{\frac{2}{q+1}+1}\left(\frac{R^2}{R^2-r^2}\right)^{\ge}.
$$
A sufficient condition for such a task is, with the help of $(\ref{F16})$,
\begin{equation}\label{F17}\BA{ll}
\Gl\geq c_3R^{3\gb-\frac{2}{q-1}}.
\EA\end{equation}
As for (\ref{F14})-(ii), it will be insured if
\begin{equation}\label{F17+}\BA{ll}\Gl\geq c_4R^{2\gb-\frac{2}{q-1}-\frac{1}{2}}
\EA\end{equation}
Thus, if
\begin{equation}\label{F18}\BA{ll}\Gl\geq c_5\max\{R^{2\gb-\frac{2}{q-1}-\frac{1}{2}},R^{3\gb-\frac{2}{q-1}}\}
\EA\end{equation}
for some $c_5>0=c_5(N,\gamma,\beta,\gd_0,|H_d|)$, the function $f$ satisfies $(\ref{F11})$.
\end{proof}

\subsection{A priori estimates}
By the Keller-Osserman estimate, it is clear that any solution $u$ of \ref{Eq1} in $\xO$
 satisfies
\be
u(x)\leq C(q,\xO,N)d^{-\frac{2}{q-1}}(x),\quad\forall x\in \xO.\label{3.4.1}
\ee
This estimate is also a consequence of the following result \cite[Prop 3.4]{1}
\begin{prop}\label{prop BR} Let $\phi_*$ be the first positive eigenfunction of $-\Gd$ in $H^1_0(\Gw)$. For $q>1$, there exists $\gamma>0$ and $\ge_0>0$ such that for any
$0\leq\ge\leq\ge_0$ the function $h_+\ge=\gamma(\phi_*-\ge)^{-\frac{2}{q-1}}$ is a supersolution of \ref{Eq1} in
$\xO_{\ge,\phi_*}:=\{x\in\Gw:\phi_*(x)>\ge\}$.
\end{prop}

We recall here that
$$W(x)=\Bigg\{\begin{array}{lll}&d^{\frac{\xa_-}{2}}(x)\qquad&\text{if}\;\xk <\frac{1}{4}\;\\\\
&d^{\frac{1}{2}}(x)|\log d(x)|\qquad&\text{if}\;\xk =\frac{1}{4}
\end{array}
$$
\begin{prop}\label{prop19}
Let $\xO$ be a bounded open domain uniformly of class
$C^2$ and let $F$ be a compact subset of the boundary. Let $u$
be a nonnegative solution of \ref{T1} in $\xO$ such that
$$\lim_{x\in\xO,\;x\rightarrow\xi}\frac{u(x)}{W(x)}=0,\quad\forall \xi\in\partial\xO\setminus F,$$
locally uniformly in $\partial\xO\setminus F$. Then there exists a constant $C$
depending only on $q, \xk $ and $\xO$
 such that,
 \be
|u(x)|\leq Cd^{\frac{\xa_+}{2}}(x)\left(\mathrm{dist}(x,F)\right)^{-\frac{2}{q-1}-\frac{\xa_+}{2}},\quad\forall x\in \xO,\label{3.4.24}
  \ee
   \be
|\frac{u(x)}{d^{\frac{\xa_+}{2}}(x)}-\frac{u(y)}{d^{\frac{\xa_+}{2}}(y)}|\leq C|x-y|^\gb\left(\mathrm{dist}(x,F)\right)^{-\frac{2}{q-1}-\gb-\frac{\xa_+}{2}}\quad\forall (x,y)\in \xO\times \xO\label{3.4.24+}
  \ee
  such that $\mathrm{dist}(x,F)\leq \mathrm{dist}(y,F),$
  \be
  |\nabla u(x)|\leq Cd^{\frac{\xa_+}{2}-1}(x)\left(\mathrm{dist}(x,F)\right)^{-\frac{2}{q-1}-\frac{\xa_+}{2}},\quad\forall x\in \xO.\label{3.4.24*}
  \ee
\end{prop}  
\begin{proof}
The proof is based on the proof of Proposition 3.4.3 in \cite{book}. Let $\xi\in\partial\xO\setminus\ F$
and put $d_F (\xi) = \frac{1}{2} \mathrm{dist} (\xi, F).$  Denote by
$\xO^\xi$ the domain
$$\xO^\xi=\{y\in\mathbb{R}^n:\;d_F(\xi)y\in\xO\}.$$
If $u$ is a positive solution of (\ref{T1}) in $\xO,$ denote by $u^\xi$ the function
$$u^\xi (y) = |d_F (\xi)|^{\frac{2}{q-1}}
u(d_F(\xi)y),\;\forall y\in\xO^\xi.
$$
Then,
$$-\xD u^\xi-\xk \frac{u}{|\mathrm{dist}(y,\partial\xO^\xi)|^2}+\left|u^\xi\right|^q=0\qquad\mathrm{in}\;\;\xO^\xi.$$
Let $R_0$ be the constant in Proposition \ref{barr}. First, we assume that
$$\mathrm{dist} (\xi, F)\leq\frac{1}{1+R_0}.$$
Set $r_0=\frac{3R_0}{4},$ then the solution $W_{r_0,\xi}$ mentioned
in Proposition \ref{barr} satisfies
$$
u^\xi(y) \leq W_{r_0,\xi}(y),\quad \forall y\in B_{\frac{3R_0}{4}}(\xi)\cap\xO^\xi .$$
Thus $u^\xi$ is bounded in $B_{\frac{3R_0}{5}}(\xi)\cap\xO^\xi$ by a constant $C>0$ depending only on $n, q, \xk $
and the $C^2$ characteristic of $\xO^\xi .$ As $d_F (\xi)\leq 1$ a $C^2$ characteristic of $\xO$
is also a $C^2$ characteristic of $\xO^\xi$ therefore the constant $C$ can be taken to be independent of $\xi.$ We note here that the constant $0<R_0<1$ depends on $C^2$ characteristic of $\xO.$

Now we note that
$$\lim_{y\in\xO^\xi,\;y\rightarrow P}\frac{u^\xi(y)}{W(x)}=0,\quad\forall P\in\partial\xO^\xi\cap B_{\frac{3R_0}{5}}(\xi).$$
Thus in view of the proof of Lemmas \ref{maincomp1} and \ref{maincomp}, by the above inequality and in view of the proof of Theorem 2.12 in \cite{F.M.T2}, we have that there exists  $C>0$ depending only on $n, p, \xk $ such that
\be
u^\xi(y) \leq\left|\mathrm{dist}(y,\partial\xO^\xi)\right|^\frac{\xa_+}{2},\quad \forall y\in B_{\frac{R_0}{2}}(\xi)\cap\xO^\xi .\label{sim}\ee
$$
\frac{u^\xi(y)}{\left|\mathrm{dist}(y,\partial\xO^\xi)\right|^\frac{\xa_+}{2}}\leq C\frac{u^\xi(x)}{\left|\mathrm{dist}(x,\partial\xO^\xi)\right|^\frac{\xa_+}{2}},\quad\forall x,y\in  B_{\frac{R_0}{2}}(\xi)\cap\xO^\xi
$$

Hence
$$
u(x)\leq d^\frac{\xa_+}{2}(x)d_F(\xi)^{-\frac{2}{q-1}-\frac{\xa_+}{2}},\quad\forall x\in B_{d_F(\xi)\frac{R_0}{2}}(\xi)\cap\xO.
$$
\be
 \frac{u(y)}{d^\frac{\xa_+}{2}}(y)\leq C\frac{u^\xi(x)}{d^\frac{\xa_+}{2}(x)},\quad\forall x,y\in  B_{d_F(\xi)\frac{R_0}{2}}(\xi)\cap\xO.\label{cor6}
\ee
Let $x\in \xO_{\frac{R_0}{2}}$ and assume that
$$d(x)\leq \frac{R_0}{2}d_F(x).$$
Let $\xi$ be the unique point in $\partial\xO\setminus F$ such that $|x-\xi|=d(x).$ Then we have
$$d_F(\xi)\leq d(x)+d_F(x)\leq(1+R_0)d_F(x)<1$$
and
$$|u(x)|\leq Cd^{\frac{\xa_+}{2}}(x)\left((1+R_0)\mathrm{dist}(x,F)\right)^{-\frac{2}{q-1}-\frac{\xa_+}{2}}.$$
If   $d(x)> \frac{R_0}{4}d_F(x),$ then by \eqref{3.4.1} we have that
$$|u(x)|\leq Cd^{-\frac{2}{q-1}}(x)\leq Cd^{\frac{\xa_+}{2}}(x)\left(\frac{R_0}{2}\mathrm{dist}(x,F)\right)^{-\frac{2}{q-1}-\frac{\xa_+}{2}}.$$
Thus (\ref{3.4.24}) holds for every $x\in\xO_{\frac{R_0}{2}}$ such that $\mathrm{dist} (x, F) < \frac{1}{1 + R_0}.$

Now we assume that $x\in \xO_{\frac{R_0}{2}}$ and $$\mathrm{dist} (x, F)\geq \frac{1}{1 + R_0}.$$
Let $\xi$ be the unique point in $\partial\xO\setminus F$ such that $|x-\xi|=d(x).$ Similarly with the proof of \ref{sim} we can prove that
$$
u(x)\leq C d^\frac{\xa_+}{2}(x)\leq  d^\frac{\xa_+}{2}(x)C\left((1+R_0)\mathrm{dist}(x,F)\right)^{-\frac{2}{q-1}-\frac{\xa_+}{2}},\quad\forall x\in B_{\frac{R_0}{2}}(\xi)\cap\xO.
$$
Now if $x\in \xO\setminus\xO_{\frac{R_0}{2}},$ the proof of \eqref{3.4.24} follows by \eqref{3.4.1}.

(ii) Let $x_0\in\xO.$ Set
$$\xO^{x_0}=\{y\in\mathbb{R}^n:\;d(x_0)y\in\xO\},$$
and $d_{x_0}(y)=\mathrm{dist}(y,\partial\xO^{x_0}).$
If $x\in B_{\frac{d(x_0)}{2}}(x_0)$ then $y=\frac{x}{d(x_0)}$ belongs to $B_\frac{1}{2}(y_0),$ where $y_0=\frac{x_0}{d(x_0)}.$ Also we have that
$\frac{1}{2}\leq d_{x_0}(y)\leq \frac{3}{2}$ for each $y\in B_\frac{1}{2}(y_0).$ Set now $v(y)=u(d(x_0)y),\;\forall y\in B_\frac{1}{2}(y_0).$ Then $v$ satisfies
$$-\xD v-\xk \frac{u}{|d_{x_0}(y)|^2}+d^2(x_0)\left|v\right|^q=0\qquad\mathrm{in}\;\;B_\frac{1}{2}(y_0).$$
By standard elliptic estimate we have
$$
\sup_{y\in B_\frac{1}{4}(y_0)}|\nabla v|\leq C\left(\sup_{y\in B_\frac{1}{3}(y_0)}|v|+\sup_{y\in B_\frac{1}{3}(y_0)}d^2(x_0)|v|^q\right),
$$
Now since $\nabla v(y)=d(x_0)\nabla u(d(x_0)y),$ by above inequality and \eqref{3.4.24} we have that
$$|\nabla u(x_0)|\leq C \left(d^{\frac{\xa_+}{2}-1
}(x_0)\left(\mathrm{dist}(x_0,F)\right)^{-\frac{2}{q-1}-\frac{\xa_+}{2}}+ d^{\frac{q\xa_+}{2}+1}(x_0)\left(\mathrm{dist}(x_0,F)\right)^{-q\left(\frac{2}{q-1}-\frac{\xa_+}{2}\right)}\right).$$
Using $\frac{2q}{q-1}=\frac{2}{q-1}+2$ and the fact that $x_0$ is arbitrary the result follows.
\end{proof}

\begin{prop} \label{U-F}Let $O\subset\prt\Gw$ be a relatively open subset and $F=\overline O$. Let $U_F$ be defined by (\ref{T7}) be the maximal solution of (\ref{T1}) which vanishes on $\prt\Gw\setminus F$. Then for any compact set $K\subset O$, there holds
\begin{equation}\label {A-I-1}
\lim_{\xi\to x}(d(\xi))^{\frac{2}{q-1}}U_F(\xi)=\ell_{\gk}=\left(\frac{2(q+1)}{(q-1)^2}+\gk\right)^{\frac{1}{q-1}}\quad\text{uniformly with respect to } x\in K.
\end{equation}
\end{prop}
\begin{proof} {\it Step 1}. We claim that for any $\ge>0$ there exists $C_\ge,\gt_\ge>0$ such that for any $z\in O$ such that
$\overline B_{2\gt_\ge}(z)\subset O$, there holds
\begin{equation}\BA {lll}\label {A-I-2}
u(x)\leq (\ge+\ell_\gk^{q-1})^{\frac{1}{q-1}}\gt^{-\frac{2}{q-1}}+C_\ge\qquad \forall \gt\in (0,\gt_\ge],\,\forall x\in \Gs_{\gt}(\overline {B_{\gt_\ge}}(z)).
\EA\end{equation}
We recall that $\Gs_{\gt}(\overline {B_{\gt_\ge}}(z))=\left\{x\in\Gw,\,x\approx (d(x),\gs(x)),d(x)=\gt, \gs(x)\in \overline B_{\gt_\ge}(z)\right\}$.
Set $g(x)=\ell d^{-\frac{2}{q-1}}(x)$, then
\begin{equation}\BA {lll}\label {A-I-3*}\CL_\gk g+g^q=\myfrac{2(N-1)}{q-1}H_dd^{-\frac{q+1}{q-1}}+\left(\ell^{q-1}-\ell^{q-1}_\gk
\right)d^{-\myfrac{2q}{q-1}},
\EA\end{equation}
where $H_d$ is the mean curvature of $\Gs_d$. If $\Gw$ is convex we take $\ell=\ell_\gk$ and $g$ is a supersolution for $d(x)\leq R_0$ for some $R_0$. In the general case, we take $\ell=\ell(\ge)=(\ge+\ell^{q-1}_\gk)^{\frac{1}{q-1}}$, and $g=g_\ge=\ell(\ge)d^{-\frac{2}{q-1}}$ is a supersolution in the set $\Gw_{\gt_\ge}$ where
$$\gt_\ge=\max\left\{\gt:0< \gt\leq \frac{R_0}{2},\frac{2(N-1)}{q-1}\|H_\gt\|_{L^\infty(\Gs_\gt)}+\ge>0\right\}.
$$
Then $f_{2\gt_\ge,z}+g_\ge$ is a supersolution of (\ref{T1}) in $B_{2\gt_\ge}(z)\cap\Gw$ which tends to infinity on
$\prt (B_{2\gt_\ge}(z)\cap\Gw)=\prt\Gw\cap B_{2\gt_\ge}(z)\cup \Gw\cap\prt B_{2\gt_\ge}(z)$. Since we can replace $g_\ge(x)$ by
$g_{\ge,\gt}(x)=\ell (d(x)-\gt)^{-\frac{2}{q-1}}$ for $\gt\in (0,\gr_\ge)$, any positive solution $u$ of (\ref{T1}) in $\Gw$ is bounded from above by $f_{2\gt_\ge,z}+g_{\ge,\gt}$ and therefore by $f_{2\gt_\ge,z}+g_{\ge}$. This implies (\ref{A-I-2}) with
$C_\ge=\max\{f_{2\gt_\ge,z}(y):|y-z|\leq \gt_\ge\}$, and it can be made explicit thanks to (\ref{BAR1}).\smallskip

\noindent{\it Step 2}. With the same constants as in step 1, we claim that
\begin{equation}\BA {lll}\label {A-I-3}
U_F(x)\geq (\ell_\gk^{q-1}-\ge)^{\frac{1}{q-1}}\gt^{-\frac{2}{q-1}}-C_\ge\qquad \forall \gt\in (0,\gt_\ge],\,\forall x\in \Gs_{\gt}(\overline {B_{\gt_\ge}}(z)).
\EA\end{equation}
If in the definition of the function $g$, we take $\ell=\ell(\ge)=(\ell^{q-1}_\gk-\ge)^{\frac{1}{q-1}}$, then $g$ is a subsolution in the same set $\Gw_{\gt_\ge}$. Since $U_F+f_{2\gt_\ge,z}$ is a supersolution of (\ref{T1}) in $B_{2\gt_\ge}(z)\cap\Gw$ which tends to infinity on the boundary, it dominates the subsolution $g_{\ge,-\gt}=\ell (d(.)+\gt)^{-\frac{2}{q-1}}$ for $\gt\in (0,\gr_\ge)$
and thus , as $\gt\to 0$, $g_{\ge}(x)\leq U_F(x)+f_{2\gt_\ge,z}(x)$. This implies (\ref{A-I-3}) with the same constant $C_\ge$.
\smallskip

\noindent{\it Step 3}. End of the proof. Since $K\subset O$ is precompact, for any $\ge>0$, there exists a finite number of points $z_j$, $j=1,...,k$ such that $K\subset\cup_{j=1}^{k}\overline {B_{\gt_\ge}}(z_j)$ with $\overline {B_{2\gt_\ge}}(z_j)\subset O$. Therefore
\begin{equation}\BA {lll}\label {A-I-4}
(\ell_\gk^{q-1}-\ge)^{\frac{1}{q-1}}\gt^{-\frac{2}{q-1}}-C_\ge\leq U_F(x)\leq (\ge+\ell_\gk^{q-1})^{\frac{1}{q-1}}\gt^{-\frac{2}{q-1}}+C_\ge
\quad \forall \gt\in (0,\gt_\ge],\,\forall x\in \Gs_{\gt}(K).
\EA\end{equation}
Since $\ge$ is arbitrary, it yields to
\begin{equation}\BA {lll}\label {A-I-5}
\lim_{\gt\to 0}\|\gt^{\frac{2}{q-1}}U_F-\ell_\gk\|_{L^{\infty}(\Gs_\gt(K))}=0
\EA\end{equation}
which is (\ref{A-I-1}).
\end{proof}

\begin{coro} \label{U-Max}
Let $U_{\prt\Gw}$ be the maximal solution of (\ref{T1}) in $\Gw$, then
\begin{equation}\label {A-I-1*}
\lim_{d(x)\to 0}(d(x))^{\frac{2}{q-1}}U_{\prt\Gw}(x)=\ell_{\gk}.
\end{equation}
\end{coro}

\noindent\emph{Acknowledgment }
The first author was supported by Fondecyt Grant 3140567.



\end{document}